\documentclass[report, defended]{cit_thesis}

\usepackage{amssymb,amsfonts}
\usepackage{amsthm}
\usepackage{changepage}
\usepackage{amsmath}
\usepackage{verbatim}
\usepackage[all,arc]{xy}
\usepackage{enumerate}
\usepackage{hhline}
\usepackage{ulem}
\usepackage{mathrsfs}
\usepackage{graphicx}
\usepackage{url}
\usepackage{mathdots}
\usepackage[latin1]{inputenc}
 \usepackage{colortbl}
\usepackage{rotating}
\usepackage{amssymb}

\newcommand{\Nearrow}{%
        \begin{turn}{45}
                \raisebox{-1ex}{$\Rightarrow$}
        \end{turn}
}

 \definecolor{umbra}{rgb}{0.7,0.8,0.9}
 \definecolor{grn}{rgb}{.7,1,0.7}
 \definecolor{purp}{rgb}{0.8,0.5,0.8}

\def\colCell#1#2{\multicolumn{1}{>{\columncolor{#1}}c}{#2}}

\newtheorem{thm}{Theorem}[section]
\newtheorem{cor}[thm]{Corollary}
\newtheorem{prop}[thm]{Proposition}
\newtheorem{lem}[thm]{Lemma}

\newtheorem{conj}[thm]{Conjecture}

\newenvironment{referencethm}[2][Theorem]{\begin{trivlist}
\item[\hskip \labelsep {\bfseries #1}\hskip \labelsep {\bfseries #2}] \itshape}{\end{trivlist}}

\newenvironment{referencecor}[2][Corollary]{\begin{trivlist}
\item[\hskip \labelsep {\bfseries #1}\hskip \labelsep {\bfseries #2}] \itshape}{\end{trivlist}}

\theoremstyle{definition}
\newtheorem{defn}[thm]{Definition}

\theoremstyle{remark}

\title{Completing $\epsilon$-Dense Partial Latin Squares}
\author{Padraic Bartlett}
\date{May 28, 2013}

\begin{document}

\maketitle
\begin{dedication}
To Laurie.  \\(In lieu of the flowers I forgot to send for Mother's Day.)
\end{dedication}
\begin{acknowledgements}
I want to open by thanking Richard Wilson, my advisor, for his invaluable assistance throughout the past five years.  Without Rick's guidance, I would have never found the papers that started this dissertation.  Without his assistance, I would have never had the conversations with Peter Dukes and Esther Lamken that led to the concept of improper trades. Without his patience, I would have never had the freedom to explore the ideas behind this dissertation.  He is the best advisor a student could hope to have.

I also want to thank my students at Caltech.  Without the community, friendship, and support that you've so graciously extended to me in my time here, I would have never made it through this dissertation.  Special thanks to everyone who's tolerated my rants on Latin squares, solved some of my random questions before class, or simply been my friend.  I will miss you all.

In the same vein, I should credit the Canada/USA Mathcamp students and staff I've had the good fortune to know over the past three years.  I will never quite understand by what luck I was able to spend three summers of my Ph.D.\ running classes and pursuing research with the best high schoolers in the country, but I am thankful regardless.  

Specific thanks go to Andre Arslan, Alan Talmage and Sachi Hashimoto for developing part of Lemma $\ref{lem1}$ and some of the initial ideas in Chapter 2, Vishank Jain-Sharma and Shakthi Shrima for a number of conversations that led to the ideas explored in Chapter 3, and Sarah Shader for a number of excellent talks on generating Latin squares with certain types of substructure. Alex Pei, Lin Xu, Rahul Sridhar, Billy Swartworth, David Lu, Kevin Lu,  and Yash Faroqui were also quite useful in developing and checking portions of Chapter $\ref{qrchap}$.  As well, Susan Durst was invaluable for structuring arguments in the dissertation and maintaining the author's sanity at various junctures throughout the year.

Finally, those who take care of me:\ Seamus, Laurie, Dale, and Gemma.  I would have fallen apart years ago without your love.

I'm certain there are people I'm missing here; enumerating all of you would rival the length of the dissertation itself.  I owe so much to so many of you.  I hope to pay it forward.

\end{acknowledgements}
\begin{abstract}

A classical question in combinatorics is the following:\ given a partial Latin square $P$, when can we complete $P$ to a Latin square $L$?   In this paper, we investigate the class of \textbf{$\epsilon$-dense partial Latin squares}:\ partial Latin squares in which each symbol, row, and column contains no more than $\epsilon n$-many nonblank cells.  Based on a conjecture of Nash-Williams, Daykin and H\"aggkvist conjectured that all $\frac{1}{4}$-dense partial Latin squares are completable.  In this paper, we will discuss the proof methods and results used in previous attempts to resolve this conjecture, introduce a novel technique derived from a paper by Jacobson and Matthews on generating random Latin squares, and use this novel technique to study $ \epsilon$-dense partial Latin squares that contain no more than $\delta n^2$ filled cells in total.  

In Chapter 2, we construct completions for all $ \epsilon$-dense partial Latin squares containing no more than $\delta n^2$ filled cells in total, given that $\epsilon < \frac{1}{12}, \delta < \frac{ \left(1-12\epsilon\right)^{2}}{10409}$.  In particular, we show that all $9.8 \cdot 10^{-5}$-dense partial Latin squares are completable.  In Chapter 4, we augment these results by roughly a factor of two using some probabilistic techniques.  These results improve prior work by Gustavsson, which required $\epsilon =  \delta \leq 10^{-7}$, as well as Chetwynd and H\"aggkvist, which required $\epsilon = \delta = 10^{-5}$, $n$ even and greater than $10^7$.

If we omit the probabilistic techniques noted above, we further show that such completions can always be found in polynomial time.  This contrasts a result of Colbourn, which states that completing arbitrary partial Latin squares is an NP-complete task.  In Chapter 3, we strengthen Colbourn's result to the claim that completing an arbitrary $\left(\frac{1}{2} + \epsilon\right)$-dense partial Latin square is NP-complete, for any $\epsilon > 0$.

Colbourn's result hinges heavily on a connection between triangulations of tripartite graphs and Latin squares.  Motivated by this, we use our results on Latin squares to prove that any tripartite graph $G = (V_1, V_2, V_3)$ such that
\begin{itemize}
\item  $|V_1| = |V_2| = |V_3| = n$,
\item For every vertex $v \in V_i$, $\deg_+(v) = \deg_-(v) \geq (1- \epsilon)n,$  and
\item $|E(G)| > (1 - \delta)\cdot 3n^2$
\end{itemize} 
admits a triangulation, if $\epsilon < \frac{1}{132}$, $\delta < \frac{(1 -132\epsilon)^2  }{83272}$.  In particular, this holds when $\epsilon = \delta=1.197 \cdot 10^{-5}$.

This strengthens results of Gustavsson, which requires $\epsilon = \delta = 10^{-7}$.  

In an unrelated vein, Chapter 6 explores the class of \textbf{quasirandom graphs}, a notion first introduced by Chung, Graham and Wilson \cite{chung1989quasi} in 1989.  Roughly speaking, a sequence of graphs is called ``quasirandom'' if it has a number of properties possessed by the random graph, all of which turn out to be equivalent.  In this chapter, we study possible extensions of these results to random $k$-edge colorings, and create an analogue of Chung, Graham and Wilson's result for such colorings.

\end{abstract}

\tableofcontents

\mainmatter

\chapter{Summary of Results}

\section{Summary of Results}
We briefly summarize the results proven in this document.  Readers curious for the definitions, background, motivation, or proofs of these results are advised to read the relevant chapters.

This dissertation is primarily a document exploring the class of $\epsilon$-dense partial Latin squares:\ partial Latin squares in which no row, column, or symbol is used more than $\epsilon n$ times.  In Chapter 2, we explore these classes of partial Latin squares, and prove the following theorem.

\begin{referencethm}{\ref{thm1}}
Any $ \epsilon$-dense partial Latin square $P$ containing no more than $\delta n^2$ filled cells in total is completable, for $\epsilon < \frac{1}{12}, \delta <\frac{(1 -12\epsilon)^2  }{10409}$.
\end{referencethm}

This has several nice corollaries.
\begin{referencecor}{\ref{plscor1}}
Any $\frac{1}{13}$-dense partial Latin square containing no more than $5.7 \cdot 10^{-7}$ filled cells is completable.
\end{referencecor}
\begin{referencecor}{\ref{plscor2}}
All $ 9.8 \cdot 10^{-5}$-dense partial Latin squares are completable.
\end{referencecor}
\begin{referencecor}{\ref{plscor3}}
All $10^{-4}$-dense partial Latin squares are completable, for $n > 1.2 \cdot 10^5.$
\end{referencecor}

In Chapter 4, we apply the probabilistic method to Theorem $\ref{thm1}$ to improve this result by slightly less than a factor of 2.
\begin{referencethm}{\ref{probthm1}}
Any $\epsilon$-dense partial Latin square $P$ is completable, for $\epsilon, \delta, n$ such that
\begin{align*}
12 &\leq n - 12 n\sqrt{36\delta + \frac{198\delta}{n} +  \frac{5346 \epsilon}{n} + \frac{1518}{100 \cdot n}  + 10956 \epsilon^2} -  12\epsilon n.\\
\end{align*}
\end{referencethm}

Again, this has a nice corollary when we set $\epsilon = \delta$.
\begin{referencecor}{\ref{probcor1}}
All $\frac{1}{6000}$-dense partial Latin squares are completable, for $n > \frac{1}{25000}$.
\end{referencecor}

In Chapter 3, we examine the runtime of the algorithm used to prove Theorem $\ref{thm1}$.  In particular, we prove the following claim.
\begin{referencethm}{\ref{polytimethm}}
The algorithm used in Theorem $\ref{thm1}$ needs no more than $O(n^3)$ steps to construct its claimed completion of the targeted partial Latin square.
\end{referencethm}

This contrasts strongly with a result of Colbourn that claims that completing an arbitrary partial Latin square is an NP-complete problem.  We strengthen his result as follows.
\begin{referencethm}{\ref{npcompletepls}}
The task of completing an arbitrary $\epsilon$-dense partial Latin square is NP-complete, for any $\epsilon > \frac{1}{2}$.
\end{referencethm}

In Chapter 5, we apply Theorem $\ref{thm1}$ to studying triangulations of graphs.  In particular, we prove the following.

\begin{referencethm}{\ref{tri1}}
 Let $G$ be a tripartite graph with tripartition $(V_1, V_2, V_3)$, with the following properties.
\begin{itemize}
\item  $|V_1| = |V_2| = |V_3| = n$. 
\item For every vertex $v \in V_i$, $\deg_+(v) = \deg_-(v) \geq (1- \epsilon)n.$ 
\item $|E(G)| > (1 - \delta)\cdot 3n^2$.  
\end{itemize} 
Then, if $\epsilon < \frac{1}{132}$, $\delta < \frac{(1 -132\epsilon)^2  }{83272}$, this graph admits a triangle decomposition.
\end{referencethm}

Finally, in Chapter 6, we change gears somewhat and study \textbf{quasirandom graphs}.  In particular, we formulate the following notion for \textbf{quasirandom $k$-edge-colorings}:\ a sequence $\mathcal{G}$ of $k$-edge colorings such that the following seven properties are all satisfied.
\begin{enumerate}
\item[$P_1(s)$:] 
For any graph $H_s$ on $s$ vertices,
\begin{align*}
N_G^*(H_s) = \left( 1 + o(1) \right) \cdot n^s \cdot k^{-\binom{s}{2}}.
\end{align*}
\item[$P_2(t)$:] 
For any given color $i$, let $C_{t,i}$ denote the cycle of length $t$ where all of the edges have color $i$.  Then
\begin{align*}
e_i(G) &\geq (1+o(1)) \cdot \frac{n^2}{2k}, \textrm{ and}\\
N_G(C_{t,i}) &\leq \left( 1 + o(1) \right) \cdot\frac{n^t}{k^t}.
\end{align*}
\item[$P_3$:] For any given color $i$, let $A(G_i)$ denote the adjacency matrix of $G_i$, and $|\lambda_1| \geq \ldots \geq |\lambda_n|$ be the eigenvalues of $A(G_i)$.  Then
\begin{align*}
e_i(G) &\geq (1+o(1)) \cdot \frac{n^2}{2k}, \textrm{ and}\\
\lambda_1 &= (1+o(1)) \cdot \frac{n}{k}, \quad \lambda_2 = o(n).
\end{align*}
\item[$P_4$:] Given any subset $S \subseteq V$ and any color $i$,
\begin{align*}
e_i(S) = \frac{|S|^2}{2k} + o(n^2). 
\end{align*}
\item[$P_5$:] Given any subset $S \subseteq V$ with $S = \lfloor n/2 \rfloor$, and any color $i$,
\begin{align*}
e_i(S) = \frac{n^2}{8k}+ o(n^2). 
\end{align*}
\item[$P_6$:] Given any pair of vertices $v, v' \in G$, let $s(v,v')$ denote the number of vertices $y$ such that both $(v, y)$ and $(v', y)$ are the same color in our coloring of $G$.  Then
\begin{align*}
\sum_{v, v'} \left| s(v,v')  - \frac{n}{k}\right| = o(n^3).
\end{align*}
\item[$P_7$:] Given any color $i$,
\begin{align*}
\sum_{v, v'} \left| \left| n_i(v) \cap n_i(v')\right| - \frac{n}{k^2} \right| = o(n^3).
\end{align*}
\end{enumerate}
We prove the following theorem.

\begin{referencethm}{\ref{qrthm}}
Suppose that $\mathcal{G}$ is a sequence of $k$-colorings of complete graphs that satisfies any one of the properties
\begin{itemize}
\item $P_1(s)$, for some $s \geq 4$, 
\item $P_2(t)$, for some $t \geq 4$, or 
\item $P_3$, or $P_4$, or $P_5$, or $P_6$, or $P_7$.
\end{itemize}
Then it satisfies all of these properties.
\end{referencethm}

\chapter{$\epsilon$-Dense Partial Latin Squares and Their Completions}

Latin squares are a longstanding object of combinatorial interest.  In the following two sections, we provide the basic definitions and concepts that a reader will need to understand the proofs and concepts communicated by this paper.  We also mention a few elementary examples and applications to hopefully motivate interest.  For readers seeking a more in-depth introduction to Latin squares, Laywine and Mullen's text \cite{Laywine_Mullen_1998} is an excellent source to consult.

Conversely, readers who are familiar with Latin squares and their terminology are advised to skip to Section $\ref{wherethingsstartforreals}$, where the specific proof techniques used in chapter 2 are defined and discussed.

\section{Basic Definitions}

\begin{defn}
A \textbf{Latin square} of order $n$ is an $n \times n$ array filled with $n$ distinct symbols, typically, $\{1, \ldots n\}$, such that no symbol is repeated in any row or column.  For example, the following two grids are a pair of distinct Latin squares of order 3.
\begin{align*}
\begin{array}{|c|c|c|c|} 
\hline 1 & 2 & 3 \\
\hline3 & 1 & 2 \\
\hline 2 & 3 & 1 \\\hline
\end{array}, \qquad
\begin{array}{|c|c|c|c|} 
\hline 1 & 2 & 3 \\
\hline 2 & 3 & 1 \\
\hline 3 & 1 & 2 \\\hline
\end{array}.
\end{align*}
\end{defn}

\begin{defn}
A \textbf{partial Latin square} of order $n$ is simply an order-$n$ Latin square where we also allow cells to be blank.  We say that a partial Latin square $P$ is \textbf{completable}, and that $L$ is a \textbf{completion} of $P$, if the blank cells of $P$ can be filled with symbols in such a way that the resulting array $L$ is a Latin square.  Given a partial Latin square $P$, finding a completion is not always possible, as the following two examples illustrate.
\begin{align*}
\begin{array}{|c|c|c|c|} 
\hline 1 & 2 & ~ \\
\hline~ & 3 & ~ \\
\hline ~ & ~ & 2 \\\hline
\end{array} \mapsto
\begin{array}{|c|c|c|c|} 
\hline 1 & 2 & 3 \\
\hline 2 & 3 & 1 \\
\hline 3 & 1 & 2 \\\hline
\end{array}, \textrm{ while } \begin{array}{|c|c|c|c|} 
\hline 1 & ~ & ~ \\
\hline~ &  1 & ~ \\
\hline ~ & ~ & 2 \\\hline
\end{array}
\textrm{ has no completion.}
\end{align*}
\end{defn}

The task of completing a partial Latin square is one of the few problems in combinatorics to have attained pop-culture status, in the form of Sudoku\footnote{A \textbf{Sudoku} grid is a popular recreational class of puzzles.  Each consists of a $9 \times 9$ partial Latin square, divided into nine $3 \times 3$ subarrays; the goal of the puzzle is to come up with a way of completing the grid so that no symbol is repeated in any row, column, or any of the nine $3 \times 3$ subarrays.} grids.  More seriously, completing arbitrary partial Latin squares is also a problem with applications in industry and computer science.  Consider, for example, the following simplistic model of a \textbf{router}.
\begin{itemize}
\item T Take a box with $n$ fiber-optic cables entering it from computers $c_1, \ldots c_n$, and $n$ fiber-optic cables leaving it corresponding to devices $r_1, \ldots r_n$. Assume that each of these cables can carry and transmit up to $n$ distinct wavelengths $\{s_1, \ldots s_n\}$ simultaneously without interference.
\item Suppose that this box is a \textbf{router}:\ i.e.\ it implements a collection of rules of the form ``allow computer $c_1$ to communicate with device $r_2$ on wavelength $s_3$.''  To avoid conflict, no two devices should communicate on the same wavelength with the same computer, nor should any two computers communicate with the same device on the same wavelength.
\item If we interpret these rules as triples $(r,c,s)$, then the rules above define a partial Latin square.  From this perspective, the task of adding additional traffic to a router is just the task of filling in cells in a partial Latin square without introducing any repeats in a row or column.  Completing a partial Latin square, in this setting, is a way to insure that all possible computers can communicate with all possible devices without conflict.  
\end{itemize}

For the more practically-minded reader, the above offers some justification for why mathematicians study partial Latin squares and their completions.  For those with a more theoretical bent, however, it bears mentioning that determining the classes of partial Latin squares that do admit completions is a class of problems with a rich history of results in combinatorics.  We list a few of the more famous classes of partial Latin squares for which we have resolved this question.
\begin{itemize}
\item Any Latin rectangle (i.e.\ any partial Latin square $P$ where the first $k$ rows of $P$ are completely filled, while the rest are blank) can be completed (Hall \cite{Hall_1945}).
 \item  If $P$ is a partial Latin square all of whose nonblank entries lie within some set of $s$ rows and $t$ columns, and $s+t \leq n$, $P$ can be completed (Ryser \cite{Ryser_1951}).
 \item If $P$ is a partial Latin square with no more than $n-1$ filled cells, $P$ can be completed (Smetianuk \cite{Smetianuk_1981}).
\item  If $P$ is an $n \times n$ partial Latin square with order greater than $5$ such that precisely two rows and two columns of $P$ are filled, $P$ can be completed (Buchanan \cite{Buchanan_2007}).
  \end{itemize}

\section{$\epsilon$-Dense Partial Latin Squares}

In this paper, we will examine the class of \textbf{$\epsilon$-dense partial Latin squares}:\ $n \times n$ partial Latin squares in which each symbol, row, and column contains no more than $\epsilon n$-many nonblank cells.  This class of partial Latin squares was first introduced in a paper by Daykin and H\"aggkvist \cite{Daykin_Haggkvist_1984}, where they made the following conjecture.
\begin{conj}
Any $ \frac{1}{4}$-dense partial Latin square can be completed.
\end{conj}

 Gustavsson \cite{Gustavsson_1991} noted in his thesis (completed under H\"aggkvist) that this bound of $\frac{1}{4}$ was anticipated somewhat by a conjecture of Nash-Williams \cite{Nash_Williams_1970} on triangle decompositions of graphs.  We define the relevant terms here.

\begin{defn}
Given a graph $G$, a \textbf{graph decomposition} $\mathcal{H}$ is a collection $\{H_1, \ldots H_k\}$ of subgraphs of $G$, such that the edges of $G$ are partitioned by these $H_i$'s.  For example, the picture below illustrates a graph being decomposed into 9 edge-disjoint triangles.
\begin{center}
\includegraphics[width=3in]{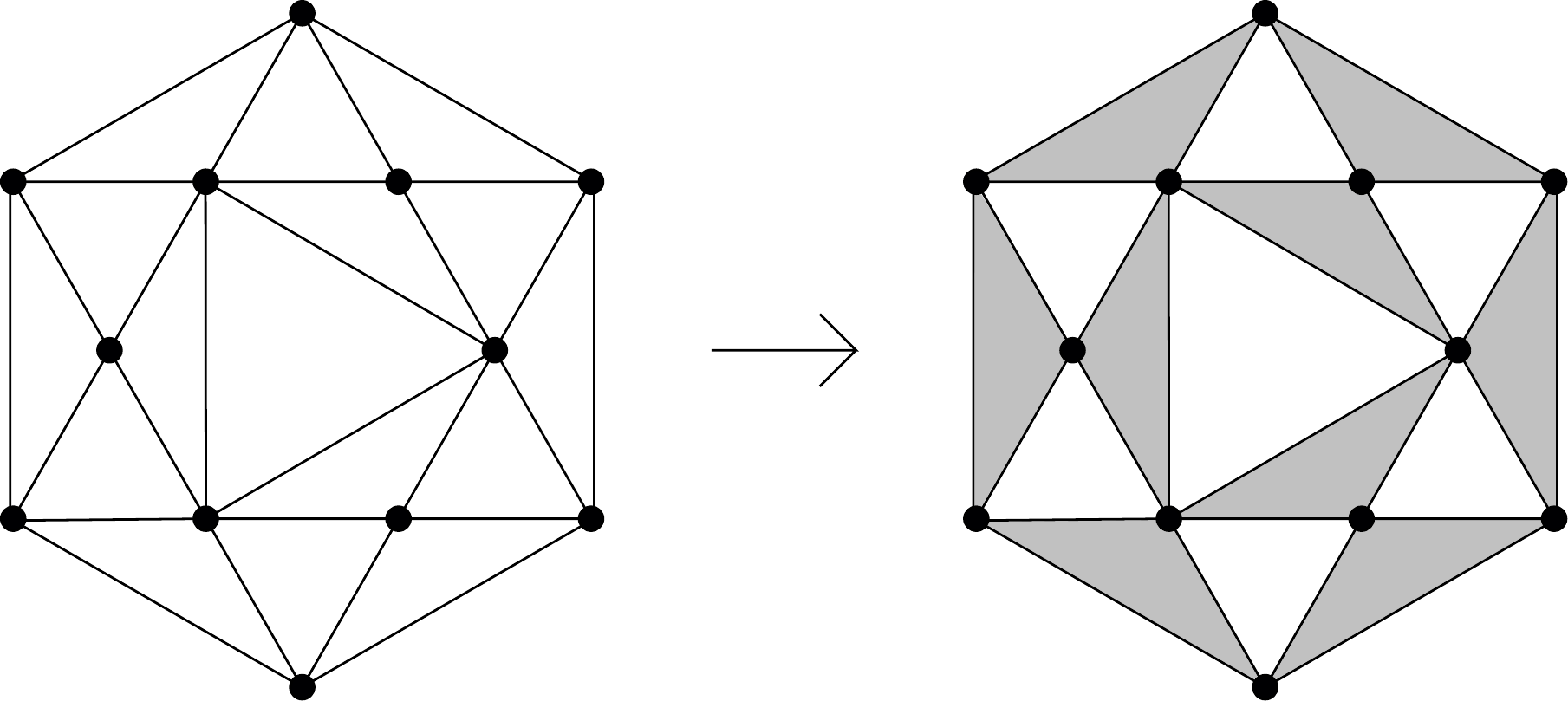}
\end{center}
\end{defn}

\begin{conj}\label{nashwilliamstri}
Suppose that $G = (V,E)$ is a finite simple graph where each vertex has even degree, $|E|$ is a multiple of $3$, and every vertex of $G$ has degree no less than $\frac{3}{4} n$.  Then $G$ admits a decomposition into edge-disjoint triangles.
\end{conj}
These two conjectures are linked via the following natural bijection between partial Latin squares and triangulations of tripartite graphs, illustrated below.
\begin{center}
\includegraphics[width=3.5in]{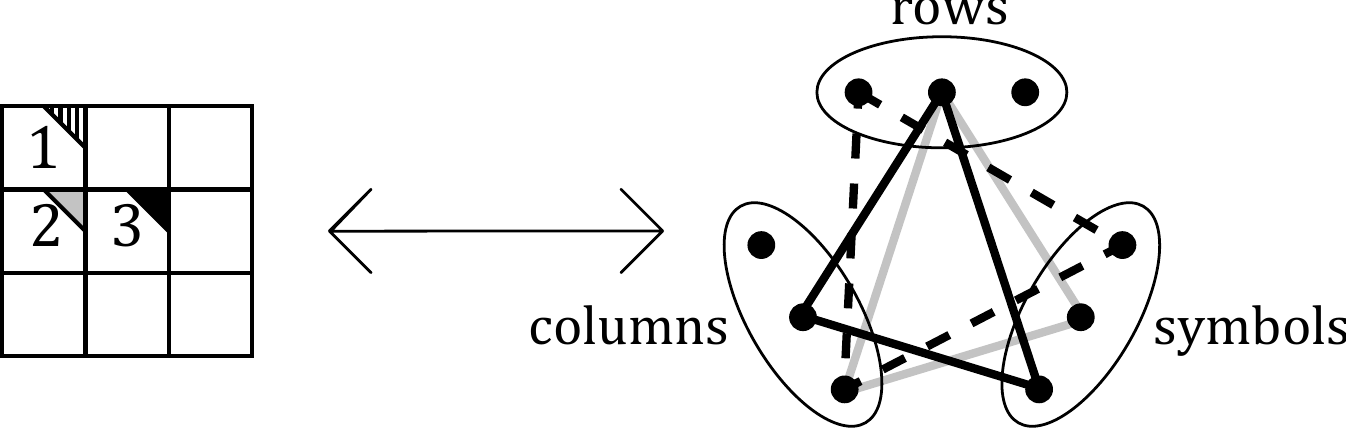}
\end{center}
In particular, given any $ \epsilon$-dense partial Latin square $P$, the above transformation allows us to transform the partial Latin square into a triangulated tripartite graph in which no vertex has more than $\epsilon n$-many neighbors in any one given part.  Therefore, triangulating the \textit{tripartite complement} of this graph corresponds to completing $P$ to a complete Latin square; in this sense, the conjectured Nash-Williams degree bound of $\frac{3}{4}n$ suggests the Daykin-H\"aggkvist conjecture that all $ \frac{1}{4}$-dense partial Latin squares are completable.

It bears noting that this bound of $\frac{1}{4}$ is tight; for any $c > 0$, there are known $ \left(\frac{1}{4} + c\right)$-dense partial Latin squares (see Wanless \cite{Wanless_2002}) that cannot be completed.

In the 1984 paper where Daykin and H\"aggkvist formed this conjecture, they also proved the following weaker form of their claim.
\begin{thm}
All $\frac{1}{2^9\sqrt{n}}$-dense partial Latin squares are completable whenever $n \equiv 0 \mod 16$.
\end{thm}
This was strengthened in 1990, by Chetwynd and H\"aggkvist \cite{Chetwynd_Haggkvist_1985}, to the following theorem.
\begin{thm}
 All $ 10^{-5}$-dense partial Latin squares are completable, for $n$ even and no less than $10^7$.
\end{thm}
Using Chetwynd and H\"aggkvist's result along with the above connection between tripartite graphs and partial Latin squares, Gustavsson was able to extend this result to all values of $n$, in exchange for a slightly worse bound on $\epsilon$.
\begin{thm}
All $ 10^{-7}$-dense partial Latin squares are completable.
\end{thm}

\section{Proper and Improper Trades}\label{wherethingsstartforreals}

The proofs in Chetwynd and H\"aggkvist's paper are difficult in parts.  However, the key idea behind their paper --- \textbf{trades on Latin squares} --- is rather simple and elegant.  We define this notion of trade here.
\begin{defn} 
A \textbf{trade} is a pair of partial Latin squares $(P, Q)$ that satisfy the following two properties.
  \begin{itemize}
\item A cell $(i,j)$ is filled in $P$ if and only if it is filled in $Q$.
\item Any row or column in $P$ contains the same symbols as that same row or column in $Q$.
  \end{itemize}
For example, the following pair of partial Latin squares form a trade.
  \begin{align*}(P,Q) = \left(~
\begin{array}{|c|c|c|c|}
\hline \colCell{umbra}{1}&  &\colCell{umbra}{3} & \\
\hline &  &  & \\
\hline \colCell{umbra}{3}&  & \colCell{umbra}{1} & \\
\hline  &  &  & \\
\hline
\end{array} , ~
\begin{array}{|c|c|c|c|}
\hline\colCell{umbra}{3}&  &\colCell{umbra}{1} & \\
\hline  &  &  & \\
\hline \colCell{umbra}{1} & &\colCell{umbra}{3} & \\
\hline &  &  & \\
\hline
\end{array}~\right)
\end{align*}
\end{defn}

In particular, Chetwynd and H\"aggkvist repeatedly use trades as a way to perform \textbf{small, local} modifications on rather large Latin squares.  This is done as follows:\ suppose that $(P,Q)$ is a trade and $L$ is a completion of $P$.  Look at the array $M$ formed by taking $L$ and replacing all of $P$'s cells in $L$ with those of $Q$; by definition, this new array is still Latin. 
\begin{align*} L =
\begin{array}{|c|c|c|c|}
\hline\colCell{umbra}{1} & 2 &  \colCell{umbra}{3} & 4\\
\hline 4 & 1 & 2 & 3\\
\hline \colCell{umbra}{3} & 4 & \colCell{umbra}{1} & 2\\
\hline 2 & 3 & 4 & 1\\
\hline
\end{array} \longmapsto M = 
\begin{array}{|c|c|c|c|}
\hline \colCell{umbra}{3}& 2 & \colCell{umbra}{1} & 4\\
\hline 4 & 1 & 2 & 3\\
\hline \colCell{umbra}{1}& 4 & \colCell{umbra}{3} & 2\\
\hline 2 & 3 & 4 & 1\\
\hline
\end{array}
\end{align*}
Trades of the above form, that consist of a pair of $2 \times 2$ subsquares, are particularly useful because they are the simplest trades that exist.   Call these trades $\mathbf{2} \times \mathbf{2}$ \textbf{trades}; we will make frequent use of them throughout this paper.

Using trades, a rough outline of Chetwynd and H\"aggkvist's paper can be thought of as the following.
\begin{enumerate}
\item Construct an $n \times n$ Latin square $L$ in which every cell is involved in ``many'' well-understood $2 \times 2$ trades. 
\item For every filled cell $(i,j)$ in $P$, use this structure on $L$ to find a ``simple'' trade on $L$ such that performing this trade causes $L$ and $P$ to agree at $(i,j)$.  (``Simple'' here means that it should not be too difficult to find using our given structure, nor should this trade disturb the contents of too many cells in $L$.)
\item If such trades can be found for every cell $(i,j)$, such that none of these trades overlap (i.e.\ no cell is involved in more than one trade), then it is possible to apply all of these trades simultaneously to $L$. Doing this results in a Latin square that agrees with $P$ at every filled cell of $P$.
\end{enumerate}

As mentioned before, these methods work for $ 10^{-5}$-dense partial Latin squares when $n$ is even and no less than $10^7$.  However, they do not seem to work on denser partial Latin squares.  The main reason for this is that using small trades creates very strong \textbf{local} and \textbf{global} constraints on our partial Latin square $P$ and our constructed Latin square $L$, in the following ways.
\begin{enumerate}
\item  In general, constructions for $n \times n$ Latin squares where every cell is involved in many small well-understood trades are not yet well understood.  In particular, the Chetwynd and H\"aggkvist paper relies on the existence of Latin squares $L$ where every cell is involved in $n/2$ distinct $2 \times 2$ trades:\ however, these squares appear to only exist in the case that $n$ is even\footnote{More generally, for a fixed constant $c$, Latin squares $L$ where every cell is involved in $n/c$ many $2 \times 2$ trades appear to be difficult to find or construct or find whenever $n$ is odd.}.
\item Moreover, if we want to follow Chetwynd and H\"aggkvist's blueprint, we will need to find a large collection of disjoint trades on $L$:\ namely, one for every filled cell in $P$.  In doing this, we need to ensure that no given row (or column, or symbol) gets used ``too often'' in our trades; otherwise, it is possible that we will need to use that row at a later date to fix some other cell in $P$, and we will be unable to find a nonoverlapping trade.  This is a strong \textbf{local} constraint, as it requires us to reserve in every row/column/symbol a large swath of ``available'' cells which we have not disturbed, so that we can use their structure to construct future trades.  This also forces us to do a lot of normalization work before and during the search for these trades, in order to preserve this structure.  (This is the ``difficult'' part of Chetwynd and H\"aggkvist's proof, which otherwise is as straightforward as our earlier outline suggests.)
\item Finally, we also have a large amount of \textbf{global} constraints that we are running into.  In order to find any of these trades, we need to preserve a large amount of structure in $L$.  However, using this structure means that we need to ensure that most of $L$ still looks like the well-structured square we started with; consequently, each trade requires much more structure than just the cells it locally disturbs.
\end{enumerate}

Given the above issues, it may seem like the technique of using trades to complete partial Latin squares is a dead-end.  However, we can overcome many of these restrictions by using the concept of  \textbf{improper} Latin squares and trades, as introduced by Jacobson and Matthews \cite{Jacobson_Matthews_1996} in a 1996 paper on generating random Latin squares.  We define these objects below.
 \begin{defn}  A \textbf{improper Latin square} $L$ is an $n \times n$ array, each cell of which contains a nonempty signed subset of the symbols $\{1, \ldots n\}$, such that the signed sum of any symbol across any row or column is 1.

A quick example.
  \begin{align*}
\begin{array}{|c|c|c|c|}
\hline 1 & 2 & 3 & 4\\
\hline 4 & 1 & 1 & 3 + 2 - 1\\
\hline 3 & 4 & 2 & 1\\
\hline 2 & 3 & 4 & 1\\
\hline
\end{array}
\end{align*}

Analogously, we define a \textbf{partial improper Latin square} as an $n \times n$ array, each cell of which contains a nonempty signed subset of the symbols $\{1, \ldots n\}$, such that the signed sum of any symbol across any row or column is either $0$ or $1$, and an \textbf{improper trade} as simply a pair of partial improper Latin squares that share the same set of filled cells and the same signed symbol sums across any row or column.
 \end{defn}
Essentially, improper Latin squares exist so that the following kinds of things can be considered trades.
   \begin{align*} P =
\begin{array}{|c|c|c|c|}
\hline \colCell{umbra}{a} & & \colCell{umbra}{b} & \\
\hline  &  &  & \\
\hline \colCell{umbra}{b} &  & \colCell{umbra}{c} & \\
\hline & & & \\
\hline
\end{array} \longmapsto Q = 
\begin{array}{|c|c|c|c|}
\hline \colCell{umbra}{b} & & \colCell{umbra}{a} & \\
\hline  & &  & \\
\hline \colCell{umbra}{a} & & \colCell{umbra}{b + c - a} & \\
\hline  &  & & \\
\hline
\end{array}
\end{align*}
Call these trades \textbf{improper $\mathbf{2 \times 2}$ trades:} in practice, these will be the only improper trades that we need to use.

The main use of these improper $2 \times 2$ trades is that they let us ignore the ``local'' constraints described earlier:\ because we do not need a cell to be involved in a proper $2 \times 2$ trade in order to manipulate it, $L$ does not require any local $2 \times 2$ structure.  In particular, this lets us use Latin squares $L$ of odd order, as it is not difficult to construct a Latin square $L$ of odd order with a large global number of $2 \times 2$ subsquares (even though some cells will not be involved in any $2 \times 2$ trades.)  We will still have the global constraints noticed earlier; in general, any system that uses only a few pre-defined types of trades seems like it will need to have some global structure to guarantee that those trades will exist.  However, just removing these local constraints gives us several advantages.
\begin{itemize}
\item Using improper trades, we can complete all partial Latin squares that are $ 9.8 \cdot 10^{-5}$-dense, an improvement on Gustavsson and Chetwynd/H\"aggkvist's results.  If we allow ourselves to examine claims that hold for larger values of $n$, we can marginally improve this to the claim that all $10^{-4}$-dense partial Latin squares are completable, for $n > 1.2 \cdot 10^5.$
\item More interestingly, because we have removed these local constraints, we can now talk about completing  $ \epsilon$-dense partial Latin squares that globally contain no more than $\delta n^2$-many filled cells, where $\epsilon$ and $\delta$ may not be equal.  In other words, we can now differentiate between our global and local constraints; this allows us to (in particular) massively improve our local bound $\epsilon$ at the expense of our global bound $\delta$.  For example, we can use improper trades to complete any $\mathbf{\frac{1}{13}}$\textbf{-dense} partial Latin square, provided that it globally contains no more than $5.7 \cdot 10^{-7} \cdot n^2$ filled cells.  
\item In fact, given any $\epsilon \in \left[ 0, \frac{1}{12} \right)$, and any value of $\delta < \frac{ \left(1-12\epsilon\right)^{2}}{10409}$, we can show that any $ \epsilon$-dense partial Latin square $P$ containing no more than $\delta n^2$ filled cells in total is completable.
\item Furthermore, because we have removed these local constraints, we can eliminate a lot of the ``normalization'' processes and techniques that Chetwynd and H\"aggkvist needed for their trades; consequently, these proof methods are (in some senses) easier to understand.
\end{itemize}

 The following process outlines how we will construct a completion of any such $ \epsilon$-dense partial Latin square $P$ containing no more than $\delta n^2$ filled cells.
\begin{enumerate}
\item First, we will create a Latin square $L$ that globally contains a large number of $2 \times 2$ subsquares.
\item Then, we will show that in any fixed row or column, it is possible to exchange the contents of ``almost any'' two cells using simple trades, provided that we have not disturbed too much of $L$'s global structure.
\item Using the above claim, we will show that given any filled cell $(i,j)$ in $P$, there is a trade that causes $L$ and $P$ to agree at this filled cell without disturbing any cells at which $P$ and $L$ agree.
\item By repeated applications of Step 3, we will turn $L$ into a completion of $P$.
\end{enumerate}

With our goals clearly stated and our techniques described, all that remains for us is to explicitly prove the above claims.
\section{The Proof}

We begin by creating Latin squares with ``many'' well-understood $2 \times 2$ subsquares.
\begin{lem}\label{lem1}
For any $k$ , there is a $2k \times 2k$ Latin square $L$ of the form $\begin{array}{|c|c|} \hline A & B \\ \hline B^T & A^T \\ \hline\end{array}$, with the following property:\ if there are two cells $(i,j), (i',j')$ in opposite quadrants containing the same symbol, then there is a $2 \times 2$ trade that exchanges the contents of these two symbols.

Furthermore, there is a way to extend this construction to an $n \times n$ odd-order square, in such a way that preserves this property at all but $3n + 7$ cells in the new odd-order square.
\end{lem}
\begin{proof}
For even values of $n$, we can simply use the following construction used by Chetwynd and H\"aggkvist in their proof.
\begin{align*}
L = \begin{array}{|c|c|c|c||c|c|c|c|}
\hline 1 & 2 & 3 & 4 & 5 & 6 & 7 & 8 \\
\hline 4 & 1 & 2 & 3 & 8 & 5 & 6 & 7 \\
\hline 3 & 4 & 1 & 2 & 7 & 8 & 5 & 6 \\
\hline 2 & 3 & 4 & 1 & 6 & 7 & 8 & 5 \\
\hhline{|=|=|=|=||=|=|=|=|} 
	   5 & 8 & 7 & 6 & 1 & 4 & 3 & 2 \\
\hline 6 & 5 & 8 & 7 & 2 & 1 & 4 & 3 \\
\hline 7 & 6 & 5 & 8 & 3 & 2 & 1 & 4  \\
\hline 8 & 7 & 6 & 5 & 4 & 3 & 2 & 1  \\
\hline
\end{array}
\end{align*}

In general, their construction is the following:\ if we set $A$ as the $k \times k$ circulant matrix on symbols $\{1, \ldots k\}$ and $B$ as the $k \times k$ circulant matrix on symbols $\{k+1, \ldots 2k\}$, we can define $L$ as the $n \times n$ Latin square given by $\begin{array}{|c|c|} \hline A & B \\ \hline B^T & A^T \\ \hline\end{array}$.  An example for $n = 8$ is provided above.

This Latin square $L$, as noted by Chetwynd and H\"aggkvist, has the following property:\ every cell in $L$ is involved in precisely $n/2$ distinct $2 \times 2$ subsquares.  To see this, notice that another way to describe $L$ is as follows.
\begin{align*}
L(i,j) = \left\{ \begin{array}{ll}
j-i + 1 \mod k & \textrm{for } i, j \leq k, \\
i-j + 1 \mod k   & \textrm{for } i >  k, j > k, \\
(j-i + 1 \mod k) + k  & \textrm{for } i \leq k, j > k, \\
 (i-j + 1 \mod k) + k  & \textrm{for }  i > k, j \leq k, \\
\end{array} \right.
\end{align*}
With this done, take any cell $(i,j)$ in our Latin square $L$.  Pick any other cell $(i,y)$ from the same row as $(i,j)$, but from the opposite quadrant.  Now, pick the cell $(x,j)$ so that it has the same symbol as the symbol in $(i,y)$.  With this done, we can observe that 
\begin{align*}
y-i + 1 \equiv x-j + 1 \mod k.
\end{align*}
This implies that
\begin{align*}
y-x + 1 \equiv i-j + 1 \mod k;
\end{align*}
i.e that the symbols in cells $(i,j)$ and $(x,y)$ are the same.  Therefore, any cell $(i,j)$ in our Latin square $L$ is involved in precisely $n/2$-many $2 \times 2$ subsquares:\ one for every cell in the same row and opposite quadrant.

For $n = 4k +1$ for some $k$, we can augment Chetwynd and H\"aggkvist's construction as follows.  First, use the construction above to construct a $4k \times 4k$ Latin square $L$. Now, consider the transversal of $L$ consisting of the following cells.
\begin{center}
$(1,1), (2,3), (3,5), (4,7) \ldots (k, 2k-1), $\\
$(k+1,2k+1), (k+2,2k+3),  \ldots (2k, 4k-1), $\\
$(2k+1, 2k+2), (2k+2,2k+4), \ldots (3k, 4k)$\\
$(3k+1, 2), (3k+2, 4), \ldots (4k,2k)$.\\
\end{center}
\footnotesize
\begin{align*}
\begin{array}{|c|c|c|c|c|c||c|c|c|c|c|c|}
\hline \colCell{umbra}{1} & 2 & 3 & 4 & 5 & 6 &A & B & C & D & E & F \\
\hline 6 & 1 & \colCell{umbra}{2} & 3 & 4 & 5 &F & A & B & C & D & E \\
\hline 5 & 6 & 1 & 2 & \colCell{umbra}{3} & 4 &E & F & A & B & C & D \\
\hline 4 & 5 & 6 & 1 & 2 & 3 & \colCell{umbra}{D} & E & F & A & B & C \\
\hline 3 & 4 & 5 & 6 & 1 & 2 &C & D & \colCell{umbra}{E} & F & A & B \\
\hline 2 & 3 & 4 & 5 & 6 & 1 &B & C & D & E & \colCell{umbra}{F} & A \\
 \hhline{>{\doublerulesepcolor{white}}|=|=|=|=|=|=|t|=|=|=|=|=|=|}
         A & F & E & D & C & B &1 & \colCell{umbra}{6} & 5 & 4 & 3 & 2 \\
\hline B & A & F & E & D & C &2 & 1 & 6 & \colCell{umbra}{5} & 4 & 3 \\
\hline C & B & A & F & E & D &3 & 2 & 1 & 6 & 5 & \colCell{umbra}{4} \\
\hline D & \colCell{umbra}{C} & B & A & F & E &4 & 3 & 2 & 1 & 6 & 5 \\
\hline E & D & C & \colCell{umbra}{B} & A & F &5 & 4 & 3 & 2 & 1 & 6 \\
\hline F & E & D & C & B & \colCell{umbra}{A} &6 & 5 & 4 & 3 & 2 & 1 \\
\hline
\end{array}\\
\textrm{(The above transversal in a }13 \times 13\textrm{ Latin square.)}~\quad
\end{align*}
\normalsize

Using this transversal, turn $L$ into a $4k+1 \times 4k+1$ Latin square $L'$ via the following  construction:\ take $L$, and augment it by adding a new blank row and column.  Fill each cell in this blank row with the corresponding element of our transversal that lies in the same column as it; similarly, fill each cell in this blank column with the corresponding transversal cell that is in the same row.  Finally, replace the symbols in every cell in our transversal (as well as the blank cell at the intersection of our new row and column) with the symbol $4k+1$.  This creates an $n \times n$ Latin square such that all but $3n-2$ cells are involved in precisely $(n/2) - 2$ distinct $2\times 2$ subsquares.

For $n = 4k-1$, things are slightly more difficult.  While we can use our earlier construction to create a $4k-2 \times 4k-2$ Latin square, the resulting square does not have a transversal.  However, we can use our $2 \times 2$ substructure to slightly modify this square so that it will have a transversal, and then proceed as before.  We outline the process for creating this transversal below.
\begin{enumerate}
\item First, use the Chetwynd and H\"aggkvist construction to create a $(4k-2) \times (4k-2)$ Latin square $L$.
\item In our discussion earlier, we noted that for any pair of cells $(i,j), (i,k)$ in the same row but from different quadrants, there is a $2 \times 2$ subsquare that contains those two elements.  Take the $2 \times 2$ subsquare corresponding to the cells containing $2$ and $4k-2$ in the last row, and perform the $2 \times 2$ trade corresponding to this subsquare.
\item Similarly, take the $2 \times 2$ subsquare corresponding to the cells containing $2k-1$ and $4k-2$ in the far-right column, and perform the $2 \times 2$ trade corresponding to this subsquare.
\item With these two trades completed, look at the four cells determined by the last two rows and columns of our Latin square.  They now form a $2 \times 2$ subsquare of the form $\begin{array}{|c|c|}\hline 1 & 4k-2 \\ \hline 4k-2 & 1\\  \hline \end{array}$.  Perform the trade corresponding to this $2 \times 2$ subsquare.
\end{enumerate}

Once this is done, we can find a transversal by simply taking the cells
\begin{center}
$(1,1), (2,3), (3,5), (4,7) \ldots (k, 2k-1), $\\
$(k+1,2k), (k+2,2k+2),  \ldots (2k-1, 4k-4), $\\
$(2k, 2k+1), (2k+2,2k+3), \ldots (3k - 2, 4k - 3),$\\
$(3k - 1, 2), (3k, 4), \ldots (4k - 3,2k - 2),$\\
$(4k-2, 4k-2).$
\end{center}
\footnotesize
\begin{align*}
\begin{array}{|c|c|c|c|c|c|c||c|c|c|c|c|c|c|}
\hline \colCell{umbra}{1} & 2 & 3 & 4 & 5 & 6 & G      &      A & B & C & D & E & F & 7 \\
\hline 7 & 1 & \colCell{umbra}{2} & 3 & 4 & 5 & 6      &      G & A & B & C & D & E & F \\
\hline 6 & 7 & 1 & 2 & \colCell{umbra}{3} & 4 & 5      &      F & G & A & B & C & D & E \\
\hline 5 & 6 & 7 & 1 & 2 & 3 & \colCell{umbra}{4}      &      E & F & G & A & B & C & D \\
\hline 4 & 5 & 6 & 7 & 1 & 2 & 3      &      \colCell{umbra}{D} & E & F & G & A & B & C \\
\hline 3 & 4 & 5 & 6 & 7 & 1 & 2      &      C & D &\colCell{umbra}{E} & F & G & A & B \\
\hline G & 3 & 4 & 5 & 6 & 7 & 1      &      B & C & D & E & \colCell{umbra}{F} & 2 & A \\
 \hhline{>{\doublerulesepcolor{white}}|=|=|=|=|=|=|=|t|=|=|=|=|=|=|=|}
\hline A & G & F & E & D & C & B      &      1 & \colCell{umbra}{7} & 6 & 5 & 4 & 3 & 2 \\
\hline B & A & G & F & E & D & C      &      2 & 1 & 7 & \colCell{umbra}{6} & 5 & 4 & 3 \\
\hline C & B & A & G & F & E & D      &      3 & 2 & 1 & 7 & 6 & \colCell{umbra}{5} & 4 \\
\hline D & \colCell{umbra}{C} & B & A & G & F & E      &      4 & 3 & 2 & 1 & 7 & 6 & 5 \\
\hline E & D & C & \colCell{umbra}{B} & A & G & F      &      5 & 4 & 3 & 2 & 1 & 7 & 6 \\
\hline F & E & D & C & B & \colCell{umbra}{A} & 7      &      6 & 5 & 4 & 3 & 2 & G & 1 \\
\hline 2 & F & E & D & C & B & A      &      7 & 6 & 5 & 4 & 3 & 1 & \colCell{umbra}{G} \\
\hline
\end{array}\\
\qquad\textrm{(The above transversal in a }14 \times 14\textrm{ Latin square.)} \qquad \quad
\end{align*}
\normalsize
Using this transversal, we can extend our Latin square to a $4k-1 \times 4k-1$ Latin square using the same methods as in the $n = 4k+1$ case.  This completes the first step of our outline.
\end{proof}

\indent Our next lemma, roughly speaking, claims the following:\ given any Latin square generated by Lemma $\ref{lem1}$, we can pick any row and exchange the contents of ``many'' pairs of cells within that row.  Moreover, we can do this without disturbing other cells in that row, more than 16 cells in our entire Latin square, or some prescribed small set of symbols that we'd like to avoid disturbing in general. 

In fact, the following result claims that we can do this \textit{repeatedly}:\ i.e.\ that we can apply this result not just to Latin squares generated by Lemma $\ref{lem1}$, but to Latin squares generated by Lemma $\ref{lem1}$ that have had the contents of up to $k n^2$ many cells disturbed by such trades, for some constant $k$ that we will determine later. 

We state this claim formally below.
\begin{lem}\label{lem2}
Initially, let $L$ be one of the $n\times n$ Latin squares constructed by Lemma $\ref{lem1}$, and $P$ be an $ \epsilon$-dense partial Latin square.  Perform some sequence of trades on $L$, and suppose that after these trades are completed that the following holds:\ no more than $kn^2$ of $L$'s cells either have had their contents altered via such trades, or were part of the $3n + 7$ potentially-disturbed cells that were disturbed in the execution of Lemma $\ref{lem1}$.  As well, fix any set $\{t_1, \ldots t_a\}$ of symbols.

Fix any positive constant $d > 0$.  Then, for any row $r_1$ of $L$ and all but 
\begin{itemize}
\item $2\frac{k}{d}n + \epsilon n + a$ choices of $c_1$, and
\item $4\frac{k}{d}n + 2d n + 3\epsilon n + a+ 1$ choices of $c_2,$
\end{itemize}
there is a trade on $L$ that
\begin{itemize}
\item does not change any cells on which $P$ and $L$ currently agree,
\item changes the contents of at most 16 cells of $L$,
\item does not use any of the symbols $\{t_1, \ldots t_a\}$, and
\item swaps the symbols in the cells $(r_1, c_1)$ and $(r_2, c_2)$,
\end{itemize}
as long as the following two equations hold.
\begin{enumerate}
\item $3 \leq n - 4\frac{k}{d}n - 6d n  - 6 \epsilon n - 3a, \textrm{ and}$
\item $12 \leq n - 12d n -  12\epsilon n - 4a.$
\end{enumerate}
\end{lem}
\begin{proof}

Call a row, column, or symbol in $L$ $d$-\textbf{overloaded} (or just overloaded, for short) if $> d n$ of the entries in this row/column/symbol have had their contents changed by the trades we have performed thus far on $L$, counting the $3n+7$ possibly-disturbed cells from $L$ 's construction as such changed cells.  Note that no more than $\frac{kn^2}{d n} = \frac{k}{d} n$ rows, columns, or symbols are overloaded.

Intuitively, overloaded rows are going to be ``difficult'' to work with.  Because most of the structure in our Latin square no longer exists within that row, we will have relatively few ways to reliably manipulate the cells in that row.  Conversely, if some row is not overloaded, we know that the contents of most of the cells within this row have not been disturbed; in theory, this will make manipulating this row much easier, as we will have access to a lot of the structure we have built into our Latin square $L$.  (Similar comments apply to overloaded symbols and columns.)

With these comments as our motivation, we begin constructing our trade.  Fix some row $r_1$:\ we want to show that for most pairs of cells within this row, there is a trade which exchanges their contents without disturbing many other cells in $L$.  Naively, we might hope that for most pairs of cells in our row, we can find the following trade.
\begin{center}
\includegraphics[width=2.7in]{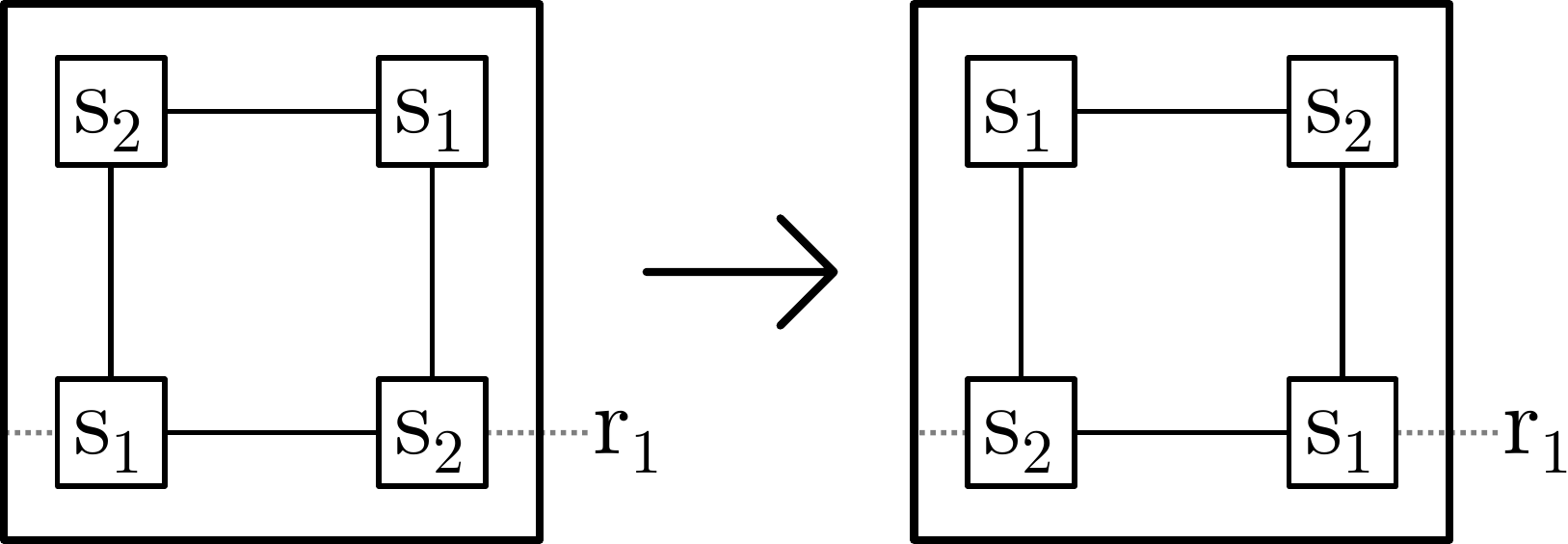}
\end{center}
If this situation occurs, we can simply perform the $2 \times 2$ trade illustrated above to swap the two cells containing $s_1$ and $s_2$.  The issue, however, is that this situation may never come up:\ if $r_1$ is an overloaded row, for example, it is entirely possible that \textbf{none} of its elements are involved in \textbf{any} $2 \times 2$ subsquares.  To fix this, we use the technique of \textbf{improper trades}:\ specifically, we will choose some other row $r_2$, and perform the following improper trade.
\begin{center}
\includegraphics[width=3.2in]{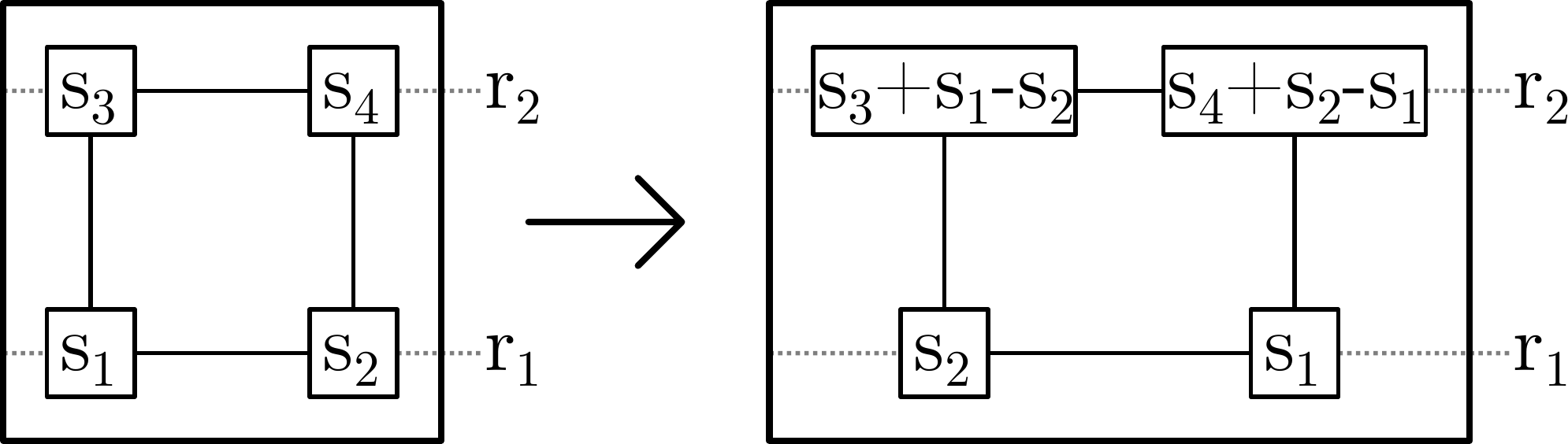}
\end{center}

This accomplishes our original goal of exchanging these two elements in $r_1$; however, we may now have an improper Latin square if either $s_3 \neq s_2$ or $s_4 \neq s_1$.  The aim of this lemma is to construct a \textbf{proper} trade on our Latin square.  Therefore, we need a way to augment this trade so that it becomes a proper one.  This is not too difficult to do; by repeatedly stringing together improper $2 \times 2$ trades that use nonoverloaded rows/columns/symbols wherever possible, and using cells that have not been disturbed by earlier trades where possible, we can augment the improper trade above to one of two possible proper trades. We illustrate the first of these below.
\begin{center}
\includegraphics[width=3.9in]{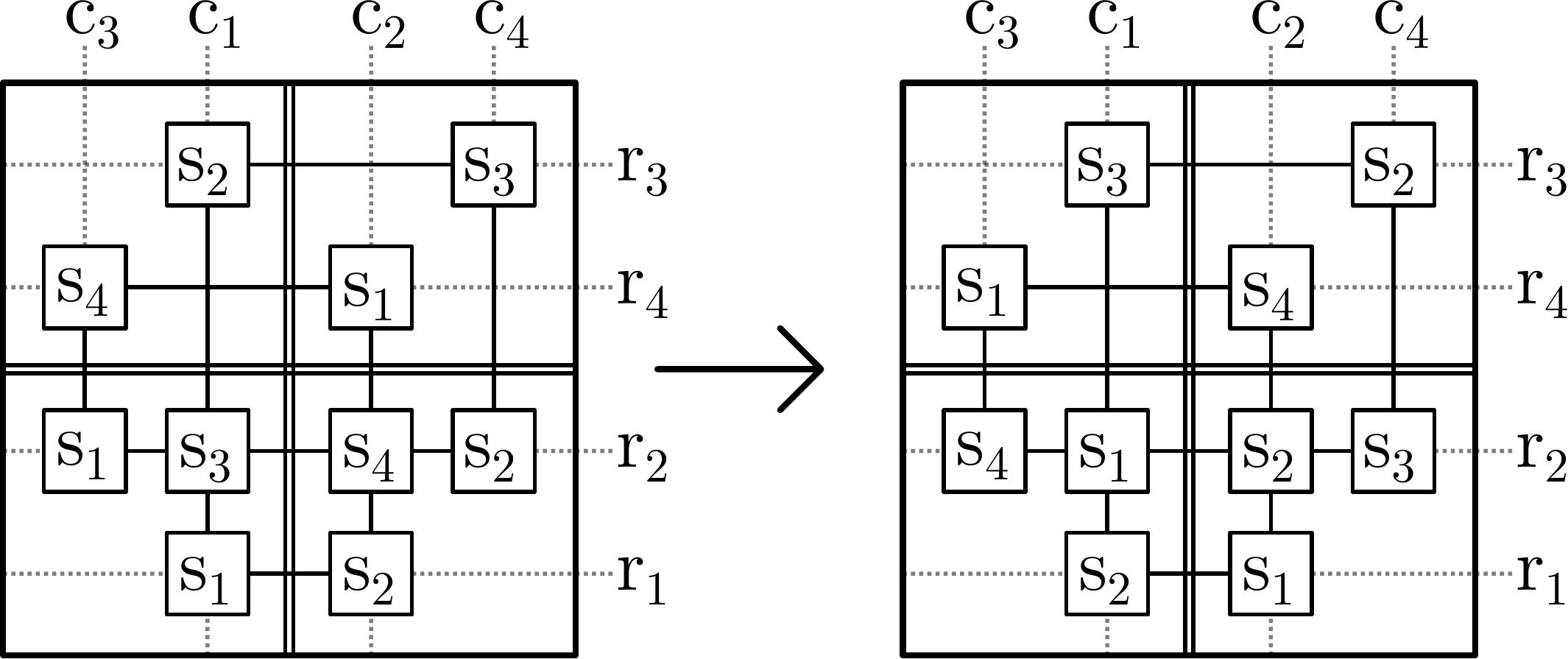}
\end{center}

The simpler trade of the two, illustrated above, occurs in the following situation.  Look at the two cells $(r_1, c_1)$ and $(r_1, c_2)$, and suppose that the symbols in both of these cells are not overloaded.  For each of these two cells, there are two possibilities:\ either the cell $(r_1, c_i)$ is still in the same quadrant that the symbol $s_i$ started in, or it has been permuted to the other quadrant.  The trade we have drawn above occurs in the situation where either both of these cells are still in the same quadrants that their corresponding symbols started in, or when neither of these cells are in the same quadrants that their corresponding symbols start in.  (This condition is equivalent to asking that rows $r_3, r_4$ are both in the same half of our square, which is necessary for our choice of $r_2$.)

We show that this trade can always be found using the following heuristic:\ we will choose the rows/columns/symbols involved in this trade one by one, choosing each so that as many of the variables determined by our choice are involved in nonoverloaded rows/columns/symbols as possible.  Furthermore, we will also attempt to insure that as many cells as possible in our trade have never had their contents disturbed by either earlier trades on $L$ or from being part of the potential $3n+7$ disturbed cells in $L$'s construction.  Finally, we will make sure that our choices never involve cells where $L$ and $P$ currently agree.

Start by choosing $s_1$ such that the following properties hold.
\begin{itemize}
\item The symbol $s_1$ is not overloaded.  As well, if $c_1$ is the column such that $(r_1, c_1)$ contains $s_1$, the column $c_1$ should also not be overloaded.  This eliminates at most $2\frac{k}{d}n$ choices.
\item The cell $(r_1, c_1)$ is not one at which $P$ and $L$ currently agree.  This eliminates at most $\epsilon n$ choices.
\item The symbol in the cell $(r_1, c_1)$ is not one of $\{t_1, \ldots t_a\}$.  This eliminates at most $a$ choices.
\end{itemize}
Therefore, we have 
\begin{align*}
n - 2\frac{k}{d }n - \epsilon n - a
\end{align*}
choices for $s_1$, as claimed.

We choose the second symbol $s_2$ so that a similar set of properties hold.
\begin{itemize}
\item The symbol $s_2$ should not be the same as $s_1$, nor should it be one of $\{t_1, \ldots t_a\}$.  This eliminates at most $a+1$ choices.
\item The cell in column $c_1$ containing $s_2$ has not been used in any earlier trades; as well, the cell containing the symbol $s_1$ in the column $c_2$ has not been used by any earlier trades.  Because neither $c_1$ nor $s_1$ are overloaded, we know that at most $dn$ entries in either of these objects have been used in previous trades.  Therefore, this restriction eliminates at most $2dn$ choices.
\item The rows $r_3, r_4$ containing these two undisturbed cells are not overloaded.  As well, we ask that neither the symbol $s_2$ nor the column $c_2$ is overloaded.  This eliminates at most $4\frac{k}{d}n$ choices.
\item Neither the cell $(r_1, c_2)$ nor the two undisturbed cells are cells at which $P$ and $L$ currently agree.  This eliminates at most $3\epsilon n$ more choices.
\end{itemize}
This leaves
\begin{align*}
n - 4\frac{k}{d}n - 2d n -  3\epsilon n - a - 1
\end{align*}
choices for $s_2$, again as claimed.

We have one final choice to make here:\ the row $r_2$.  Observe that because the cells $(r_1, c_i)$ are either both in the same quadrant that the cell containing $s_i$ in row $r_1$ started in, or both permuted to the quadrants they were not in, the cells $(r_3, c_1)$ and $(r_4, c_2)$ are both either in the top half or both in the bottom half of our Latin square.  Using this observation, we can choose $r_2$ so that the following conditions hold.

\begin{itemize}
\item The row $r_2$ is in the opposite half from the rows $r_3, r_4$, and is not $r_1$.  This eliminates at most $\lceil n/2 \rceil + 1$  choices.
\item None of the cells $(r_2, c_1)$, $(r_2, c_2)$, $(r_2, c_3)$, $(r_2, c_4)$,  $(r_3, c_4)$, or $(r_4, c_3)$  have been used in prior trades.  Because neither the columns $c_1$, $c_2$, nor the symbols $s_1, s_2$, nor the rows $r_3, r_4$ are overloaded, this restriction eliminates at most $6d n$ choices.  
\item None of these cells are cells at which $P$ and $L$ currently agree.  This eliminates at most $6\epsilon n$ choices.
\item Neither $s_3$ or $s_4$ are equal to any of the symbols $\{t_1, \ldots t_a\}$.  This eliminates at most $2a$ choices.
\end{itemize}
This leaves
\begin{align*}
\left\lfloor\frac{n}{2}\right\rfloor - 6d n -  6\epsilon n - 2a - 1
\end{align*}
choices for the row $r_2$. 

Notice that because $r_2$ has been chosen to be from the opposite half of $r_3, r_4$, and none of these cells nor the earlier have been disturbed by earlier trades, we know that the symbol $s_3$ is in the cell $(r_3, c_4)$ and the symbol $s_4$ is in the cell $(r_4, c_3)$; this is because of $L$'s previously-discussed ``many $2 \times 2$ subsquares'' structure.  Therefore, whenever we can make all of these choices,  we have constructed the trade that we claimed was possible.  

The slightly more complex trade that we have to consider is when the above choice of $r_2$ is impossible; i.e.\ when one of the rows $r_3,r_4$ is in the top half and the other is in the bottom half of $L$.  We deal with this obstruction via the following trade, which (again) was constructed by repeatedly applying improper $2 \times 2$ trades.
\begin{center}
\includegraphics[width=5in]{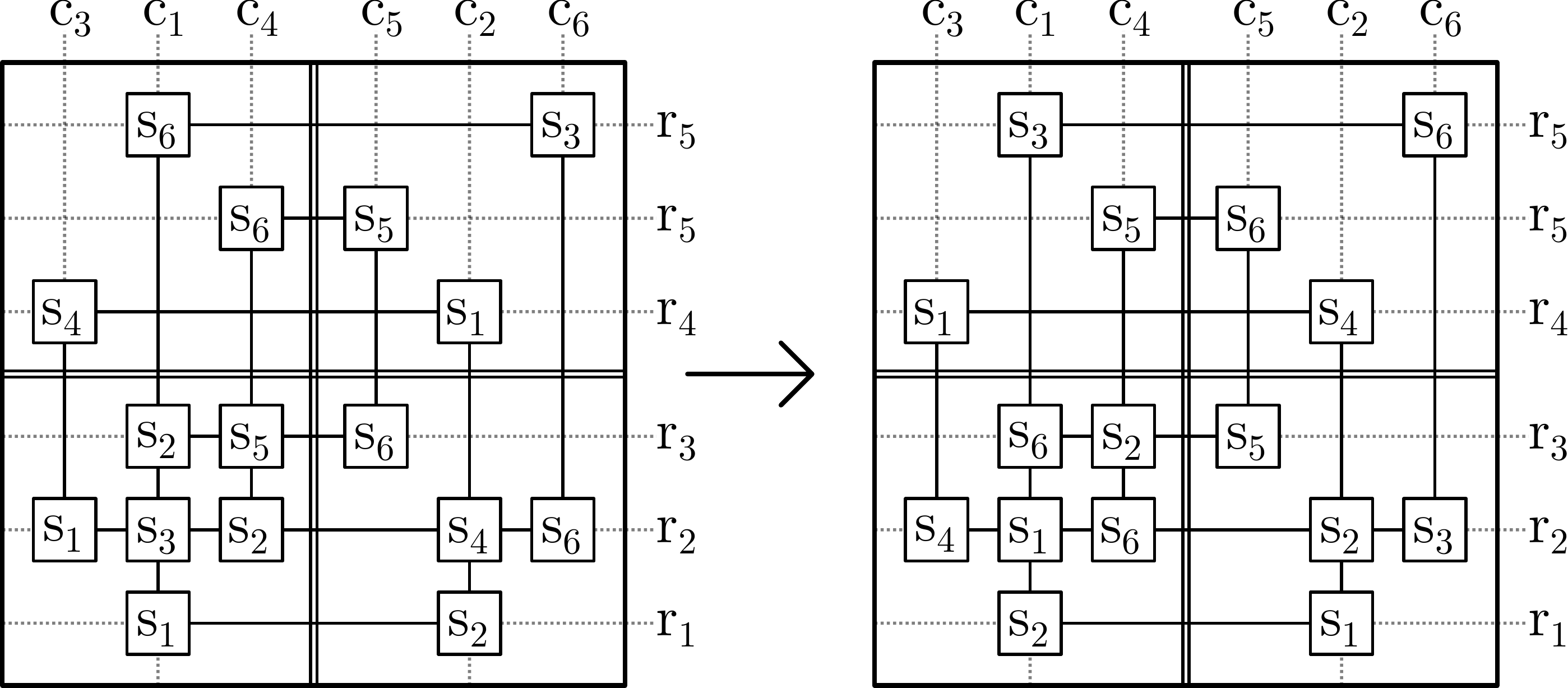}
\end{center}

Choose $s_1, s_2$ exactly as before.  We will now choose values of $r_2$ and $s_6$ so that the trade illustrated above exists.

We start out with $n$ possible choices of $r_2$:\ for any such choice, either $r_2$ and $r_3$ will be in different halves, or $r_2$ and $r_4$ will be in different halves of our Latin square.  Assume without loss of generality that $r_2$ and $r_4$ are in different halves of $L$:\ the case where $r_2$ and $r_3$ are in different halves is identical.  In this case, pick $r_2$ such that the following properties hold.
\begin{itemize}
\item  None of the cells $(r_2, c_1)$, $(r_2, c_2)$, $(r_2, c_3)$, $(r_2, c_4)$,  $(r_4, c_3)$, or $(r_3, c_4)$  have been disturbed in prior trades.  Because neither the columns $c_1$, $c_2$, nor the symbols $s_1, s_2$, nor the rows $r_3, r_4$ are overloaded, this restriction eliminates at most $6d n$ choices. Notice that because $r_2$ and $r_4$ are in different halves, we know that the same symbol $s_4$ is in $(r_4, c_3)$ and $(r_2, c_2)$; again, this is caused by $L$'s well-understood ``many $2\times 2$ subsquares'' structure.  However, unlike our earlier case, we cannot make a similar assumption for the cell $(r_3, c_4)$.
\item None of the following are overloaded:\ the symbol $s_3$, the symbol $s_5$, the row $r_2$, or the column $c_4$.   Furthermore, none of the cells determined by these choices are in use in our trade thus far.  This eliminates at most $4\frac{k}{d}n + 2$ choices.
\item None of these cells are cells at which $P$ and $L$ currently agree.  This eliminates at most $6\epsilon n$ choices.
\item None of the symbols $s_3, s_4,$ or $s_5$ are equal to any of the symbols $\{t_1, \ldots t_a\}$.  This eliminates at most $3a$ choices.
\end{itemize}
This leaves at least
\begin{align*}
n - 4\frac{k}{d}n - 6d n  - 6 \epsilon n - 3a -2
\end{align*}
choices.

Before making our final choice, notice that in the original $\begin{array}{|c|c|}\hline A & B \\ \hline B^T & A^T \\ \hline\end{array}$ form of our Latin square $L$, the symbols $s_3$ and $s_5$ have to originally have came from the same quadrant.  This is because none of the cells $(c_1, r_2),$ $(c_4, r_3),$ $(c_1, r_3),$ $(c_4, r_2)$ have been disturbed in prior trades, and the two cells $(c_2, r_4), (c_3, r_2)$ contain the same symbol $s_2$.

Using this observation, choose $s_6$ so that the following properties hold.
\begin{itemize}
\item $s_6$ should be in the opposite half from the symbols $s_3, s_5$.  This eliminates at most $\lceil n/2\rceil$ choices.
\item None of the cells containing $s_6$ in columns $c_1, c_4$ or rows $r_2, r_3$, nor the cell containing $s_3$ in column $c_6$, nor the cell containing $s_5$ in column $c_5$, have been used in previous trades.  Because the columns $c_1, c_4$, rows $r_2, r_3$, and symbols $s_3, s_5$  are all not overloaded, this is possible, and eliminates at most $6 d n$ choices.
\item None of these cells are places where $P$ and $L$ agree.  This eliminates at most $6\epsilon n$ choices.
\item The symbol $s_6$ has not been chosen before, nor is it equal to any of the symbols $\{t_1, \ldots t_a\}$.  This eliminates at most $5+a$ choices.
\end{itemize}
This leaves at least 
\begin{align*}
\left\lfloor n/2\right\rfloor  - 6d n  - 6 \epsilon n - a - 5
\end{align*}
choices. 

Using our ``many $2 \times 2$ subsquares'' structure tells us that we have constructed the claimed trade.  Therefore, as long as we can make these choices, we can find one of these two trades.  By looking at all of the choices we make during our proof and choosing the potentially strictest bounds (under certain choices of $d, \epsilon, k, n$) we can see that such trades will exist as long as
\begin{align*}
3 &\leq n - 4\frac{k}{d }n - 6d n  - 6 \epsilon n - 3a, \textrm{ and} \\
12 &\leq n - 12d n -  12\epsilon n - 4a.\\
\end{align*}

\end{proof}
Note that an analogous result holds for exchanging the contents of almost any two cells in a given column using the exact same proof methods.

The next lemma, built off of Lemma $\ref{lem2}$, is the main tool we use in this paper.
\begin{lem}\label{lem3}
As before, let $L$ be one of the $n\times n$ Latin squares constructed by Lemma $\ref{lem1}$, and $P$ be an $ \epsilon$-dense partial Latin square.  Suppose that we have performed a series of trades on $L$ that have changed the contents of no more than $k n^2$ of $L$'s cells.  (Note that we count the $3n+7$ potentially-disturbed cells from the construction of $L$ when we enumerate these changed cells.)

Fix any cell $(r_1,c_1)$ such that $P(r_1,c_1)$ is filled and does not equal $L(r_1,c_1)$.  Then, there is a trade on $L$ that
\begin{itemize}
\item does not change any cells on which $P$ and $L$ currently agree,
\item changes the contents of at most 69 other cells of $L$, and
\item causes $P$ and $L$ to agree at the cell $(r_1,c_1)$, 
\end{itemize}
whenever we satisfy the bound
\begin{align*}
20 &\leq n - 12 n\sqrt{k} -  12\epsilon n.\\
\end{align*}
\end{lem}
\begin{proof}

Let $d > 0$ be some constant corresponding to the notion of an ``overloaded'' row as introduced earlier.  

Suppose that $L(r_1, c_1) =s_1 \neq P(r_1, c_1) = s_2$.  Let $r_2$ be a row and $c_2$ be a column such that $L(r_2, c_1) = L(r_1, c_2) = s_2$.  

Our goal in this lemma is to construct a trade that causes $L$ and $P$ to agree at $(r_1, c_1)$, without disturbing any cells at which $P$ and $L$ already agree.  We will do this using four successive applications of Lemma $\ref{lem2}$, one each on row $r_1$, row $r_3$, column $c_1$, and column $c_3$, as illustrated in the picture below.

\begin{center}
\includegraphics[width=5in]{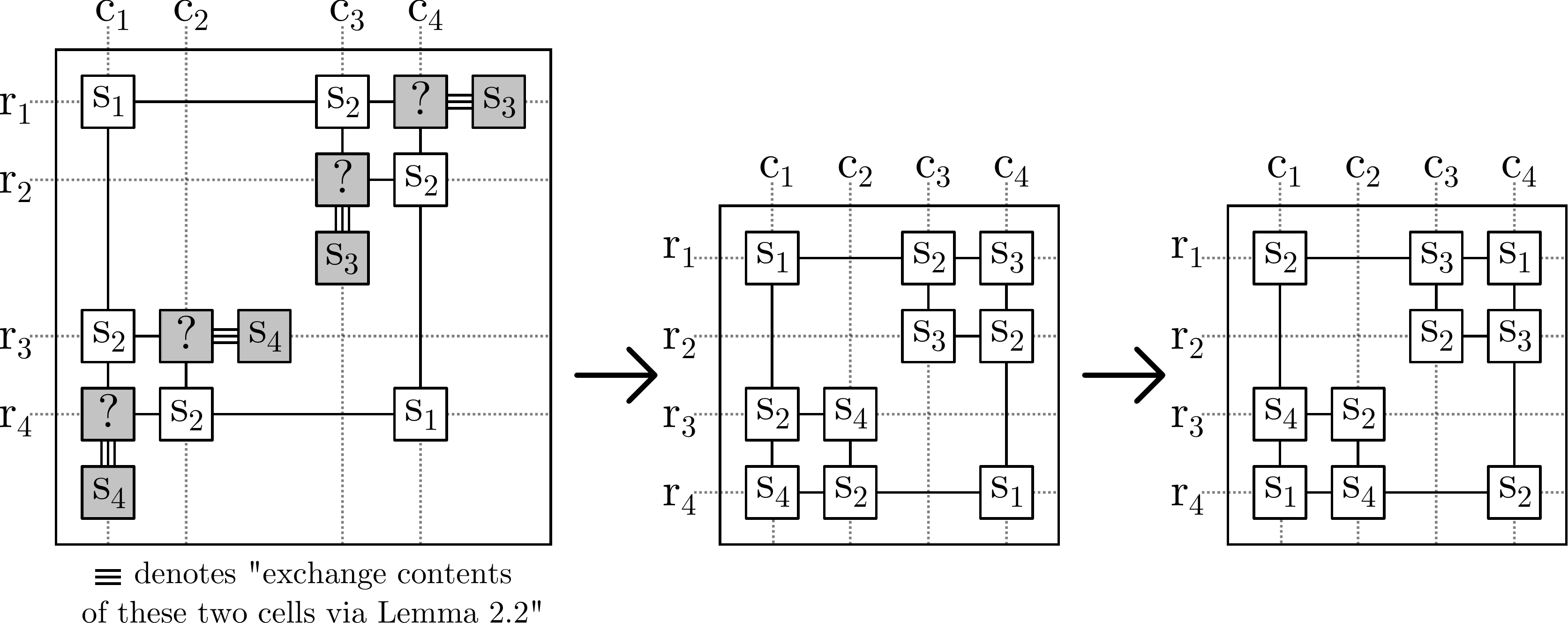}
\end{center}
Assuming we apply Lemma $\ref{lem2}$ as claimed, performing the subsequent trade illustrated in the diagram causes $L$ and $P$ to agree at the cell $(r_1, c_1)$.  Therefore, it suffices to show how we can use Lemma $\ref{lem2}$ as illustrated above. 

 First, we note that in our applications of Lemma $\ref{lem2}$ we will avoid the symbols $\{s_1, s_2\}$ to prevent conflicts.  Next, notice that picking the cell $(r_4, c_4)$ that contains $s_1$ determines the rows $r_2, r_4$ and the columns $c_2, c_4$. Because of this, we want to choose this cell such that the following properties hold.
\begin{itemize}
\item The cells $(r_4, c_4), (r_4, c_2), (r_2, c_4)$ are not ones at which $P$ and $L$ currently agree.  This eliminates at most $3\epsilon n$ choices
\item The four cells $(r_4, c_1), (r_2, c_3), (r_3, c_2), (r_1, c_4)$ are all valid choices for the first cell to be exchanged in an application of Lemma $\ref{lem2}$. Via Lemma $\ref{lem2}$, this eliminates at most $4\left( 2\frac{k + \frac{48}{n^2}}{d}n + \epsilon n + 2\right)$ choices.  (The $+ \frac{48}{n^2}$ comes from the fact that we are applying Lemma $\ref{lem2}$ four consecutive times, and therefore on the fourth application of our lemma our Latin square $L$ may contain up to $kn^2 + 48$ disturbed cells.)  Notice that Lemma $\ref{lem2}$ insures that these cells are not ones at which $P$ and $L$ agree.
\end{itemize}
This leaves us with 
\begin{align*}
n - 8\frac{k+ \frac{48}{n^2}}{d} n - 7\epsilon n - 8
\end{align*}
choices for this cell.

Now, choose the symbol $s_3$ such that the following property holds.
\begin{itemize}
\item The cells containing $s_3$ in row $r_1$ and column $c_3$ are both valid choices for the second cell to be exchanged in an application of Lemma $\ref{lem2}$.  By Lemma $\ref{lem2}$, this eliminates at most $2\left( 4\frac{k+ \frac{48}{n^2}}{d}n + 2d n + 3 \epsilon n + 3\right)$ choices.  Note that in this calculation we have already ensured that these cells are not ones at which $P$ and $L$ agree.
\end{itemize}
This leaves us with 
\begin{align*}
n - 8\frac{k+ \frac{48}{n^2}}{d} n - 4d n - 6\epsilon n - 6
\end{align*}
choices of this symbol. By an identical chain of reasoning, we have precisely the same number of choices for $s_4$.

 Therefore, if we can make the above pair of choices and additionally satisfy the bounds
\begin{align*}
9 &\leq n - 4\frac{k+ \frac{48}{n^2}}{d}n - 6d n  - 6 \epsilon n ,  \\
20 &\leq n - 12d n -  12\epsilon n \\
\end{align*}
required by Lemma $\ref{lem2}$, we can find the requested trades.  We do so one by one, performing one of the $s_3$ trades, then the other, then the corresponding $s_2$-$s_3$ $2 \times 2$ trade created by these two squares, then one of the $s_4$ trades, then the other, and finally the corresponding $s_2$-$s_4$ $2\times 2$ trade.  Because these Lemma $\ref{lem2}$ applications were restricted to not use the symbols $\{s_1, s_2\}$, none of these trades disturb the work done by previous trades, or the $s_1, s_2$ cells in our trade.  Therefore, after performing these trades, we can finally perform the $s_1$-$s_2$ $2 \times 2$ trade created by all of our work, and get the symbol $s_2$ in $(r_1, c_1)$.

Because each application of Lemma $\ref{lem2}$ disturbs the contents of at most 16 cells, and our final trade disturbs 5 other cells apart from $(r_1, c_1)$, we have constructed a trade that disturbs no more than 69 cells other than $(r_1, c_1)$ whenever we can satisfy these five inequalities.

Using basic calculus, it is not too difficult to see that the best choice of $d$ for maximizing the values of $\epsilon, k$ available to us is roughly $\sqrt{k}$.  Therefore, if we let $d = \sqrt{k}$, we can (by comparing our five inequalities) reduce the number of bounds we need to consider down to just one:\ specifically,
\begin{align*}
20 &\leq n - 12 n\sqrt{k} -  12\epsilon n.\\
\end{align*}
\end{proof}

With these lemmas established, we can now prove the central claim of this paper.
\begin{thm}\label{thm1}
Any $ \epsilon$-dense partial Latin square $P$ containing no more than $\delta n^2$ filled cells in total is completable, for $\epsilon < \frac{1}{12}, \delta <\frac{(1 -12\epsilon)^2  }{10409}$.
\end{thm}
\begin{proof}
Take any $ \epsilon$-dense partial Latin square $P$ with no more than $\delta n^2$-many filled cells, and let $L$ be a Latin square of the same dimension as $P$ as generated by Lemma $\ref{lem1}$.  Cell by cell, select a filled cell $(r_1, c_1)$ of $P$ at which $P(r_1, c_1) \neq L(r_1, c_1)$, and apply Lemma $\ref{lem3}$ to find a trade that disturbs at most 69 other cells and that causes $L$ and $P$ to agree at this cell.  Again, if we can always apply this lemma, iterating this process will convert $L$ into a completion of $P$.

We start with a square in which at most $3n+ 7$ cells were disturbed, and proceed to disturb $69 \delta n^2$ more cells via our repeated applications of Lemma $\ref{lem3}$.  If we want this to be possible, we merely need to choose $n, \delta, \epsilon$ such that the inequality
\begin{align*}
20 &\leq n - 12 \sqrt{69\delta n^2 + 3n + 7} -  12\epsilon n\\
\end{align*}
holds.  For any $\epsilon < \frac{1}{12}$, we can choose any value of
\begin{align*}
 \delta <\frac{(1 -12\epsilon)^2 - \frac{40}{n}(1-12\epsilon) + \frac{400}{n^2} - \frac{432}{n} - \frac{1008}{n^2} }{9936},
\end{align*} 
and our inequality will hold.  For $ \delta < \frac{1}{n}$, our Latin square contains $\delta n^2 < n$ symbols, and is therefore completable via a result of Smetianuk.  Otherwise, solve the above inequality for $(1-12\epsilon)^2$ to get
\begin{align*}
  9936 \cdot  \delta  + \frac{40}{n}(1-12\epsilon) + \frac{432}{n} + \frac{608}{n^2}  < (1 -12\epsilon)^2.\\
\end{align*}
Now, if we use our observation that $\delta \geq \frac{1}{n}$, and also the observation that this theorem will only give nontrivial results for values of $n > 10^4$,  we can simplify this to the slightly weaker but much more compact inequality
\begin{align*}
 \delta <\frac{(1 -12\epsilon)^2 }{10409}.
\end{align*} 
\end{proof}
\begin{cor}\label{plscor1}
Any $\frac{1}{13}$-dense partial Latin square containing no more than $5.7 \cdot 10^{-7}$ filled cells is completable.
\end{cor}
\begin{cor}\label{plscor2}
All $ 9.8 \cdot 10^{-5}$-dense partial Latin squares are completable.
\end{cor}
\begin{cor}\label{plscor3}
All $10^{-4}$-dense partial Latin squares are completable, for $n > 1.2 \cdot 10^5.$
\end{cor}

\begin{proof}
The first corollary is immediate from Theorem $\ref{thm1}$.  For the second and third:\ if we set $\epsilon = \delta$, we are simply dealing with a $ \epsilon$-dense partial Latin square.  Whenever $\delta = \epsilon < \frac{2}{n}$, a $ \epsilon$-dense partial Latin square is a Latin square with no more than 1 entry in any row, column, and symbol, and is completable (either exactly one entry is used in every row, column, and symbol, in which case this is a transversal and is easily completable; otherwise, we have a Latin square with less filled cells than its order, which we know is completable due to a result of Smetianuk.)  Otherwise, if $\delta = \epsilon > \frac{2}{n}$, we can use the first, longer inequality in Theorem $\ref{thm1}$ to see the other two inequalities.
\end{proof}

\section{Future Directions}

The most obvious direction for future study is determining whether this value of epsilon can be improved, ideally to the conjectured $\frac{1}{4}$-bound.  

There are other viable areas of study, however.  In particular, the techniques used in Lemmas $\ref{lem2}$ and $\ref{lem3}$ are (in theory) applicable to studying a number of other classes of partial Latin squares.  For example, consider classes of partial Latin squares where some small handful of rows/columns/symbols are allowed to be used rather often, but the others are left blank -- i.e. a generalization of a Ryser-type result.  Or perhaps consider squares that contain no more than $3n$ symbols in total, but we have some other conditions at hand --- i.e.\ some sort of generalization of a Smetianuk-type result.  These techniques should be able to create completions for partial Latin squares of any of these types with some slight adaptation. 

Finally, it might be worthwhile to develop trades that use structure other than the $2 \times 2$ trades our proofs are based on.  Basically:\ the algorithm we have described in Theorem $\ref{thm1}$ stops working as soon as our Latin square $L$ runs out of $2 \times 2$ structure.  However, this is not the only kind of structure a square can have.  $2 \times 3$ trades, $2 \times k$ trades, $3 \times 3$ trades, and other such small things are kinds of structure that our square may still have when it runs out of $2 \times 2$ trades.  If we could come up with an argument that would make the trades generated in Lemmas $\ref{lem2}$ and $\ref{lem3}$ ``robust'' --- i.e.\ able to use \textbf{any} kind of small trades to build themselves up, not just $2 \times 2$'s --- we might hope to see marked improvements in our values of $\epsilon, \delta$.

\chapter{NP-Completeness and $\epsilon$-Dense Partial Latin Squares}

In the introduction to Chapter 2, we briefly discussed a handful of historical results on completing Latin squares, along with their proof methods.  We did this mostly to contrast classical methods of completing Latin squares (induction) with the Chetwynd/H\"aggkvist tool of $2 \times 2$ trades, and our own extensions of their idea to $2 \times 2$ improper trades.  In this section, we will focus on a key aspect in which all of these methods are the same:\ they are all \textbf{algorithmic} in nature.  In particular, if closely inspected, Theorem $\ref{thm1}$ doesn't just prove that certain families of $ \epsilon$-dense partial Latin squares are completable:\ it also provides an algorithm for constructing such a completion.

From a complexity theory standpoint, this raises an obvious question:\ how ``quick'' is this algorithm?  Furthermore, given a fixed $\epsilon$, how quick can \textbf{any} algorithm hope to be, if it completes arbitrary $\epsilon$-dense Latin squares?

We start this chapter by reviewing a few basic defintions in complexity theory\footnote{Readers looking for a more in-depth discussion of the ideas at hand are encouraged to read Garey and Johnson's classic text \cite{Garey_Johnson_1979} on the subject.}.   From there, we will show that the algorithm given by Theorem $\ref{thm1}$ runs in polynomial time.  We will then contrast this with a famous result of Colbourn, which states that completing an arbitrary partial Latin square is an $NP$-complete problem.  This pair of results suggests the following line of questioning:\ for what values of $\epsilon$ is completing an arbitrary $\epsilon$ partial Latin square an NP-complete problem?  We conjecture that the task of completing such squares is NP-complete for any $\epsilon > \frac{1}{4}$, and prove that it is NP-complete for any $\epsilon > \frac{1}{2}$.

\section{Basic Definitions}
\begin{defn}
In the following discussion, a \textbf{problem} is some sort of general question that we want to find a yes or no answer to, along with some sort of list of associated parameters that (when specified) give a specific \textbf{instance} of this problem.  For example, consider the \textbf{traveling salesman problem}, which asks (given a list of cities $C =  \{c_0, \ldots c_n\}$ and a distance cap $B$) the following question:\ starting from $c_0$, is it possible to visit each city exactly once and return to $c_0$ in such a way that the total distance traveled is less than $B$?  Stated in this way, the parameters of this problem are a list of cities $C = \{c_0, \ldots c_n\}$, a distance function $d:\ C^2 \to \mathbb{Z}^+,$ and a bound $B$. 

We say that the \textbf{input length} of a given problem is the number of parameters needed to specify a given instance of this problem.  Note that this process can vary wildly depending on how the inputs to our problem are described.  For example, the input length of the traveling salesman problem as written above is $n+3$; however, if we were to write its distance function as a collection of $m^2$ distinct labeled integers, instead of as a single function, we would regard it as a different problem with input length $(n+1)^2 + n + 2$.
\end{defn}
\begin{defn}\label{tspref}
Given a problem, we often want to find an \textbf{algorithm} --- i.e.\ a a step-by-step procedure --- that will take in a given instance of a problem and output a solution to our problem in this given instance.  The precise notion of what constitutes a ``step'' is typically context-dependent, but roughly denotes a task that can be completed by a computer or some other device in some constant or fixed amount of time.  For example, the step ``find the shortest tour of a set of $m$ cities'' is not a step we would want to include in an algorithm that solves the traveling salesman problem, as there is no known fixed-time way to perform such a task.  Conversely, the following does denote an algorithm for solving the traveling salesman problem.
\begin{enumerate}
\item One-by-one, select a permutation $\pi$ of $\{1, \ldots n\}$.
\item For each permutation, find the total distance $d(c_0, c_{\pi(1)}) + \sum\limits_{i=2}^n  d(c_{\pi(i-1)}, c_{\pi(i)}) + d(c_{\pi(n)}, c_0)$.
\item If this distance is less than $B$, return ``yes.''  Otherwise, if there is another permutation of $\{1\ldots n\}$ remaining, return to 1.  Otherwise, if we are out of permutations, return ``no.''
\end{enumerate}
The individual steps in this algorithm are the pairwise summations, $n+1$ of which occur in line 2; the comparison to $B$, which occurs in line 3; and the selection of a permutation of $\{1, \ldots n\}$, which takes a single step for each choice if we select them one-by-one in some predetermined order.  

For the most part, we will assume common-sense guidelines for what does or does not constitute a step, as the constants involved in our problems are largely irrelevant for our purposes.  Formally, we assume that the model of computation we are working with is a deterministic Turing machine in all of our discussions.
\end{defn}
\begin{defn}
If we have an algorithm, we will often want to know how ``efficient'' this algorithm is; in other words, if it is faster or takes up less space than another algorithm that purports to solve the same problem.  To make this concept concrete, we introduce the notion of \textbf{complexity}.  Given an algorithm, we define the \textbf{time complexity function} for this algorithm as a map that takes in a possible input length $n$, and outputs the largest amount of steps that the algorithm will need to solve a problem instance with input size $n$.  We typically only care about the asymptotics of this time complexity function, and describe a given algorithm as having complexity $O(n^2)$ if its time complexity function is $O(n^2)$, rather than worrying about the precise constants involved.  

For example, consider the algorithm offered in Defintion $\ref{tspref}$.  In the worst-case scenario, where there is no path for which the total distance is $< B$, our algorithm will take $O(n!)$ many steps to find that this is the case, as it will have to check all possible permutations of $\{1, \ldots n\}$ to determine that this is the case.

We say that a given algorithm is a \textbf{polynomial time algorithm} if its time complexity function is $O(p(n))$, for some polynomial function $p(n)$.  We say that a given problem lies in the class P if there is a polynomial time algorithm that solves our problem.
\end{defn}
\begin{defn}
For many classes of problems, a ``yes'' answer to a given instance can often be accompanied by a ``proof'' that consists of a solution to our problem for that specific instance.  For example, if we were to assert that the answer to a given instance of the traveling salesman problem was yes, we could accompany this claim with a given tour of the $n$ cities that takes less than $B$ units of distance to traverse.  These proofs sometimes appear to be easier to verify than to solve. Again, if we consider the traveling salesman problem, we can evaluate in $O(n)$ steps whether a given tour offers a solution to a given instance, while the algorithm in Defintion $\ref{tspref}$ takes $O(n!)$ many steps to create any such tour.

Given an instance $I$ of some problem and some sort of structure $S$, consider an algorithm that takes in $I$ and $S$ and does one of the two following things.
\begin{itemize}
\item If $S$ in fact proves that the problem instance given by $I$ has ``yes'' as an answer, it checks this proof and returns ``yes.''
\item Otherwise, if $S$ is not a proof that the problem instance given by $I$ has ``yes'' as an answer, it notes that $S$ fails as a proof of $I$, and returns ``no.''
\end{itemize}
Call such an algorithm a \textbf{proof verifier} for a problem.  We say that a problem is in the class NP if it has a polynomial-time proof verifier.  Intuitively, problems in the class NP are problems whose answers are ``easy'' to verify.  It bears noting that any problem in P is in NP, as we can simply use the P problem's polynomial-time solver to decide whether or not any given instance is true in polynomial time, and thus skip the entire ``checking'' process.
\end{defn}
\begin{defn}
Let $\Pi_1$, $\Pi_2$ be a pair of problems, and $D_{\Pi_1}, D_{\Pi_2}$ be the respective collections of all instances of these problems.  We say that the problem $\Pi_1$ is \textbf{polynomially reducible} to the problem $\Pi_2$ if there is a function $f$ with the following properties.
\begin{itemize}
\item $f$ is computable by some polynomial-time algorithm.
\item For any instance $I \in D_{\Pi_1}$, $I$'s instance evaluates to ``yes'' in $\Pi_1$ if and only if $f(I)$'s instance evaluates to ``yes'' in $\Pi_2$.
\end{itemize}
Intuitively, we think of this as saying that an algorithm for solving $\Pi_2$ can be used to solve $\Pi_1$ without much of a loss in efficiency.  This is because to evaluate an instance $I$, we can just apply $f(I)$ and use our $\Pi_2$-algorithm on $f(I)$, with the only expense incurred in this transformation being the work done in applying $f$.

Any two problems in P are trivially polynomially reducible to each other.  We say that a problem $\Pi$ is \textbf{NP-complete} if $\Pi$ is in NP and every other problem in NP can be polynomially reduced to $\Pi$.  Equivalently, a problem is NP-complete if another NP-complete problem can be polynomially reduced to it.

There are many known NP-complete problems.  We list a small handful here.
\begin{itemize}
\item The traveling salesman problem, as described earlier, is an NP-complete problem (Karp \cite{Karp_1972}.)
\item The \textbf{Boolean satisfiability problem}, or \textbf{SAT}, is the following task:\ take an arbitrary Boolean formula of length n.  Is there some assignment of true and false to the variables of this formula so that the entire expression evaluates to true?  This task is NP-complete (Karp \cite{Karp_1972}.)
\item Triangulating an arbitrary tripartite graph is NP-complete (Holyer \cite{Holyer_1981}.)
\end{itemize}
\end{defn}

Determining whether the class of P problems is the same as the class of NP problems is one of the most famous open problems in mathematics.  Accordingly, determining whether various problems are in P or NP, and whether various problems in NP are complete, is a classic genre of questions in complexity theory, and perhaps acts as motivation for our own proofs in this chapter.

\section{Completing $\epsilon$-Dense Squares in Polynomial Time}
We start by proving that the algorithm given in Theorem $\ref{thm1}$ runs in polynomial time; in other words, that the task of completing certain families of $ \epsilon$-dense partial Latin squares is in P.
\begin{thm}\label{polytimethm}
Any $\epsilon$-dense partial Latin square $P$ containing no more than $\delta n^2$ filled cells in total is completable, for $\epsilon < \frac{1}{12}, \delta <\frac{(1 -12\epsilon)^2  }{10409}$.  Furthermore, the algorithm used to create this completion $L$ needs no more than $O(n^3)$ steps to construct $L$.
\end{thm}
\begin{proof}
 An instance of our problem consists of an $n \times n$ partial Latin square $P$ containing less than or equal to $\delta n^2$ filled cells; consequently, we can regard our input list as $O(n^2)$ many triples $(r,c,s)$.  Given any such instance, consider the following functions we used to create a completion of $P$:

\begin{samepage}
\begin{center} \textbf{Algorithm for Lemma \ref{lem1}} \end{center}
\begin{enumerate}
\item[Input:]  A natural number $n$.
\item If $n = 2k$, populate a $2k \times 2k$ array as directed by Lemma $\ref{lem1}$.
\item Otherwise, if $n = 2k+1$ is odd, populate an $2k \times 2k$ array as directed by Lemma $\ref{lem1}$.  Construct a transversal as directed, and use it as indicated to create the desired $n \times n$ Latin square. 
\item[Runtime:] This clearly takes at most $O(n^2)$ steps to complete; we need $O(n^2)$ steps to populate any grid, and at most $O(n)$ steps to construct and use a transversal to augment our grid in the odd case.
\end{enumerate}
\end{samepage}

\begin{center} \textbf{Algorithm for Lemma \ref{lem2}} \end{center}
\begin{enumerate}
\item[Input:] A partial Latin square $P$, a Latin square $L$, a list of $kn^2$ cells in $L$ that are labeled as ``disturbed by earlier trades,'' a list of rows, columns, and symbols that are $d$-overloaded, constants $k, n, d, \epsilon,$, symbols $\{s_1, \ldots s_a\}$, a row $r_1$, and columns $c_1, c_2$.
\item Determine whether the symbols in the cells $(r_1, c_1)$ and $(r_1, c_2)$ are either both still in the correct quadrants, if exactly one is out of place, or it both are out of place.
\item If neither or both are out of place, choose the row $r_2$ as directed by the lemma.  Specifically, start from the first row in the opposite half from the rows $r_3, r_4$ determined by our given $(r_1, c_1)$ and $(r_1, c_2)$.  For each row, check whether it satisfies the criteria asked for in the lemma:\ there are a constant number of checks that need to be performed, for each possible choice of row.  Proceed until a satisfactory row is found.
\item Otherwise, choose $r_2, s_6$ as directed by the lemma, again starting from the first available choices and proceeding until satisfactory choices are found.  Again, to check whether any given $r_2$ or $s_6$ is satisfactory only involves performing a constant number of checks.
\item[Runtime:]  This takes $O(n)$ many steps to run:\ determining whether $(r_1, c_1)$ and $(r_1, c_2)$' symbols are still in the correct quadrants takes a constant amount of time, and making our choices of $r_2, s_6$ takes $O(n)$ steps to complete, as there are $n$ choices for each.
\end{enumerate}

\begin{center} \textbf{Algorithm for Lemma \ref{lem3}} \end{center}
\begin{enumerate}
\item[Input:] A partial Latin square $P$, a Latin square $L$, a list of $kn^2$ cells in $L$ that are labeled as ``disturbed by earlier trades,'' a list of rows, columns, and symbols that are $\sqrt{k}$-overloaded, constants $n, \epsilon,$ and a cell $(r_1, c_1)$ at which $P$ and $L$ disagree.
\item As discussed in the lemma, choose the $(r_4, c_4)$ cell containing $s_1$ so that all of the conditions requested by the lemma are satisfied.  Given any choice of this cell, there are $O(n)$-many checks that need to be made to insure that all of the consequently-determined cells satisfy the properties requested of them. 
\item From here, determine the symbols $s_3, s_4$ as requested by the lemma.  Again, note that it takes a constant number of checks to know whether a given cell determined by a choice of $s_3$ satisfies the given conditions.
\item Apply Lemma $\ref{lem2}$ as directed four times; this takes $O(n)$ steps, as discussed.  Then apply the resulting final trade.
\item[Runtime:]  Again, this takes at most $O(n)$ steps to complete, as at each stage we are either ranging over $n$ objects and making a constant number of checks for each object, or simply applying Lemma $\ref{lem2}$ four times, which we already know takes $O(n)$ steps.
\end{enumerate}

\begin{center} \textbf{Algorithm for Theorem \ref{thm1}} \end{center}
\begin{enumerate}
\item[Input:] A partial Latin square $P$, constants $\epsilon$ and  $\delta$.
\item Use Lemma $\ref{lem1}$ to construct the square $L$.  Also create a grid (currently empty) that will contain markers for the $kn^2$ disturbed cells in our grid.  Finally, for each row, column, and symbol, associate a tally to count the number of disturbed cells in each object; this will be used to keep track of whether an object is $d$-overloaded.
\item Pick a cell in $P$.  If $P$ and $L$ disagree there, run Lemma $\ref{lem3}$.  
\item After doing this, update the grid and tallies of disturbed cells.  This takes a constant amount of time.  
\item If we have not yet looked at every cell in $P$, return to step 2.
\item[Runtime:]  This takes $O(n^3)$ many steps to run.  This is because Lemma $\ref{lem3}$ takes $O(n)$ steps to run at each instance, we have to run Lemma $\ref{lem3}$ at most $O(n^2)$ times, and the Lemma $\ref{lem1}$/updating steps only take $O(n^2)$ steps in aggregate over the entire run of the theorem.
\end{enumerate}

We have thus proven that Theorem $\ref{thm1}$ runs in $O(n^3)$ time, as claimed.  
\end{proof}

It bears noting that this $O(n^3)$ is roughly the best runtime we could hope any algorithm attains.  In particular, consider any algorithm that fills a partial Latin square cell-by-cell.  If such an algorithm is applied to an $\epsilon$-dense partial Latin square, it will have at least $(1-\epsilon)n^2$ cells to fill in.  For each cell, it will have to make at least $O(n)$ checks just to insure that the square remain Latin; consequently, this algorithm will have to make $O(n^3)$ checks in total, no matter what its implementation is.  Any method that hopes to best this, then, would have to somehow fill in the target partial Latin square with whole chunks at a time --- i.e.\ filling in whole rows or columns at once --- and moreover do this in a remarkably fast way (i.e.\ if it was proceeding row-by-row, it would have to find each row in $o(n^2)$ time.)  This seems very improbable.

\section{Colbourn's Theorem on NP-Completeness and Completing Partial Latin Squares}

The situation for general partial Latin squares is markedly different; as Colbourn \cite{Colbourn_1984} showed in 1984, the problem of completing an arbitrary partial Latin square is NP-complete.  Our focus for the remainder of this chapter will be on strengthening Colbourn's result to $\frac{1}{2}$-dense partial Latin squares.  Because our strengthening of his result will involve closely working with his proof methods, we will provide an overview of his proof.  First, we note a few key definitions.

\begin{defn}
Given a partial Latin square $L$, recall from our discussion of Conjecture $\ref{nashwilliamstri}$ that there is a natural way to visualize this partial Latin square as the triangulation of some tripartite graph $G$ with vertex set $(R,C,S)$.  The \textbf{defect} of a partial Latin square is simply the tripartite graph arising from the tripartite complement of this corresponding $G$.
\end{defn}
\begin{defn}
Take a tripartite graph $G = (R, C, S)$, with $|R|=r, |C| = c, |S| = s$.  A \textbf{Latin framework} for such a tripartite graph $G$, denoted $LF(G;r,c,s)$, is an $r \times c$ array, where each entry is either empty or filled with a symbol from the set $\{1, \ldots s\}$, that satisfies the following properties.
\begin{itemize}
\item If $G$ contains the edge $(r_i, c_j)$, the cell $(i,j)$ in our $LF(G;r,c,s)$ is empty.  Otherwise it is filled with a symbol from $\{1 \ldots s\}$.
\item If $G$ contains the edge $(r_i, s_k)$, then row $i$ of our $LF(G;r,c,s)$ does not contain symbol $k$.
\item If $G$ contains the edge $(c_j, s_k)$, then the column $j$ of our $LF(G;r,c,s)$ does not contain symbol $k$.
\end{itemize}

Note that if $r = c = s$, then $G$ is precisely the defect of $LF(G;r,r,r)$.  In fact, $LF(G;r,r,r)$ is a partial Latin square, and any completion of this partial Latin square corresponds to a triangulation of $G$.
\end{defn}

\begin{thm}
[Colbourn, 1984]   The task of completing an arbitrary partial Latin square is NP-complete.
\end{thm}
\begin{proof}
Given a partial Latin square $P$ and a claimed completion $L$ of $P$, checking whether $L$ is in fact a completion of $P$ can trivially be done in polynomial time --- just examine all of the entries of $P$ to see if $P$ and $L$ agree there, and then check all of $L$'s rows/columns to see if the Latin property is preserved.  Therefore, membership in NP is immediate; so it suffices to reduce the task of completing an arbitrary partial Latin square to another NP-complete problem.  

Consider the task of triangulating an arbitrary uniform\footnote{A tripartite graph $G = (V_1, V_2, V_3)$ is called \textbf{uniform} if for any $v \in V_i, \deg_{i+1}(v) = \deg_{i-1}(v)$.} tripartite graph $G$ with tripartition $(R, C, S)$, $|R| = |C| = |S| = n$.  Colbourn starts by showing that the task of completing any such graph is NP-complete, strengthening the earlier-mentioned result \cite{Holyer_1981} of Holyer.

From here, Colbourn reduces the above task to the problem of completing an arbitrary partial Latin square.  This is done in three stages.
\begin{enumerate}
\item Take any uniform tripartite graph $G = (R,C,S)$.  Construct a $LF(G;n,n,2n)$ as follows:\ if the edge $(r_i, c_j)$ exists in $G$, leave the cell $(i,j)$ blank.  Otherwise, fill this cell with the symbol $1 + n + ( (i+j) \mod n)$. 
\item Column by column, extend this $LF(G;n,n,2n)$ to a $LF(G;n,2n,2n)$.  The proof methods used here are analogous to those used in Ryser's theorem \cite{Ryser_1951}, and involve constructing each new column $c_i$ by picking appropriate systems of distinct representatives.  The main feature that we care about is that this process completely fills each column, as required by the definition of a Latin framework, and that it can be found in polynomial time.
\item Row by row, extend this $LF(G;n,2n,2n)$ to a $LG(G;2n,2n,2n)$ in precisely the same fashion.  
\end{enumerate}
As noted before, the resulting Latin framework is a partial Latin square of order $2n$, such that any completion of this partial Latin square corresponds to a triangulation of the tripartite graph $G$.  Therefore, we have reduced the task of completing a uniform tripartite graph to that of completing a partial Latin square.  Because the first task is NP-complete, as noted earlier, we know that the second is as well.
\end{proof}

The partial Latin squares resulting from Colbourn's construction have their last $n$ rows, columns, and symbols each used $2n$ times in the resulting construction.  Correspondingly, we can regard Colbourn's proof as the statement that completing an arbitrary $1$-dense partial Latin square is NP-complete, and note that his proof does not immediately extend to $\epsilon$-dense partial Latin squares for any $\epsilon < 1$.

However, with some work, we can extend his theorem to the following result.
\begin{thm}\label{npcompletepls}
The task of completing an arbitrary $\epsilon$-dense partial Latin square is NP-complete, for any $\epsilon > \frac{1}{2}$.
\end{thm}
\begin{proof}
Take any uniform tripartite graph $G$ with vertex set $(R,C,S)$, $|R| = |C| = |S| = n$.  Augment this graph by adding in $n^3 - n$ vertices to each partition $R, C, S$, all with degree 0.  Apply Colbourn's construction to this new graph:\ this yields a $2n^3 \times 2n^3$ partial Latin square $P$, any completion of which corresponds to a triangulation of $G$.  As before, filling in the entries in the $n \times n$ subarray in the upper-left-hand corner of this partial Latin square with entries from $\{1, \ldots n\}$ corresponds to triangulating $G$.  For convenience's sake, call this subarray $Q$.

We now seek to find a clever way of setting most of the cells in our square blank, in such a way that any resulting completion of this new partial Latin square will still correspond to a triangulation of $G$.  

To do this, for every row $r_i$ in our square $P$, let $X_i$ denote the collection of all symbols that do not occur in this given row.  Note that the construction given by Colbourn ensures that every $X_i$ contains at most $n$ symbols.  Similarly, for every column $c_j$ in $P$, let $Y_j$ denote the collection of all symbols that do not occur in that given column.  Finally, let $Z$ denote the collection of all symbols that occur in the cells spanned by $Q$.

Define $A^c$ to be the union of all of the sets $X_i, Y_j,$ as well as the sets $Z$ and $Q.$  Let $A$ be the complement of this set.  Using the bounds described above, we can see that $A$ has cardinality at least $2n^3 - 3n^2 - n$.  Arbitrarily divide $A$ into two sets $A_1, A_2$, each with size at least $\dfrac{2n^3 - 3n^2 - n}{2}$.

Take $P$, and delete all occurrences of a few symbols from $A_1$ from the columns $c_1, \ldots c_n$.  Consider any possible completion of this new square $P'$.  Is it possible that there is some new completion of $P'$ that results in a new triangulation of $G$?

It is certainly possible that there are new completions.  Perhaps we deleted the contents of cells corresponding to a $2 \times 2$ trade in $P$; when we go to complete $P'$, we suddenly have choices for how we will fill in this $2 \times 2$.  However, suppose we are only concerned with the cells that correspond to a triangulation of $G$:\ i.e.\ the cells in $Q$.  If we are concerned about insuring that any completion of $P'$ still corresponds to a triangulation of $G$, we just need to insure that the only possible symbols that can go in these cells are symbols from $\{1, \ldots n\}$.  

This property definitely exists before we delete any cells; by construction, the only symbols that potentially haven't been used $2n^3$-many times in $P$ are those in $\{1, \ldots n\}$.  However, when we delete these cells from the columns $c_1, \ldots c_n$, what happens?  Well:\ from the perspective of our columns, these $A_1$ symbols are now potentially usable in any completion of $Q$.  However, from the perspective of the rows $r_1, \ldots r_n$, these $A_1$ symbols are still not valid symbols to be used in a completion of $Q$, because each of these symbols occurs in each of the rows $r_1, \ldots r_n$.  In other words, deleting these symbols has no effect on the potential triangulations of $G$!

Delete all occurrences of the symbols in $A_1$ from the rows $r_1, \ldots r_n$, and all occurrences of the symbols in $A_2$ from the columns $c_1, \ldots c_n$.  By the logic established above, any completion of the resulting partial Latin square $P'$ will still correspond to a triangulation of $G$.  Moreover, delete the contents of any cell $(r_i, c_j)$ with $i, j > n$.  This has no effect on the the potential completions of $Q$, as these all still need to come from the set of symbols $\{1, \ldots n\}$ and therefore still correspond to triangulations of $G$.

The resulting partial Latin square $P$ has at most $\dfrac{2n^3 + 3n^2 + n}{2}$ symbols in its first $n$ rows and columns, and uses any symbol at most $2n$ times, at most once in each of these rows and columns.  Therefore, for any $\epsilon > \frac{1}{2}$, there are sufficiently large values of $n$ for which these squares are always $\epsilon$-dense.   The construction of the set $A$ and the consequent deletion of elements are all steps that occur in polynomial time; therefore, we have constructed a polynomial reduction from triangulating a uniform tripartite graph to completing an $\epsilon$-dense partial Latin square, for any $\epsilon > \frac{1}{2}$.

Therefore, completing an $\epsilon$-dense partial Latin square is an NP-complete task, for any $\epsilon > \frac{1}{2}$.
\end{proof}
\section{Future Directions}
Motivated by the results of this chapter and the Nash-Williams conjecture discussed in Chapter 1, we offer the following conjecture.
\begin{conj}
The task of completing an arbitrary $\epsilon$-dense partial Latin square is NP-complete, for any $\epsilon > \frac{1}{4}$. Conversely, if $\epsilon \leq \frac{1}{4}$, the task of completing an arbitrary $\epsilon$-dense partial Latin square is in P.
\end{conj}  

In other words, we conjecture that this task becomes NP-complete precisely when it is conjectured that completions potentially do not exist.  

It does not seem likely that the proof methods being currently used in Theorem $\ref{npcompletepls}$ can be immediately brought down to $\frac{1}{4}$.  However, one important thing to note about the algorithm used in this theorem is that it does have some room for error that could be used to improve its results.  Specifically, while several of the rows and columns in the construction created have up to half of their cells filled, none of the symbols used have more than $2n$ occurrences over the entire partial Latin square, a far cry from the $\frac{2n^3}{2}$ that we might expect.  It is plausible, though the construction is not necessarily obvious, that there is some way to ``add in'' a number of symbols in otherwise-empty rows and columns in a clever way that would ``block'' more potential symbols from occurring in $Q$.  This in turn might let us reduce the number of filled cells in the first $n$ rows and columns of $P$, and allow us to decrease our overall value of $\epsilon$.

\chapter{Probabilistic Approaches to Completing $\epsilon$-Dense Partial Latin Squares}

In the previous chapter, we established that not only is Theorem $\ref{thm1}$ constructive in nature, it is relatively efficient in terms of its runtime:\ it runs in polynomial time (in fact $O(n^3)$,) which is about as efficient as we could hope for.  In attempts to improve Theorem $\ref{thm1}$'s bounds, it is reasonable to wonder if perhaps this efficiency comes at the expense of a better range of values of $\epsilon$ --- i.e.\ that perhaps if we want to strengthen our bounds, we should do so by applying methods that are perhaps nondeterministic in nature!

This idea is what we explore in this chapter.  In particular, we show how some careful probabilistic augmentations to the algorithms in Theorem $\ref{thm1}$ can yield a roughly twofold improvement in $\epsilon$, for sufficiently large $n$. 

\section{A Probabilistic Improvement of Theorem $\ref{thm1}$}

Suppose we are in the setting given by Theorem $\ref{thm1}$:\ i.e.\ we have an $\epsilon$-dense partial Latin square $P$ containing no more than $\delta n^2$ filled cells.  Take a Latin square $L$ created by Lemma $\ref{lem1}$, and suppose that we are currently attempting to turn $L$ into a completion of $P$.  Notice that each application of Lemma $\ref{lem3}$ disturbs up to 69 cells in $L$ for each cell $(r_1,c_1)$ that it makes $L$ and $P$ agree at.  When $L$ has most of its structure, this is a large overestimate.  In many cases, we would hope to find cells $(r_1, c_1)$ in $P$ such that \textbf{one} $2 \times 2$ trade on $L$ will make $L$ and $P$ agree at this cell.  So:\ for a given cell in $P$, what possible obstructions could there be to a $2\times2$ trade existing on $L$ that causes $P$ and $L$ to agree at this cell?

First, it is possible that none of $L$'s $2 \times 2$ trades can cause $L$ and $P$ to agree at this cell.  Given any cell of $P$, there are at most  $\lceil n/2 \rceil + 2$ many choices of symbol for this cell for which no such $2 \times 2$ trades exist.

Second:\ even if these $2 \times 2$ trades exist, it is possible that they ``conflict'' with each other.  In other words, it is possible that some of our $2 \times 2$ trades share cells in common.  In how many ways can this happen?  Well, suppose that $(r_1, c_1)$ is a filled cell in $P$ containing the symbol $s_1$.  If a $2 \times 2$ trade exists on $L$ that makes $L$ and $P$ agree at this cell, it is necessarily of the following form.

\begin{center}
\includegraphics[width=1in]{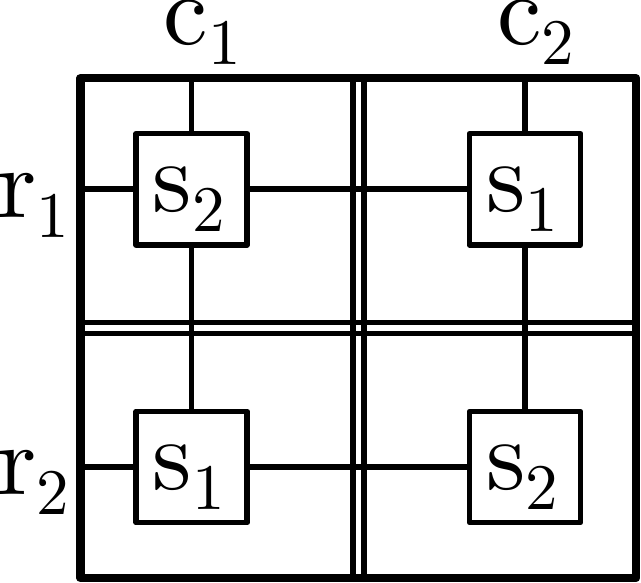}
\end{center}
Call the cell $(r_1, c_1)$ in $L$ the \textbf{overlap}-cell, the cell $(r_1, c_2)$ in $L$ the \textbf{row-dependent} cell, the cell $(r_2, c_1)$ in $L$ the \textbf{column-dependent} cell, and the cell at $(r_2, c_2)$ in $L$ the \textbf{symbol-dependent} cell.  If a cell in $L$ is used in two different $2 \times 2$ subsquares, it occurs in each of those two squares as one of these four different possible ``types'' of cells.  There are, \textit{a priori}, $\binom{4}{1} + \binom{4}{2} = 10$ possible ways in which this can happen:\ i.e. a cell can be both an overlap cell for one $2 \times 2$ trade and a row-dependent cell for another trade, or perhaps a symbol-dependent cell for two different $2 \times 2$ trades.

However, we might hope that the above difficulties only arise in pathological cases.  In particular, given a partial Latin square $P$, suppose that we select and apply random permutations to $P$'s columns, rows, and symbols.  We might hope that for the most part, ``half'' of the cells of $P$ have potential $2 \times 2$ trades available, and that ``most'' of the potential overlap conflicts identified above do not occur.  We prove this in the following lemma.

\begin{lem}\label{lem4}
Suppose that $P$ is an $\epsilon$-dense partial Latin square containing $\delta n^2$-many filled cells, and $L$ is a Latin square of the same dimension as $P$ as generated by Lemma $\ref{lem1}$.  Then we can permute the rows, columns, and symbols of $P$ in such a way to ensure that there are at least $\delta n \cdot \left(\left\lfloor\frac{n}{2}\right\rfloor - 2\right)$ cells in $P$ with associated $2 \times 2$ trades, such that at most
\begin{align*}
81 \epsilon n + \frac{39}{100} n +97\epsilon^2 n^2
\end{align*}
cells in these $2 \times 2$ trades are claimed by multiple trades.  
\end{lem}
\begin{proof}
Take an $\epsilon$-dense partial Latin square $P$, and generate a Latin square $L$ of the same dimension using Lemma $\ref{lem1}$.  
Fix an $\epsilon$-dense partial Latin square $P$ containing $\delta n^2$ many filled cells, and create a Latin square $L$ of the same dimension as $P$ using Lemma $\ref{lem1}$. Randomly\footnote{Under the uniform distribution.} choose three permutations of $\{1, \ldots n\}$, and use these three permutations to permute respectively the rows, columns, and symbols of $P$.  Notice that it does not matter in which order we apply these three permutations:\ i.e.\ first permuting the rows and then the columns of a Latin square is equivalent to instead permuting the columns first and then the rows.

We start by showing that the expected number of cells in $P$ that have corresponding $2 \times 2$ trades in $P$, possibly overlapping, is at least $\delta n \cdot \left(\left\lfloor\frac{n}{2}\right\rfloor - 2\right)$.  To do this, assume that we have already permuted the rows and columns of $P$, and are about to permute the symbols. Fix any filled cell in $P$; the probability that our random permutation of $P$'s symbols places a symbol in this cell that has a $2 \times 2$ trade is  $\geq \frac{\lfloor n/2  \rfloor -2}{n}$.  Summing this over all filled cells in $P$ gives us that the expected number of cells with corresponding $2 \times 2$ trades is at least $\delta n \cdot \left(\left\lfloor\frac{n}{2}\right\rfloor - 2\right)$.

We now calculate the expected number of these $2 \times 2$ squares that ``conflict'' with each other.  We do this by considering the ten possible ways in which a cell can be claimed by two different $2\times2$ squares; for each of these specific types of conflict, we show that the expected number of conflicts is relatively small.

We start by first observing that it is impossible for a cell in $L$ to be a row-dependent cell for two different cells in $P$.  This is because if the cell $(r_1, c_2)$  containing $s_1$ in $L$ was row-dependent for two different cells in $P$, these two cells would have to both be in the same row $r_1$ and contain the same symbol $s_1$, a contradiction.  Similarly, this argument shows that it is impossible for a cell to be column-dependent for two different cells in $P$.  As well, it is impossible by definition for a cell in $L$ to be an overlap cell for two different cells in $P$.

The rest of the cases, however, are possible.  We start with the easiest case to calculate expected values for:\ counting the expected number of cells in $L$ that are simultaneously row-dependent and overlap cells for different $2 \times 2$ trades.  To do this, assume that we have already permuted the rows and symbols in $P$.  Fix a row $r$.  In this row in $P$, there are currently no more than $\epsilon n$ filled cells by definition.  Consequently, there are no more than $\epsilon n$  overlap cells in row $r$ in $L$, and no more than $\epsilon n$ row-dependent cells in row $r$ in $L$.  Notice that permuting the columns of $P$ changes the location of these overlap cells, but does not change the location of these row-dependent cells, as the only information used to determine the row-dependent cell corresponding to a filled cell in $P$ is its row and symbol.

For any individual filled cell in row $r$ in $P$, the probability that our permutation lands it on one of these fixed cells is no more than $ \frac{\epsilon n}{n} = \epsilon$.  Therefore, the expected number of cells in row $r$ that are both overlap cells and row-dependent cells is  no more than $\epsilon n \cdot \epsilon = \epsilon^2 n$, and the expected total number of such cells in our entire Latin square is no more than $\epsilon^2 n^2$.

An identical argument counts the (column + overlap) cells.  

As well, it is not much harder to count the number of (row + column) cells.  To do this, pick any symbol $s$ in $L$, and look at the $n$ cells in $L$ that contain this symbol.  Call this set of $n$ cells the \textbf{$s$-set in $L$} for shorthand.  Notice that a cell in $P$ induces a row-dependent cell in our $s$-set if and only if it contains an $s$; similarly, it induces a column-dependent cell if and only if it contains an $s$.  Therefore, the number of row-dependent cells in this $s$-set is always no more than $\epsilon n$, as is the number of column-dependent cells.

Suppose that we have already permuted the rows and symbols of $P$, and are about to permute the columns of $P$.  As noted above, doing this does not change which cells in our $s$-set are row-dependent; however, it does change which cells are column-dependent.  In particular, a column induces a cell in our $s$-set if and only if it contains an $s$.  There are no more than $\epsilon n$ such columns.  Therefore, the probability that we place any one of these columns in one of the $\epsilon n$ locations where it will induce a cell that is both row and column-dependent is no more than $\epsilon$.  Thus, the total expected number of such collisions over our entire Latin square is no more than $\epsilon^2 n^2$.

Counting (symbol + overlap), or (symbol + symbol) cells is harder.  In particular, unlike our calculations above, there is no row, column or symbol in which we are guaranteed to have no more than $\epsilon n$ many symbol-dependent cells show up.  However, we would expect that these situations are relatively rare, and that under ``most'' permutations of $P$'s rows, columns, and symbols, these situations do not occur.  We do this as follows.  Suppose that we have already permuted the rows of $P$, and are about to permute the columns.
\begin{enumerate}
\item First, we will take any symbol $s$ in $L$, and show that the expected number of overlap cells in $L$ that contain $s$ after permuting columns is no more than $\epsilon n$.  Therefore, because the only way that one of these cells is symbol-dependent is if a corresponding cell containing $s$ in $L$ is an overlap cell, the expected number of symbol-dependent cells corresponding to symbol $s$ is also no more than$\epsilon n$.
\item From there, we will then calculate the variance of this expected value, and show that it is no more than$\epsilon n$.
\item Finally, we will use this information in a similar argument to the ones given to enumerate (row + overlap) cells to enumerate the expected number of (symbol + overlap) and (symbol + symbol) cells.
\end{enumerate} 

Pick any symbol $s$ in $L$, and look at the corresponding $s$-set in $L$.  Fix any row $r$ of $P$.  When we permute the columns of $P$, the probability that one of the filled cells in $P$ lands on the one cell containing $s$ in row $r$ of $L$ is no more than $\epsilon$.  Therefore, if we sum these probabilities over all $n$ rows of $L$, the expected number of overlap cells in our $s$-set is no more than $\epsilon n$, and thus the expected number of symbol-dependent cells in our $s$-set is also no more than $\epsilon n$.

We now calculate the variance of the number of overlap cells in our $s$-set.  Let $X_s$ denote the number of overlap cells in our $s$-set after our column permutation, and $\chi_i$ denote the event that the cell containing symbol $s$ in row $i$ is an overlap cell after this permutation.  As noted before, Pr($\chi_j) = \epsilon$ for any $j$.  Furthermore, for any $i \neq j$, we have that $Pr(\chi_j \textrm{ and }\chi_i) \leq \epsilon \cdot \frac{\epsilon n}{n-1}$.  To see this, let $a$ be the column incident with (row $i$, symbol $s$) in $L$, and $b$ be the column incident with (row $j$, symbol $s$) in $L$.  If we are permuting $P$'s columns, have just decided which column $a$ maps to, and are now deciding which column $b$ maps to, there are at most $\epsilon n$ choices for $b$ that cause $\chi_j$ to hold, out of $n-1$ total options.  Therefore, $\textrm{Pr}(\chi_i \textrm{ and } \chi_j) = \textrm{Pr}(\chi_i) \cdot \textrm{Pr}(\chi_j \textrm{ given } \chi_i )\leq \epsilon \cdot \frac{\epsilon n }{n-1}$, and thus the variance is
\begin{align*}
\mathbb{E}(X_s^2) - \left(\mathbb{E}(X_s)\right)^2= &\mathbb{E}\left(\left( \sum_{i=1}^n \chi_i \right)^2\right) - \left( \mathbb{E}\left( \sum_{i=1}^n \chi_i \right) \right)^2 \\
\leq& \left(\sum_{i=1}^n \textrm{Pr}(\chi_i) + \sum_{i\neq j} \textrm{Pr}(\chi_i \textrm{ and } \chi_j)\right) - (\epsilon n)^2\\
\leq & \left( \epsilon n + \sum_{i \neq j} \frac{\epsilon^2 n}{n-1}\right) - (\epsilon n)^2\\
= &\epsilon n.
\end{align*}

Suppose now that we have taken an $\epsilon$-dense partial Latin square and randomly permuted its rows and columns.  Fix a symbol $s$ in $L$, and look at the corresponding $s$-set in $L$.  After we permute the symbols of $P$, what is the expected number of cells in our $s$-set that are both overlap cells and symbol-dependent cells, or symbol-dependent cells in two different ways?  

Suppose that $\frac{n}{j}  \geq X_s > \frac{n}{j+1}$, for some $j$.  Notice that permuting the symbols of $P$ doesn't change which of the cells in our $s$-set are overlap cells, while it \textbf{does} permute which cells are symbol-dependent.  In particular, take any filled cell in $P$ that corresponds to an overlap cell in our $s$-set.  There are at most $\frac{n}{j}$ out of the total $n$ choices of symbol to place in this filled $P$-cell, that will cause its corresponding symbol-dependent cell in $L$ to land on one of the $\frac{n}{j}$ overlap cells in our $s$-set.  Therefore, the expected number of such cells over the entire $s$-set is no more than $\frac{n}{j^2}$.  

Notice that each choice of symbol creates precisely one symbol-dependent cell to avoid.  Therefore, the probability that a given choice of symbol creates a cell that is symbol-dependent in two different ways is bounded above by the probability that choosing the \textbf{last} symbol for our cells in $P$ creates such a cell, which is no more than $\frac{n/j}{n - n/j} = \frac{1}{j-1}$.  So the expected number of such cells is no more than $\frac{n}{j(j-1)}$.  In the event that $j=1$, we can do better and bound this above by $n$, as there are at most $n$ symbols in $X_s$.

Using similar logic, note that if $X_s \leq k \epsilon n$, for some $k$, then the expected number of (overlap+symbol-dependent) cells is no more than $k^2 \epsilon^2 n$ and the expected number of doubly-symbol-dependent cells is no more than $ \frac{k^2\epsilon^2 n}{1 - k \epsilon}$.

To bound the likelihood that $X_s$ exceeds $\frac{n}{j}$, we can use Chebyshev's inequality to get the following inequality.
\begin{align*}
\textrm{Pr}\left(X_s  > \frac{n}{j}\right) < \frac{1}{\left( \frac{1}{\epsilon j} - 1 \right)^2\epsilon n}.
\end{align*}
Therefore, if we want to count the total number of these cells, we can simply split $X_s$ into $\left\lceil\frac{1}{k \epsilon}\right\rceil$ cases.  Either $X_s$ is between $\frac{n}{j}$ and $\frac{n}{j+1}$ for some $j$ in $\{1, \ldots \left\lceil \frac{1}{k\epsilon}\right\rceil  -1\}$, or $X_s < k \epsilon n$.  If we use the Chebyshev-derived inequality that we discussed earlier, we can bound the probability that $X_s$ lands into any of these cases, and therefore bound the expected number of (overlap+symbol) and (symbol+symbol) cells in our $s$-set with the following sum.
\begin{align*}
& \left(\frac{1}{\left( \frac{1}{2\epsilon } - 1 \right)^2\epsilon n}\right) (2n) + \sum_{j=2}^{\lceil 1/k\epsilon \rceil - 1}\left(\frac{1}{\left( \frac{1}{\epsilon (j+1)} - 1 \right)^2\epsilon n}\right) \cdot \frac{2n}{j(j-1)}+  k^2 \epsilon^2 n + \frac{k^2\epsilon^2 n}{1 - k \epsilon}\\
\leq & 81\epsilon + \sum_{j=26}^{\lceil 1/k\epsilon \rceil - 1} 2\epsilon\left( \frac{(j+1)^2}{\left(1 - \epsilon j - \epsilon \right)^2 \cdot (j-1)j} \right)+  k^2 \epsilon^2 n + \frac{k^2\epsilon^2 n}{1 - k \epsilon}.\\
\leq & 81 \epsilon + \sum_{j=26}^{\lceil 1/k\epsilon \rceil - 1} \frac{5\epsilon}{4}\left( \frac{1}{\left(1 - \epsilon j - \epsilon \right)^2} \right)+  k^2 \epsilon^2 n + \frac{k^2\epsilon^2 n}{1 - k \epsilon}\\
\leq &81\epsilon + \frac{5k}{4(k - 1 - k\epsilon )^2} + k^2 \epsilon^2 n+ \frac{k^2\epsilon^2 n}{1 - k \epsilon}.\\
\end{align*}
Some optimization suggests that for all of the values of $\epsilon, n$ for which this result will not be superseded by Theorem $\ref{thm1}$, we should set $k=8$.  Doing this, simplifying, and summing this expectation over all symbols $s$ in $L$ gives us the following upper bound on the expected number of (overlap+symbol) and (symbol+symbol) cells.
\begin{align*}
\mathbb{E}\left( \#\textrm{(overlap+symbol)}+\#\textrm{(symbol+symbol)} \right) = 81\epsilon n + \frac{205}{1000}n + 129 \epsilon^2 n^2
\end{align*}
To perform the simplification above, we assumed that $\epsilon < \frac{1}{100}$.  In the event that $\epsilon > \frac{1}{100}$, the bound on the total number of nonoverlapping cells in our lemma's statement is vacuous; so we are free to ignore this case.

The number of (row + symbol) cells is counted in a similar fashion.  Again, suppose that we have already permuted the rows of $P$.  Fix a symbol $s$.  If we permute the columns of $P$ at random, our work above has shown that the expected number of symbol-dependent cells that land in our $s$-set is $\epsilon n$, and that the variance is $\epsilon n$.

Suppose for the moment that the number of symbol-dependent cells in our $s$-set is no more than $l\epsilon n$, for some $l > 1$.  Now, randomly permute the symbols of $P$.  Specifically, choose our random permutation as follows:\ first select the symbol that gets mapped to $s$, and then choose where all of the remaining symbols that occur in the cells in $P$ that induce symbol-dependent cells in our $s$-set.  

In choosing the symbol $t$ that maps to $s$, we simultaneously fix all of the cells in our $s$-set that are row-dependent, as well as all of the cells that are symbol-dependent that originally contained a $t$.  In doing this, it is possible that through ``poor luck'' all of the symbol-dependent cells that used to contain a $t$ land on the row-dependent cells.  In this situation, when we choose where to send the remaining symbols that induce symbol-dependent cells in our $s$-set, there are at most $\epsilon n$ row-dependent cells that we could induce a symbol-dependent cell in.  Therefore, given any cell inducing a symbol-dependent cell in our $s$-set, the probability that mapping the symbol contained in that cell makes this symbol-dependent cell land on a row-dependent cell is no more than $\frac{\epsilon n}{n-1}$.  Accordingly, the expected number of cells that are both row-dependent and symbol-dependent generated by these choices is no more than $l\frac{\epsilon^2 n^2}{n-1}$.

So, we simply need to deal with the ``poor luck'' case above.  We do this as follows:\ look at the cells in $P$ that induce symbol-dependent cells in our $s$-set.  Call a symbol ``bad'' if it occurs more than $d l \epsilon n$ times in this collection, for some constant $d$ that we will decide later.  At most $\frac{1}{d}$ symbols are bad.  Therefore, in our first step, when we select the symbol $t$ that maps to $s$, we have at most a $\frac{1/d}{n}$ chance of picking a bad symbol for $t$, and at least a $\frac{n-\frac{1}{d}}{n}$ chance of not doing so.  In the case where we choose a ``bad'' symbol, we can simply assume that all $\epsilon n$ row-dependent cells are also symbol-dependent.  In the case where we have not done so, we can assume that at most the $d l \epsilon n$ resulting symbol-dependent cells are also row-dependent.

Consequently, the expected number of cells in our $s$-set that are both row-dependent and symbol-dependent is at most
\begin{align*}
\left( \frac{1}{(l-1)^2\epsilon n} + \frac{1}{dn} \right) \cdot \epsilon n +   l \frac{\epsilon^2 n^2}{n-1} + d l \epsilon n .
\end{align*}
Some simple calculus suggests that setting $d = \frac{1}{\sqrt{\epsilon n l}}, l = 10$ is roughly optimal for all of the cases of $n, \epsilon$ where this lemma is useful.  If we do this, sum over all $n$ $s$-sets in $L$, and use the simplifying observation that $\epsilon n \geq 1$ for any nontrivial choices of $P$, we get the following upper bound on the expected number of (row+symbol) cells in $L$.
\begin{align*}
\mathbb{E}(\#\textrm{(row+symbol)}) \leq \frac{n}{81} + 17 \epsilon^2 n^2
\end{align*}
The same argument counts the number of (column + symbol) cells.  

We have therefore bounded the expected number of all possible conflicts that a pair of $2 \times 2$ trades can have with each other.  If we sum these bounds, we have that the expected total number of cells involved in multiple $2 \times 2$ trades is at most
\begin{align*}
81 \epsilon n + \frac{23}{100} n +166\epsilon^2 n^2.
\end{align*}
\end{proof}

Our final theorem, roughly speaking, is the claim that Lemma $\ref{lem4}$ does allow us to roughly improve the bounds of Theorem $\ref{thm1}$ by a factor of 2.

\begin{thm}\label{probthm1}
Any $\epsilon$-dense partial Latin square $P$ is completable, for $\epsilon, \delta, n$ such that
\begin{align*}
12 &\leq n - 12 n\sqrt{36\delta + \frac{198\delta}{n} +  \frac{5346 \epsilon}{n} + \frac{1518}{100 \cdot n}  + 10956 \epsilon^2} -  12\epsilon n.\\
\end{align*}
\end{thm}

\begin{proof}

Take an $\epsilon$-dense $n \times n$ partial Latin square containing no more than $\delta n^2$ many cells, and use Lemma $\ref{lem1}$ to construct an $n \times n$ Latin square $L$ of the same dimension.  Using Lemma $\ref{lem4}$, select a permutation of $P$'s rows/columns/symbols such that there are at least
\begin{align*}
(\ddag) = \delta n \left( \left\lfloor \frac{n}{2} \right\rfloor  - 2 \right) -81 \epsilon n -\frac{23}{100} n - 166\epsilon^2 n^2
\end{align*}
 distinct cells in $P$ with associated nonoverlapping $2 \times 2$ trades.  Apply these permutations to $P$; if this permuted partial Latin square can be completed to some $L$, then reversing these permutations on both $P$ and $L$ will yield a completion of our original square.

Perform all of the nonoverlapping $2 \times 2$ trades guaranteed by Lemma $\ref{lem4}$; this disturbs $3(\ddag)$ cells that are not places where $P$ and $L$ agree.  Now, cell by cell, select a filled cell $(r_1, c_1)$ of $P$ at which $P(r_1, c_1) \neq L(r_1, c_1)$, and apply Lemma $\ref{lem3}$ to find a trade that disturbs at most 69 other cells and that causes $L$ and $P$ to agree at this cell. Doing this for every remaining cell at which $P$, $L$ disagree disturbs at most $69(\delta n^2 - (\ddag))$ cells.

In total, we have disturbed at most
\begin{align*}
kn^2 := \left(36\delta + \frac{198\delta}{n} +  \frac{5346 \epsilon}{n} + \frac{1518}{100 \cdot n}  + 10956 \epsilon^2\right) n^2
\end{align*}
cells in total by the end of our proof.

Therefore, by Lemma $\ref{lem3}$, to decide whether we can perform all of these trades, it suffices to find constraints on $\delta, \epsilon, n$ such that we can consistently perform Lemma $\ref{lem3}$ on $L$ until it is a completion of $P$.  In other words, it suffices to choose  $\delta, \epsilon, n$ such that the inequality
\begin{align*}
12 &\leq n - 12 n\sqrt{36\delta + \frac{198\delta}{n} +  \frac{5346 \epsilon}{n} + \frac{1518}{100 \cdot n}  + 10956 \epsilon^2} -  12\epsilon n.\\
\end{align*}
holds.
\end{proof}
For somewhat small values of $\epsilon, \delta$ and somewhat large values of $n$, the above formula is effectively
\begin{align*}
12 &\leq n - 12 n\sqrt{36\delta} -  12\epsilon n,\\
\end{align*} 
which is an improvement on the bounds of Theorem $\ref{thm1}$ by about a factor of 2.  We finally note that this improvement is particularly noticeable when $\epsilon = \delta$.
\begin{cor}\label{probcor1}
All $\frac{1}{6000}$-dense partial Latin squares are completable, for $n > \frac{1}{25000}$.
\end{cor}

\section{Future Directions}

While the twofold improvement above is decent, it still leaves us rather far from our conjectured bound of $\frac{1}{4}$.  This is largely because the techniques above are only used to improve the input $P$ that we put into our algorithm.  Once we start the algorithm itself, however, we are still proceeding deterministically; consequently, we still need the large amount of structure/resources to create these trades.

This, however, is not necessarily a property held by every system of trades.  In particular, the Jacobson and Matthews \cite{Jacobson_Matthews_1996} paper that introduced the concept of improper trades proved that randomly applying such trades generates a random walk on the space of Latin squares.  In particular, this means that if a partial Latin square $P$ has a completion, we can \textbf{find} this completion by simply starting with an arbitrary Latin square $L$ and randomly performing improper $2 \times 2$ trades.  Eventually, should a completion of $P$ exist, we will randomly walk to this completion.

In particular, this tells us that the bottlenecks imposed by our need for ``global'' structure in our earlier proof are in some sense artificial; Jacobson and Matthew's result tells us that we could start from \textbf{any} Latin square $L$ and find appropriate trades to turn it into a completion of $P$.

\chapter{$\epsilon$-Dense Partial Latin Squares and Triangulations of Dense Graphs}

Throughout this paper, we have frequently used the connection between triangle decompositions of tripartite graphs and partial Latin squares to shed insight on how partial Latin squares work.  In this chapter, we will pursue a little bit of the converse; i.e.\ we will attempt to use our results on partial Latin squares to create and study triangle decompositions of graphs.  

To be specific:\ Gustavsson's thesis \cite{Gustavsson_1991} used the Chetwynd-H\"aggkvist result to triangulate various families of ``dense'' graphs.  With some slight modifications to his techniques, we can slot our Theorem $\ref{thm1}$ into these proofs and improve his results.  The bounds achieved improve previous work in a similar fashion to the improvements we got on Chetwynd and H\"aggkvist's result:\ i.e.\ we achieve a decent improvement on the value of $\epsilon$, and (more interestingly) an ability to decouple the local bounds $\epsilon$ on the degrees from the global bounds $\delta$ on the total number of edges.

\section{Triangulating a Tripartite Graph}

Gustavsson's thesis opens with the following result.
\begin{thm}\label{gtri1}
[Gustavsson, 1991].  Let $G$ be a tripartite graph with tripartition $(V_1, V_2, V_3)$, with the following properties.
\begin{itemize}
\item  $|V_1| = |V_2| = |V_3| = n$. 
\item For every vertex $v \in V_i$, $\deg_+(v_i) = \deg_-(v_i) \geq (1- \epsilon)n.$ 
\end{itemize}  
Then, if $\epsilon \leq 10^{-7}$, this graph admits a triangle decomposition.
\end{thm}

Using Theorem $\ref{thm1}$, we improve this result as follows.
\begin{thm}\label{tri1}
 Let $G$ be a tripartite graph with tripartition $(V_1, V_2, V_3)$, with the following properties.
\begin{itemize}
\item  $|V_1| = |V_2| = |V_3| = n$. 
\item For every vertex $v \in V_i$, $\deg_+(v) = \deg_-(v) \geq (1- \epsilon)n.$ 
\item $|E(G)| > (1 - \delta)\cdot 3n^2$.  
\end{itemize} 
Then, if $\epsilon < \frac{1}{132}$, $\delta < \frac{(1 -132\epsilon)^2  }{83272}$,, this graph admits a triangle decomposition.
\end{thm}
\begin{proof}
The idea for this proof is relatively straightforward, and uses the connections between Latin squares and tripartite graphs that we have seen in previous chapters.  Recall that by identifying the three parts of a tripartite graph with the rows, columns, and symbols of a Latin square, we can turn any triangulation of such a tripartite graph into a partial Latin square:\ we do this by mapping triangles $(r_i, c_j, s_k)$ to filled cells $(r_i, c_j, s_k)$.

Suppose for the moment that the tripartite complement $\overline{G}$ of our graph admitted a triangle decomposition.  Then $G$ itself would be the defect of the partial Latin square $L_{\overline{G}}$ corresponding to this triangulation of $\overline{G}$.  Then, any completion of $L_{\overline{G}}$ would necessarily correspond to a triangulation of $G$ itself, which is what we are looking for in this proof.

Unfortunately, however, an arbitrary dense $\overline{G}$ will not always admit a triangle decomposition.  To create a relatively simple example, consider the graph $G$ corresponding to the partial Latin square 
\begin{align*}
L = \begin{array}{|c|c|c|c|} \hline 1 & 2 & 3 & \\ \hline 2 & 1 & 4 & 3 \\ \hline 3 & 4 & 1 & 2 \\ \hline & 3 & 2 & 4 \\ \hline \end{array}.
\end{align*}
The tripartite complement $\overline{G}$ to $G$ is just a hexagon.
\begin{center}
\includegraphics[width=2in]{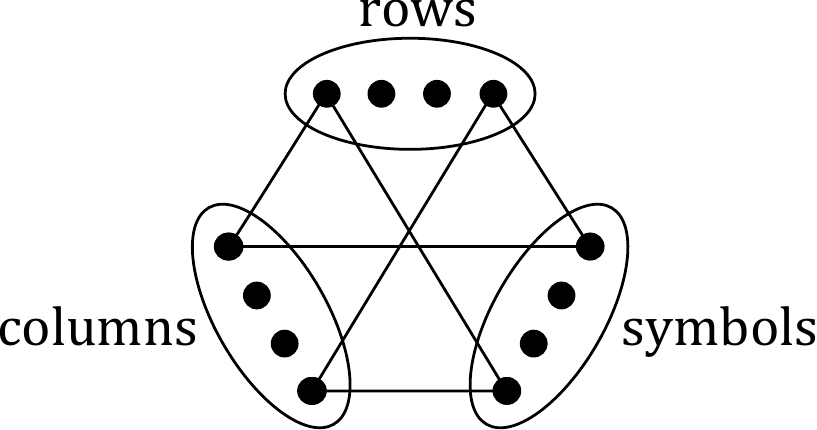}
\end{center}
In general, take any Latin square $L$ that contains a $2 \times 2$ trade of the form
\begin{center}
\includegraphics[width=2.7in]{LS_completions_talk_gen_trade_motive.pdf}.
\end{center}
Delete the contents of the four cells involved in this trade.  Fill in the top-left cell with the symbol $s_1$, and the bottom-right cell with the symbol $s_2$.  This new partial Latin square cannot be completed.  Correspondingly, if $G$ is the graph that corresponds to our modified $L$, $\overline{G}$ does not admit a triangle decomposition, and indeed consists of a hexagon with edges $(r_1, c_2), (c_2, s_1), (s_1, r_2), (r_2, c_1), (c_1, s_2), (s_2, r_1)$.

However, in these examples the graphs $G$ all still admitted triangle decompositions, as demonstrated by their corresponding partial Latin squares.  Therefore, the condition that $\overline{G}$ admits a triangle decomposition is sufficient but not necessary for $G$ to admit a triangle decomposition.

This acts as motivation for the proof methods we use here.  Suppose that we take $G$, and delete a number of edge-disjoint triangles from our graph to get a new tripartite graph $G'$.  Any triangulation of the tripartite complement $\overline{G'}$ of this new graph will, as before, correspond to a partial Latin square $L_{\overline{G'}}$.  Furthermore, any completion of $L_{\overline{G'}}$ will still correspond to a triangulation of $G'$, and therefore to $G$ itself, by ``adding back in'' the triangles we deleted earlier.  Therefore, if we can delete triangles from $G$ in a sufficiently clever way that allows us to
\begin{itemize}
\item use these triangles to create a triangle decomposition of $\overline{G'}$, while
\item not decreasing the local degree of any vertex too much,
\end{itemize}
we can use Theorem $\ref{thm1}$ on $L_{\overline{G'}}$ and apply the corresponding completion to prove our theorem.

To do this, notice that because $\deg_+(v) = \deg_-(v)$ for any vertex $v$ in $G$, this property also holds for any vertex in the complement.  Therefore, suppose we take any vertex $v \in V_1$ with nonzero degree.  Travel from $v$ to some neighbor $w$ in $V_2$; because $\deg_+(w) = \deg_-(w)$, there is some edge from $w$ to $V_3$.  Travel along that edge to some new neighbor, which by the same logic has a neighbor in $V_1$, and repeat this process until we travel to a vertex we have already visited.  By starting and ending at this repeated vertex, we have found a cycle of length $0$ mod $3$ in $\overline{G}$.  Because deleting this cycle doesn't change the property that $\deg_+(v) = \deg_-(v)$  for any $v$, we can repeat this process to decompose $\overline{G}$ into a collection of cycles, all of length $0 \mod 3$.

We now want to add triangles to $\overline{G}$ in such a way that we can transform these cycles into more triangles.  To do this, first note that the notion of \textbf{trades} is not limited in concept to Latin squares.
\begin{defn}
Let $G$ be a graph with associated decomposition $\mathcal{H} = \{H_1, \ldots H_k\}$.  Pick any subset of these subgraphs $\{H_1', \ldots H_l'\}$.  The union of these subgraphs creates some specific subgraph of $G$ that may in turn have some other graph decomposition $\{H_1^\star, \ldots H_m^\star\}$.  Suppose such a secondary decomposition exists.  Then, if we take $\mathcal{H}$ and exchange the  $\{H_1', \ldots H_l'\}$ subgraphs for the $\{H_1^\star, \ldots H_m^\star\}$ subgraphs, this newly modified $\mathcal{H}$ is still a decomposition of $G$.  We call any such pair  $\{H_1', \ldots H_l'\}$, $\{H_1^\star, \ldots H_m^\star\}$ a \textbf{trade} on $(G, \mathcal{H})$.

From this perspective, if we view a Latin square $L$ as a triangulation of the complete tripartite graph, a trade on $L$ is just a way of switching between different sets of triangle decompositions of $L$.
\end{defn}

Using this idea, Gustavsson's trick was the repeated use of the following pair of trades on $\overline{G}$.

\begin{center}
\includegraphics[width=5in]{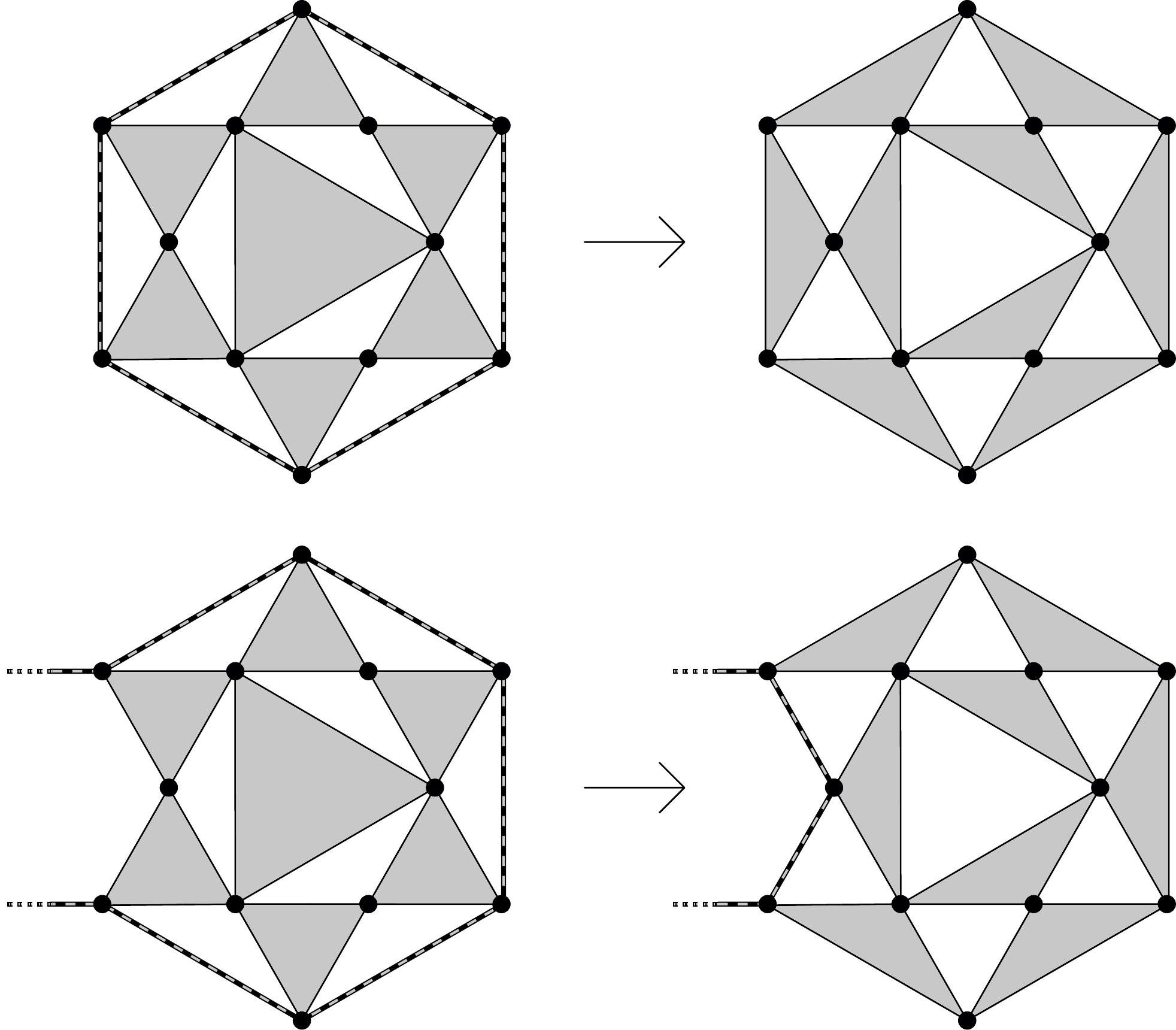}
\end{center}

Select any cycle in $\overline{G}$ that is not already a triangle.  For any such cycle, choose six adjacent vertices within that cycle if it is not a hexagon, or take the whole hexagon if it is a hexagon.  We will then attempt to find six vertices in $G$, such that the 21 distinct edges needed to form the seven-triangle configuration illustrated above exists.  If it does, we will then temporarily ``delete'' these seven triangles from $G$, and trade their edges along with the cycle edges in $\overline{G}$ as illustrated above.  If this leaves a smaller cycle, pick the three vertices that were just used in the earlier trade, along with three more neighboring ones, so that these three are neither the first three or last three adjacent vertices we have picked, and repeat this process.  (We do this to make sure we do not pick any vertex in a cycle too many times.)  Otherwise, if this process occurs on a hexagon, simply pick a new cycle from $\overline{G}$ and repeat this process.

Each time we perform such a trade on $\overline{G}$, the number of edges in cycles goes down by at least 3.  So, to triangulate $\overline{G}$, it suffices to show that we can find $\epsilon n^2$ many such trades, as there are at most $3\epsilon n^2$ edges in $G'$.

We find these trades one at a time, choosing the vertices involved carefully so that we do not select any vertex ``too often'' and thereby reduce its degree in $G$ too heavily.  To aid in this process, call a vertex \textbf{$\gamma$-overloaded} if it has been used no less than $\gamma n$ times by these trades.  Note that by the time we have completed all of our trades, there are no more than $\dfrac{2 \cdot \delta n}{\gamma}$ overloaded vertices in any part $V_i$ of our graph, as we are looking for $\delta n^2$ trades, each of which uses two vertices from any given part.  We will insure that no vertex is ever $\gamma$-overloaded throughout our proof.

 Let $w_1, w_2, \ldots w_6$ denote the six vertices involved in the $\overline{G}$ cycle we seek to eliminate, listed in the order they occur in our cycle.  Assume without loss of generality that $w_1, w_4 \in V_1, w_2, w_5 \in V_2, w_3, w_6 \in V_3$.

Choose the vertex $x_1 \in V_3$ such that 
\begin{itemize}
\item $x_1$ is not $\gamma$-overloaded.  This eliminates at most $\left\lceil \dfrac{2 \cdot \delta n}{\gamma} \right\rceil \leq \dfrac{2 \cdot \epsilon n}{\gamma}  + 1$ choices.
\item $x_1$ has edges to $w_1, w_2$ in $G$.  Because $\deg(x_1) \geq (1-\epsilon)n$, this eliminates at most $2 \epsilon n$ choices.
\item The edges that $x_1$ has to $w_1, w_2$ in $G$ have not been used in previous trades as $x_i\leftrightarrow x_j$-type edges.  Because $w_1, w_2$ are not $\gamma$-overloaded, this eliminates at most $2 \cdot (3 \gamma n)$ many edges from situations where one of these two vertices were used as an $x_i$ in a previous trade, because each $x_i$ in our trade has at most three edges to any given part.  
\item  The edges that $x_1$ has to $w_1, w_2$ in $G$ have not been used in previous trades as $w_i\leftrightarrow x_j$-type edges.  Note that each vertex in a cycle uses at most one edge to another part in any such trade.  Moreover, note that no vertex shows up in more than $\epsilon n$ many cycles, and for a given cycle no vertex is used more than twice (as described in our earlier discussion for how we iteratively find these trades.)  Therefore, this elminates at most $4\epsilon n$ choices.  
\item $x_1$ is not one of the $w_i$'s.  Because there are two $w_i$'s in each part, this eliminates two choices.
\end{itemize}
In total, we have at most
\begin{align*}
\frac{2 \cdot \delta n}{\gamma} + 6\epsilon n + 6 \gamma n + 3
\end{align*}
many disallowed choices of $x_1$.

Similarly, choose $x_3 \in V_2$ such that
\begin{itemize}
\item $x_2$ is not $\gamma$-overloaded.  This eliminates at most $\left\lceil \dfrac{2 \cdot \delta n}{\gamma} \right\rceil = \dfrac{2 \cdot \epsilon n}{\gamma}  + 1$ choices.
\item $x_2$ has edges to $w_3, w_4, x_1$ in $G$.  This eliminates at most $3 \epsilon n$ choices.
\item The edges that $x_3$ has to $w_3, w_4, x_1$ in $G$ have not been used in previous trades as $x_i\leftrightarrow x_j$-type edges.  This eliminates at most $9 \gamma n$ choices.
\item  The edges that $x_3$ has to $w_3, w_4, x_1$ in $G$ have not been used in previous trades as $w_i\leftrightarrow x_j$-type edges. This elminates at most $6\epsilon n$ choices.  
\item $x_3$ is not one of the $w_i$'s.  Because there are two $w_i$'s in each part, this eliminates two choices.
\end{itemize}
In total, we have at most
\begin{align*}
\frac{2 \cdot \epsilon n}{\gamma} + 9\epsilon n + 9\gamma n + 3
\end{align*}
many disallowed choices of $x_3$.

In an identical process, pick $x_5 \in V_1$ such that it has the requisite edges to $w_5, w_6, x_1, x_3$; then, pick $x_2$ so that it has edges to $w_2, w_3, x_1, x_3$, $x_4$ so that it has edges to $w_4, w_5,x_3, x_5$, and $x_6$ so that it has edges to $w_6, w_1, x_5, x_1$.  If we have done this in the manner described above, these edges will all exist and not have been picked too often or already used in this trade.  For any of these vertices, there are at most
\begin{align*}
\frac{2 \cdot \delta n}{\gamma} + 12\epsilon n + 12\gamma n + 4
\end{align*}
many disallowed choices, using the logic above.  If we can always make these choices, we can always find these trades, and by performing them create a new graph $G'$ with the following properties.
\begin{itemize}
\item  $|V_1| = |V_2| = |V_3| = n$. 
\item For every vertex $v \in V_i$, $\deg_+(v) = \deg_-(v) \geq (1- \epsilon - 3\gamma)n.$   The $3\gamma$ comes from the fact that a vertex loses at most $3$ edges per trade performed, and no vertex is ever used once it is $\gamma$-overloaded.
\item $|E(G)| > (1 - 8\delta)\cdot 3n^2$.  The additional $7 \delta n^2$ comes from the fact that we have to find at most $\delta n^2$ many such trades, and each uses 21 edges.
\item $\overline{G'}$ has a triangle decomposition.
\end{itemize} 

If we regard $\overline{G'}$ as a partial Latin square, triangulating this $G'$ is equivalent to completing this partial Latin square, and can be done with Theorem $\ref{thm1}$ whenever $\epsilon + 3\gamma < \frac{1}{12}, 8\delta <\frac{(1 -12(\epsilon + 3\gamma))^2  }{10409}$.  So it suffices to determine what choices of $\gamma$ will simultaneously satisfy 
\begin{align*}
(\ddag) \qquad n - \frac{2 \cdot \delta n}{\gamma} - 12\epsilon n - 12\gamma n \geq 5
\end{align*}
and maximize our possible choices for $\epsilon, \delta$.  

The optimal choice of $\gamma$ varies somewhat on whether we are trying to optimize $\epsilon$ at the expense of $\delta,$ or whether we are studying the situation where $\epsilon \sim \delta$.  For that reason, in specific edge cases the reader is advised to simply take these three bounds and optimize on their own if a shift in bounds is needed.  

In most cases, however, it is relatively clear that we want to make $\gamma$ as small as reasonably possible, so that our range of choices of $\epsilon$ is as broad as possible.  This can be done, with some rudimentary optimization, by setting $\gamma = 3.3\epsilon$.  In this situation, if $\epsilon < \frac{1}{132}$, $\delta < \frac{(1 -132\epsilon)^2  }{83272}$, we satisfy the bounds required by Theorem $\ref{thm1}$ and $(\ddag)$ for all $n$ where this graph is not empty.  Therefore, we have proven our claim.
\end{proof}

Setting $\epsilon = \delta$ gives us the following corollary.
\begin{thm}
 Let $G$ be a tripartite graph with tripartition $(V_1, V_2, V_3)$, with the following properties.
\begin{itemize}
\item  $|V_1| = |V_2| = |V_3| = n$. 
\item For every vertex $v \in V_i$, $\deg_+(v) = \deg_-(v) \geq (1- \epsilon)n.$ 
\end{itemize} 
Then, if $\epsilon <1.197 \cdot 10^{-5}$, this graph admits a triangle decomposition.
\end{thm}

\section{Future Directions}

Gustavsson uses his Theorem $\ref{gtri1}$ to study triangulations of dense graphs in general.
\begin{thm}
[Gustavsson, 1991]  Let $G$ be a graph with the following properties. 
\begin{itemize}
\item $|V(G)| = n$.
\item $\deg(v)$ is even, for every $v\in V(G)$.
\item $|E(G)|$ is a multiple of 3.
\item $\deg(v) > (1-\epsilon) n$, for every $v \in V(G)$.
\end{itemize}
Then, if $\epsilon \leq 10^{-22}$, $G$ admits a triangle decomposition.
\end{thm}

Extending our theorem to such a result is a fairly natural direction to want to go in.  More interestingly, Gustavsson's thesis itself consists around the extension of this idea to finding $K_n$-decompositions of very large dense graphs.  Improving his results here would be valuable.

\chapter{Quasirandom Colorings of Graphs}\label{qrchap}

In this chapter, we take a break from Latin squares and study the concept of quasirandom graphs, a notion first introduced by Chung, Graham and Wilson \cite{chung1989quasi} in 1989.  Roughly speaking, a sequence of graphs is called \textbf{quasirandom} if it has a number of properties possessed by the random graph, all of which (surprisingly) turn out to be equivalent.  In this chapter, we study possible extensions of these results to random $k$-edge colorings, and create an analogue of Chung, Graham and Wilson's result for such colorings.

\section{Basic Definitions}

Consider the following method for generating a ``random'' graph.
\begin{itemize}
\item Take $n$ labeled vertices $\{1, \ldots n\}$.
\item  For each unordered pair of vertices $\{a,b\}$, flip a fair coin.  If it comes up heads, connect these vertices with an edge; otherwise, do not.
\end{itemize}
This process induces a following probability space $G_{n, 1/2}$, as described below.
\begin{itemize}
\item The set for $G_{n, 1/2}$ is the collection of all graphs on $n$ vertices.
\item The probability measure for $G_{n, 1/2}$ is the measure that says that all graphs are equally likely:\ i.e. for any $H \in G_{n, 1/2}$, $\mathbb{P}(H) = \frac{1}{2^{\binom{n}{2}}}$.  
\end{itemize}

Given this model, a natural sequence of questions to ask is the following:\ what properties is a random graph generated by the process above likely to have, as $n$ goes to infinity?  To answer this question, we make the following definitions.
\begin{defn}
A \textbf{graph property} $P$ is simply a collection of labeled graphs.  We say that a given graph $G$ satisfies $P$ if $G$ is an element of $P$.  

Similarly, suppose we have a sequence of graphs $\mathcal{G} = \{G_n\}_{n =1}^\infty$, where $|V(G_n)| = n$.   We say that $\lim_{n \to \infty} G_n$ satisfies $P$ if there is some $N$ such that for all $n > N$, $G_n$ satisfies $P$.
\end{defn}
With these definitions, we can rephrase our earlier question as follows:\ what graph properties $P$ are such that $\lim_{n \to\infty} G_{n, 1/2}$ satisfies $P$?  In most introductory classes to the probabilistic method in combinatorics, a student will quickly calculate a number of such properties that the random graph ``almost always'' possesses.   We list several of these properties below.  To describe them, however, we need some basic notation.
\begin{defn}
Suppose that $G$, $H$ are two graphs.  Let $N_G^*(H)$ denote the number of labeled occurrences of $H$ as an induced subgraph of $G$.  Similarly, let $N_G(H)$ denote the number of labeled occurrences of $H$ as a subgraph of $G$ (not necessarily induced.)
\end{defn}
With this done we list the graph properties below.  In these properties, we make heavy use of the notation $o(n), o(n^2)$, etc.  When we do so, we are using this as shorthand to denote a general class of graph properties, any specific instance of which can be given by making the $o(n)$-portions explicit:\ i.e.\ by replacing every instance of $o(1)$ in a given property with some specific function whose growth rate is $o(1).$
\begin{enumerate}
\item[$P_1(s)$:] 
For any graph $H_s$ on $s$ vertices, 
\begin{align*}
N_G^*(H_s) = \left( 1 + o(1) \right) \cdot n^s \cdot 2^{-\binom{s}{2}}.
\end{align*}
\item[$P_2(t)$:] 
Let $C_{t}$ denote the cycle of length $t$.  Then
\begin{align*}
e(G) &\geq (1+o(1)) \cdot \frac{n^2}{4}, \textrm{ and}\\
N_G(C_{t}) &\leq \left( 1 + o(1) \right) \cdot\frac{n^t}{2^t}.
\end{align*}
\item[$P_3$:] Let $A(G)$ denote the adjacency matrix of $G$, and $|\lambda_1| \geq \ldots \geq |\lambda_n|$ be the eigenvalues of $A(G)$.  Then
\begin{align*}
e(G) &\geq (1+o(1)) \cdot \frac{n^2}{4}, \textrm{ and}\\
\lambda_1 &= (1+o(1)) \cdot \frac{n}{2}, \quad \lambda_2 = o(n).
\end{align*}
\item[$P_4$:] Given any subset $S \subseteq V$,
\begin{align*}
e(S) = \frac{|S|^2}{4} + o(n^2). 
\end{align*}
\item[$P_5$:] Given any subset $S \subseteq V$ with $S = \lfloor n/2 \rfloor$, 
\begin{align*}
e(S) = \frac{n^2}{16}+ o(n^2). 
\end{align*}
\item[$P_6$:] Given any pair of vertices $v, v' \in G$, let $s(v,v')$ denote the number of vertices $y$ such that both $(v, y)$ and $(v', y)$ are either both edges or both nonedges in $G$.  Then
\begin{align*}
\sum_{v, v'} \left| s(v,v')  - \frac{n}{2}\right| = o(n^3).
\end{align*}
\item[$P_7$:]
\begin{align*}
\sum_{v, v'} \left| n(v)\cap n(v')  - \frac{n}{4}\right| = o(n^3).
\end{align*}
\end{enumerate}
Proofs that the random graph satisfies these properties can be found in almost any reference text, e.g.\ \cite{bollobas1998modern}.

Motivated by these properties, we make the following definition.
\begin{defn}
A sequence $\mathcal{G}$ of graphs is called \textbf{quasirandom} if it satisfies all of the properties $P_i$ listed above.
\end{defn}

The most interesting feature of the properties $P_1, \ldots P_7$, as proved by Chung, Graham, and Wilson \cite{chung1989quasi}, is that they are all \textbf{equivalent}.  Specifically, we have the following theorem.
\begin{thm}
[Chung, Graham, Wilson, 1989]  Suppose that $\mathcal{G}$ is a sequence of graphs that satisfies any one of the properties
\begin{itemize}
\item $P_1(s)$, for some $s \geq 4$, 
\item $P_2(t)$, for some $t \geq 4$, or 
\item $P_3$, or $P_4$, or $P_5$, or $P_6$, or $P_7$.
\end{itemize}
Then it satisfies all of the properties.  
\end{thm}

One rather remarkable consequence of the result above is how strong the seemingly-weak property $P_2(4)$ is.  For example, suppose a sequence of graphs has asymptotically the ``same number'' of 4-cycles in its members as the random graph $G_{n, 1/2}$.   Then, the above theorem states that the members of this sequence are somehow forced to have the ``same number'' of copies of induced subgraphs of \textbf{any} graph $H$ --- say, the Petersen graph, or $K_{42}$, or anything else --- as asymptotically occur in $G_{n, 1/2}$.

Since their result, a number of authors have extended Chung, Graham and Wilson's results to generalized random graph models (Lov\'asz and S\'os, \cite{Lovasz_Sos_Gen_2008}), graph sequences with given degree sequences (Chung and Graham, \cite{Chung_Graham_DegreeSeq_2008}), and sequences of hypergraphs (notably by Chung \cite{Chung_Hyp1_1990},  \cite{Chung_Hyp2_2012} and more recently by Lenz and Mubayi \cite{Lenz_Mubayi_Hyp4_2012}, \cite{Lenz_Mubayi_Hyp3_2012}.)  The aim of this chapter is similar:\ we want to extend Chung, Graham and Wilson's results to a notion of \textbf{quasirandom $k$-edge-colorings}.  Before we do this, we first note a few examples of quasirandom graphs, to build some intuition for what we are attempting to study.

\section{Examples of Quasirandom Graphs}

These examples and several others can be found in \cite{krivelevich2006pseudo}, which is an excellent survey article on quasirandom graphs.  We start with some definitions.

\begin{defn}
Let $p$ be an odd prime that is $1$ mod $4$.  Define the Paley graph as follows.
\begin{itemize}
\item Our vertex set is $\mathbb{F}_p = \mathbb{Z} / p \mathbb{Z}$.
\item Connect two elements $x, y$ with an edge $\{x, y\}$ if and only if $x-y$ is a \textbf{quadratic residue}:\ i.e. there is some element $a \in \mathbf{F}_p$ such that $x-y = a^2$.
\end{itemize}
This forms an undirected graph whenever $-1$ is a quadratic residue, which is true for $p \equiv 1 \mod 4$.  (When $-1$ is not a quadratic residue, this forms a directed graph, which we do not want to consider here.)
\end{defn}
It is relatively easy to verify that this sequence of graphs is quasirandom.  In particular, verifying the property $P_6$ is just an exercise in counting (for fixed $x, y \in \mathbb{F}_p$) the number of $z \in \mathbb{F}_p$ such that $z-x, z-y$ are either both quadratic residues or both quadratic nonresidues, and verifying that the result is $o(n^3)$.  

This has the nice benefit of giving us certain number-theoretic results without having to actually perform or know any number theory ourselves.  For example, $P_1$ says that for any $k$, there are sufficiently large primes $q$ and subsets $S \subset \mathbb{F}_q$, $|S| = k$ for which any two elements of $S$ differ by a quadratic residue.  Actually constructing such an object would require nontrivial work; but with quasirandom graphs, we are assured of their existence (and indeed the existence of many such subgraphs as $q$ grows large) with no effort at all.

We list two other fairly elegant examples of quasirandom graphs.
\begin{defn}
Let $k$ be an odd integer. Define $H_k$ as the following graph.
\begin{itemize}
\item Our vertex set is $\mathbb{F}_2^k \setminus \{(1, \ldots 1)\}$.
\item Connect two vertices with an edge if and only if the dot product of their two corresponding vectors is an odd number. 
\end{itemize}
\end{defn}
This is a quasirandom graph, as verifying property $P_6$ quickly shows.

\begin{defn}
An \textbf{affine plane} is a collection of points and lines with the following three properties.
\begin{enumerate}
\item[(A1):]  Given any two points, there is a unique line joining any two points.
\item[(A2):]  Given a point $P$ and a line $L$ not containing $P$, there is a unique line that contains $P$ and does not intersect $L$.
\item[(A3):] There are four points, no three of which are collinear.
\end{enumerate}
For example, the following set of nine points and twelve lines defines an affine plane.
\begin{center}
\includegraphics[width=2in]{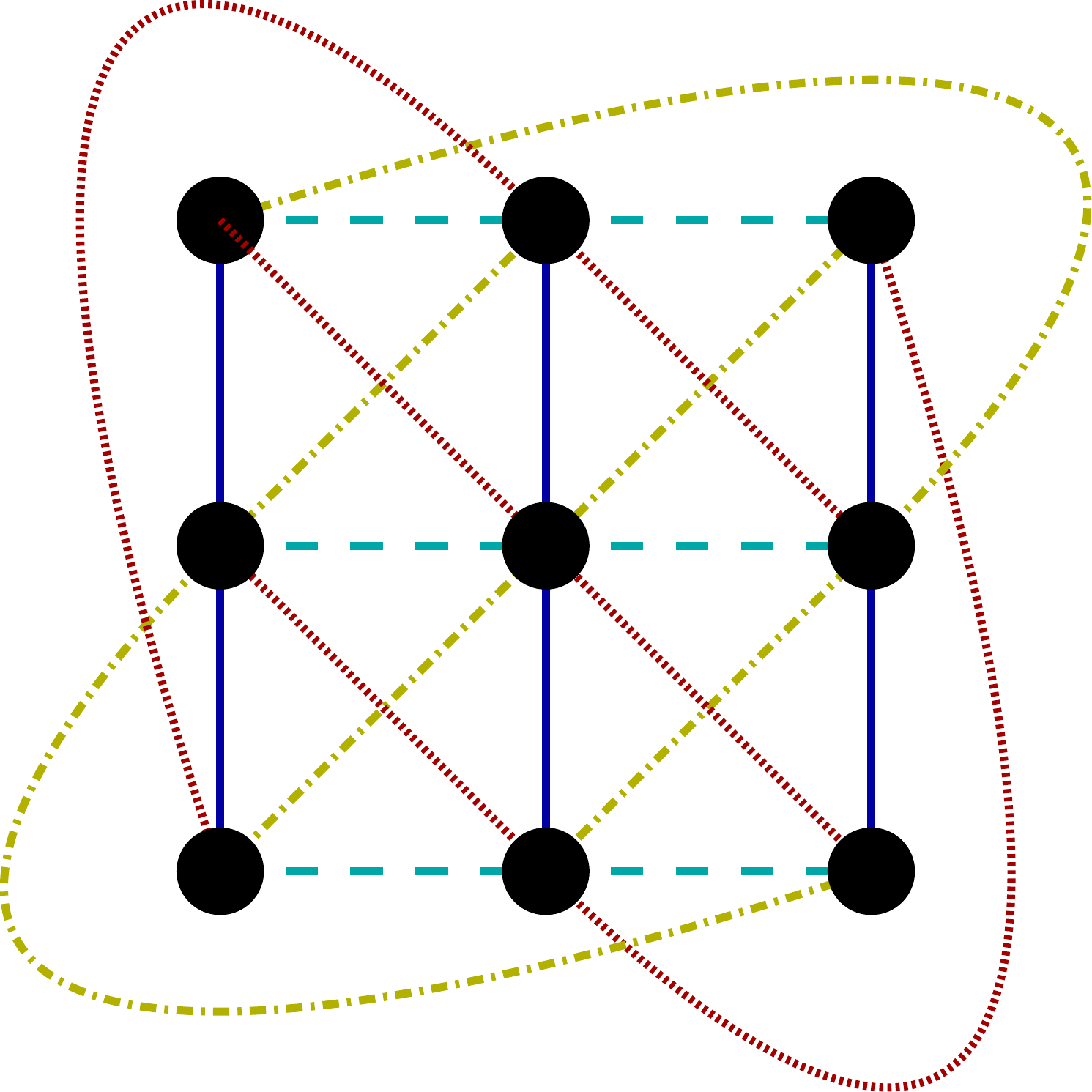}
\end{center}

Similarly a \textbf{projective plane} is a collection of points and lines with the following three properties.
\begin{enumerate}
\item[(P1):]  Given any two points, there is a unique line joining any two points.
\item[(P2):]  Any two distinct lines intersect at a unique point.
\item[(P3):] There are four points, no three of which are collinear. 
\end{enumerate}
\end{defn}

An affine plane of order $n$ can be transformed into a projective plane of order $n$, and vice-versa, via the following construction. 
\begin{itemize}
\item Take an affine plane, and split its lines into $n+1$ collections of $n$ parallel lines.  Label these collections $C_1, \ldots C_{n+1}$.
\item For each class $C_i$, add a point $\infty_i$, and extend every line in this class $C_i$ to contain this point $\infty_i$.
\item Create a line $L_\infty$ that consists of all of these points $\{\infty_i\}_{i=1}^{n+1}$.
\end{itemize}
We illustrate the results of this process below.
\begin{center}
\includegraphics[width=3.5in]{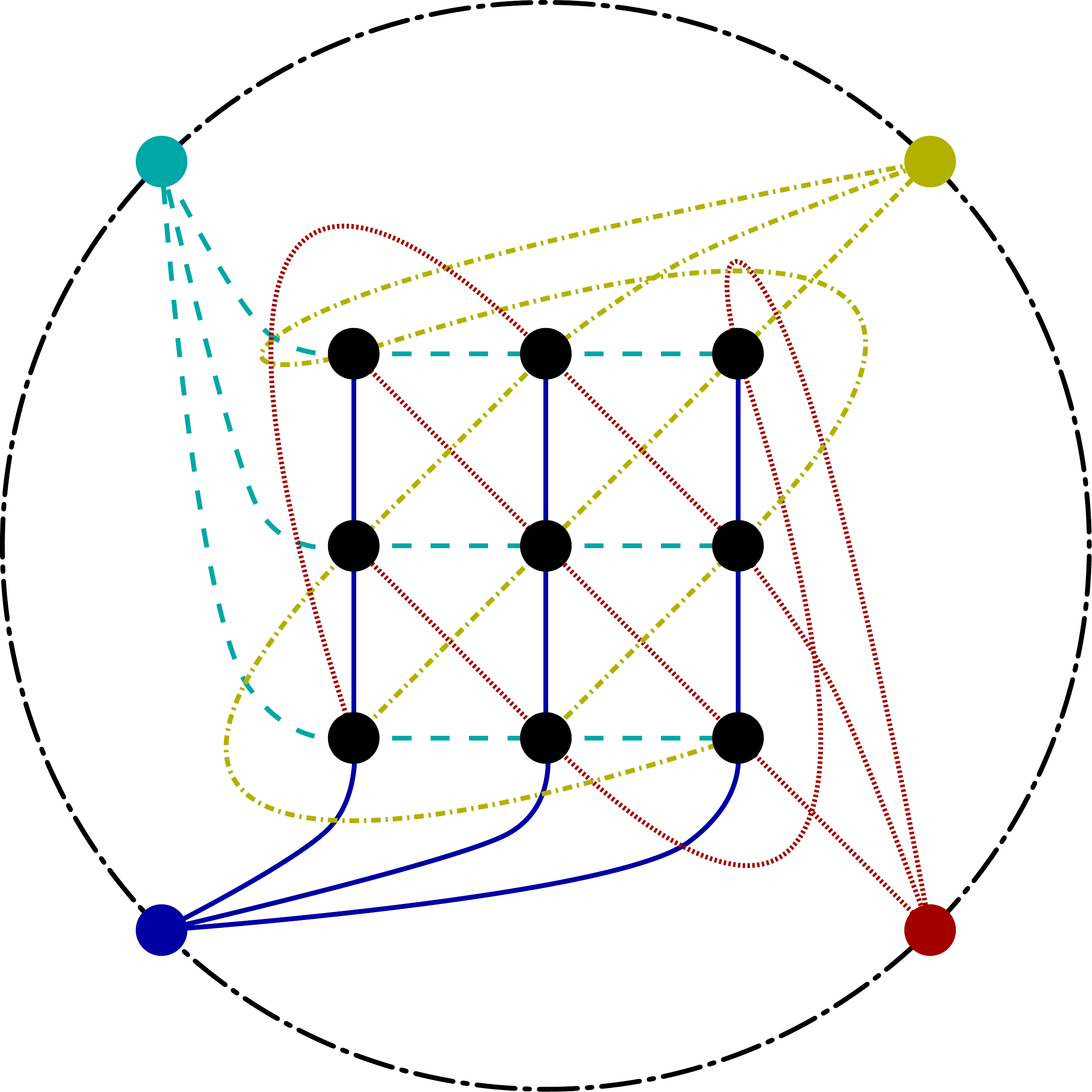}
\end{center}

Using this transformation, we can create quasirandom graphs out of affine planes as illustrated by the following construction.
\begin{itemize}
\item Start with an affine plane $A$ of order n, and augment it via the construction above into a projective plane $A'$.  Let $L_\infty$ be the line ``at infinity'' created by the above process.
\item Take the line $L_\infty$, and split its points up into two disjoint sets $N_+$, $N_-$, such that each of these sets contains half of the points on this line at infinity.  
\item Turn this into a graph as follows:\ our vertex set is the $n^2$-many points in our affine plane $A$.  Connect two vertices $x,y$ in our vertex set with an edge $\{x,y\}$ if and only if the following happens:\ take the unique line in our affine plane $A$ through these two points $x,y$.  If it goes through one of the $\infty_i$-points in our set $N_+$, draw an edge connecting $\{x,y\}$; instead, if it goes through a point in $N_-$, do not draw an edge.
\end{itemize}

Again, by verifying property $P_6$ you can verify that this graph is quasirandom.  (Readers desperate for a connection to Latin squares are welcome to note at this juncture the well-known constructions that transforms a set of $n-1$ MOLS of order $n$ into an affine plane of order $n$, and thereby into a quasirandom graph.)

Before we leave these examples, it bears noting that they illustrate some of the key ways in which quasirandom graphs are \textbf{not} just $G_{n, 1/2}$.  For example, the size of the largest clique in $G_{n, 1/2}$ is $(1 + o(1))(\log(n))/(log(2))$.  However, as shown by S.\ Graham and C.\ Ringrose \cite{graham1990lower}, the clique number of the Paley graph on $q$ vertices is as large as $c \log(q)\log(\log(q))$ for infinitely many values of $q$.

\section{Quasirandom $k$-edge-colorings of graphs}

Throughout this section, we will use the following notation.
\begin{itemize}
\item Let $G = (V,E_1, \ldots E_k)$ denote a $k$-edge coloring of the complete graph on the vertex set $V$, where each $E_i$ is the collection of all $i$-colored edges.  Let $|V| = n$, and $|E_i| = e_i$.
\item Given such a $k$-coloring $G$ and a vertex $v$, let $n_i(v) = \{w \in V:\ \{v, w \} \in E_i\}$, the collection of $i$-\textbf{neighbors} of $v$; i.e. vertices connected to $v$ by $i$-colored edges.  Analogously, let $\textrm{deg}_i(v) = |n_i(v)|,$ the $i$-\textbf{degree} of $v$, and $G_i$ denote the graph $(G, E_i)$ made by discarding all of the non-$i$-colored edges.
\item Suppose that $G$, $H$ are two such $k$-colorings of a complete graphs on vertex sets $V, W$.  Let $N_G^*(H)$ denote the number of labelled occurrences of $H$ as an induced $k$-colored subgraph of $G$ agreeing with $G$'s current coloring.  Similarly, let $N_G(H)$ denote the number of labeled occurrences of $H$ as a $k$-colored subgraph of $G$, not necessarily induced, that agrees with $G$'s current coloring.
\end{itemize}

Given any $k$, we can create a ``random'' $k$-edge-coloring of the complete graph $K_n$ by rolling a $k$-sided die for each edge, and coloring each edge with the result.  Formally, we can identify the results of this process with the probability space $G_{n:\ \frac{1}{k}, \ldots \frac{1}{k}}$.
\begin{itemize}
\item The set for $G_{n:\ \frac{1}{k}, \ldots \frac{1}{k}}$ is the collection of all $k$-colorings of $K_n$.
\item The probability measure for $G_{n, 1/2}$ is the measure that says that all $k$-colorings are equally likely.
\end{itemize}

Given this notion, the following seven properties are the most natural extensions of Chung, Graham and Wilson's results to this notion of $k$-edge-coloring.

\begin{enumerate}
\item[$P_1(s)$:] 
For any graph $H_s$ on $s$ vertices,
\begin{align*}
N_G^*(H_s) = \left( 1 + o(1) \right) \cdot n^s \cdot k^{-\binom{s}{2}}.
\end{align*}
\item[$P_2(t)$:] 
For any given color $i$, let $C_{t,i}$ denote the cycle of length $t$ where all of the edges have color $i$.  Then
\begin{align*}
e_i(G) &\geq (1+o(1)) \cdot \frac{n^2}{2k}, \textrm{ and}\\
N_G(C_{t,i}) &\leq \left( 1 + o(1) \right) \cdot\frac{n^t}{k^t}.
\end{align*}
\item[$P_3$:] For any given color $i$, let $A(G_i)$ denote the adjacency matrix of $G_i$, and $|\lambda_1| \geq \ldots \geq |\lambda_n|$ be the eigenvalues of $A(G_i)$.  Then
\begin{align*}
e_i(G) &\geq (1+o(1)) \cdot \frac{n^2}{2k}, \textrm{ and}\\
\lambda_1 &= (1+o(1)) \cdot \frac{n}{k}, \quad \lambda_2 = o(n).
\end{align*}
\item[$P_4$:] Given any subset $S \subseteq V$ and any color $i$,
\begin{align*}
e_i(S) = \frac{|S|^2}{2k} + o(n^2). 
\end{align*}
\item[$P_5$:] Given any subset $S \subseteq V$ with $S = \lfloor n/2 \rfloor$, and any color $i$,
\begin{align*}
e_i(S) = \frac{n^2}{8k}+ o(n^2). 
\end{align*}
\item[$P_6$:] Given any pair of vertices $v, v' \in G$, let $s(v,v')$ denote the number of vertices $y$ such that both $(v, y)$ and $(v', y)$ are the same color in our coloring of $G$.  Then
\begin{align*}
\sum_{v, v'} \left| s(v,v')  - \frac{n}{k}\right| = o(n^3).
\end{align*}
\item[$P_7$:] Given any color $i$,
\begin{align*}
\sum_{v, v'} \left| \left| n_i(v) \cap n_i(v')\right| - \frac{n}{k^2} \right| = o(n^3).
\end{align*}
\end{enumerate}

Notice that when $k = 2$, we are looking at a set of graph properties modeled on the random graph $G_{n, \frac{1}{2}}$, and therefore that these are precisely the quasirandom graph properties from Chung, Graham and Wilson's paper.

We also introduce the two properties $P_0, P_0'$.  These are strictly weaker than the properties listed above (i.e.\ there are graph sequences that satisfy $P_0$ but not $P_1$,) but they are still remarkably useful in the course of our paper.
\begin{itemize}
\item[$P_0$:] For any color $i$,
\begin{align*}
\sum_{v \in V} \left| \textrm{deg}_i(v) - \frac{n}{k} \right| = o(n^2).
\end{align*} 
\item[$P_0'$:] All but $o(n)$ of the vertices in $G$ have $i$-degree $(1 + o(1))\frac{n}{k}$.
\end{itemize}
Notice that these two properties are equivalent, via Cauchy-Schwarz; we will refer to either interchangeably as $P_0$.

We now prove that all of the properties $P_1, \ldots P_7$ above are equivalent.  
\begin{thm}\label{qrthm}
Suppose that $\mathcal{G}$ is a sequence of $k$-colorings of complete graphs that satisfies any one of the properties
\begin{itemize}
\item $P_1(s)$, for some $s \geq 4$, 
\item $P_2(t)$, for some $t \geq 4$, or 
\item $P_3$, or $P_4$, or $P_5$, or $P_6$, or $P_7$.
\end{itemize}
Then it satisfies all of these properties.  
\end{thm}
\noindent\textit{Proof.}  We proceed by mimicking the proofs used by Chung, Graham and Wilson wherever possible.  The following map illustrates the chain of equivalences we will attempt to show.
\begin{align*}
\begin{array}{ccccccccccc}
P_1(s+1) & \Rightarrow& P_2(s+1)\\
\Downarrow & & \Downarrow\\
P_1(s) & \Rightarrow& P_2(s)\\
\Downarrow & & \Downarrow\\
\vdots & & \vdots& & & & \\
P_1(4) & \Rightarrow &P_2(4) & \Rightarrow & P_3 & \Rightarrow & P_4 & \Rightarrow & P_6 & \Rightarrow( P_1(t), \textrm{ for all }t).\\
 & & \Downarrow &\Nearrow & & & \Updownarrow\\
 & & P_7 & &  & & P_5\\
\end{array}
\end{align*}

\begin{prop}
$P_1(s+1) \Rightarrow P_1(s)$.
\end{prop}
\begin{proof}
Take any $k$-edge-coloring $M(s)$ of the complete graph $K_s$, and make the following observations.
\begin{enumerate}
\item There are $k^s$-many ways, counting different labellings as distinct, to extend any such $M(s)$ to a labeled $k$-coloring of $K_{s+1}$.  
\item As well, if you take any copy of $M(s)$ in $G$, there are precisely $n-s$ induced subgraphs of $G$ that contain that $M(s)$ as a subgraph, as any such graph is formed by simply choosing another vertex of $G$.
\end{enumerate}
 By combining these observations, we can derive the following relation between $N_G^*(M(s+1))$ and $N_G^*(M(s))$.
\begin{align*}
\frac{ N_G^*(M(s))  \cdot (n-s)}{k^s} = N_G^*(M(s+1)).
\end{align*}
Therefore, if $P_1(s+1)$ holds, we can use the property that
\begin{align*}
N_G^*(M(s+1)) = \left( 1 + o(1) \right) \cdot n^{s+1} \cdot k^{-\binom{s+1}{2}}
\end{align*}
to deduce that
\begin{align*}
& \frac{ N_G^*(M(s))  \cdot (n-s)}{k^s} = \left( 1 + o(1) \right) \cdot n^{s+1} \cdot k^{-\binom{s+1}{2}} \\
\Rightarrow \quad & N_G^*(M(s)) =  \left( 1 + o(1) \right) \cdot n^s \cdot k^{-\binom{s}{2}}. \\
\end{align*}
This is precisely $P_1(s)$.
\end{proof}

\begin{prop}
$P_1(3) \Rightarrow P_0$.
\end{prop}
\begin{proof}
Let $H_{h,i,j}$ denote a triangle with one edge colored $h$, one colored $i$, and the other colored $j$.  These kinds of graphs are precisely the objects that $P_1(3)$ tells us about; specifically, if we have $P_1(3)$, we know that
\begin{align*}
N_G^*(H_{h,i,j}) = \left( 1 + o(1) \right) \cdot \frac{n^3}{k^3}.
\end{align*}

To control the $i$-degrees in our graph $G$, we will count these triangles in two ways.  First, observe that
\begin{align*}
\sum_{v \in V} \deg_i(v) \cdot (\deg_i(v) - 1) = \sum_{j=1}^k N^*_G(H_{i,i,j}).
\end{align*}
To see this, notice that the left side simply counts the number of ways to consecutively choose a vertex $v$, an $i$-neighbor $w$, and a second $i$-neighbor $x$.  Therefore, we know that triangles with edge-colorings $(i,i,j)$ with $j \neq i$ will show up precisely twice on the left-hand side, as there are precisely two ways to pick which of $v$'s neighbors will be $w$.  Similarly, for triangles with edge-colorings $(i,i,i)$, there are exactly 6 ways for this to occur; we have three choices of $v$, and two for $w$.  In both cases, these quantities agree precisely with the number of different labellings that these triangles can be given:\ i.e. the number of times these triangles are counted on the right by $N^*_G(H_{i,i,j})$.

Similarly, notice that
\begin{align*}
\sum_{v \in V} \deg_i(v) \cdot (n-2)= \sum_{h=1}^k \sum_{j=1}^k N^*_G(H_{h,i,j}).
\end{align*}
The left side is counting triangles that contain at least one $i$-edge by first picking a vertex $v$, then selecting an $i$-neighbor $w$ and any third choice of neighbor $x$; the right counts them using the $N^*_G(H_{i,i,j})$'s.  By the same methods as above, we can see that both sides count these edge-labeled triangles the same number of times:\ triangles with edges colored $(h,i,j), h \neq i \neq j$ show up once on the left and once on the right, triangles of the form $(j, i, j), i \neq j$ show up twice at left and twice at right, triangles of the form $(j, i, i), i \neq j$ show up four times at left and four times at right, and finally triangles of the form $(i,i,i)$ show up six times at the left and six times at the right.

If we apply $P_1(3)$ to this second equation, we get
\begin{align*}
\sum_{v \in V} \deg_i(v) &= \left( 1 + o(1) \right) \cdot \frac{n^2}{k};\\
\end{align*}
if we then apply $P_1(3)$ to the first equation and plug in the above result, we get as well that
\begin{align*}
 \sum_{v \in V} (\deg_i(v))^2& = \left( 1 + o(1) \right) \cdot \frac{n^3}{k^2}.\\ 
\end{align*}
Finally, using Cauchy-Schwarz tells us that if the above two equations hold, at most $o(n)$ of the $\deg_i(v)$'s are allowed to not be $(1 + o(1)) \frac{n}{k}$.  This is precisely $P_0$.
\end{proof}

\begin{prop}
$P_1(t) \Rightarrow P_2(t)$, for $t \geq 4$.
\end{prop}
\begin{proof}
First, notice that because $P_1(t) \Rightarrow P_1( t-1) \Rightarrow \ldots \Rightarrow P_1(3) \Rightarrow P_0$, we immediately have the edge condition
\begin{align*}
e_i(G) &\geq (1+o(1)) \cdot \frac{n^2}{2k}
\end{align*}
of $P_2$, for every color $i$.  So it suffices to verify that we also have the ``right number'' of $i$-colored $t$-cycles.
Choose any color $i$, and take any $i$-colored cycle $C_{t}$ with labeled vertices.  There are precisely $k^{\binom{t}{2} - t}$-many ways to extend this cycle to a $k$-coloring $H$ of the complete graph on $t$ vertices.  Therefore, we have that
\begin{align*}
N_G(C_{t, i}) = \sum_{H} N_G^*(H) = k^{\binom{t}{2} - t} \cdot \left( 1 + o(1) \right) \cdot n^{t} \cdot k^{-\binom{t}{2}} = \left( 1 + o(1) \right) \cdot \frac{n^{t}}{k^{t}}.
\end{align*}
This is precisely $P_2(t)$.
\end{proof}

\begin{prop}
$P_2(4) \Rightarrow P_3$.
\end{prop}
\begin{proof}
First, notice that we trivially have
\begin{align*}
e_i(G) &\geq (1+o(1)) \cdot \frac{n^2}{2k}
\end{align*}
as it is a condition of $P_2(4)$.

Choose any color $i$, let $A = A(G_i)$ be the adjacency matrix of $G_i$, and $|\lambda_1| \geq \ldots \geq |\lambda_n|$ the eigenvalues of this symmetric 0-1 matrix.  Let $\mathbf{v} = (1,1\ldots1)$.  Then, because $\lambda_1$ is the largest eigenvalue of $A$, we have that 
\begin{align*}
|\lambda_1| \geq \frac{\langle A\mathbf{v}, \mathbf{v} \rangle}{\langle \mathbf{v}, \mathbf{v} \rangle} = \frac{\sum_{v \in V} \deg_i(v)}{n} \geq (1 + o(1)) \cdot \frac{n}{k}.
\end{align*}

Now, examine $A^4$.  The trace of this matrix, on one hand, is precisely $\sum \lambda_j^4$; on the other, it is precisely the sum over all vertices $v$ of
\begin{align*}
\#(i\textrm{-colored }4\textrm{-cycles starting at }v) + \#(\textrm{pairs of }i\textrm{-neighbors of }v);
\end{align*}
i.e. 
\begin{align*}
\textrm{tr}(A^4) = N_G(C_{4,i}) + \sum_{v \in V} (\deg_i(v))^2 = (1 + o(1)) \frac{n^4}{k^4}.
\end{align*}

Because we have already shown that $\lambda_1 = (1 + o(1))\frac{n}{k},$ this forces
\begin{align*}
\sum_{j=2}^n \lambda_j^4 = o(n^4);
\end{align*}
i.e. that $|\lambda_2|$ and all of the other eigenvalues are $o(n)$.  This is precisely $P_3$.
\end{proof}

\begin{prop}
$P_3 \Rightarrow P_0$.
\end{prop}
\begin{proof}
Let $\mathbf{v}$ be the all-$1$'s vector $(1,1\ldots 1)$ and $A = A(G_i)$ be the adjacency matrix of $G_i$, as before.  Then, if we have $P_3$, we know that 
\begin{align*}
||A\mathbf{v}||^2 &= \sum_{v \in V} (\deg_i(v))^2 \leq \lambda_1^2 \cdot n = (1 + o(1)) \frac{n^3}{k^2}.\\
\end{align*}
If we apply Cauchy-Schwarz, this tells us that 
\begin{align*}
 (1 + o(1)) \frac{n^4}{k^2} &\geq n  \sum_{v \in V} (\deg_i(v))^2 \geq \left(  \sum_{v \in V} \deg_i(v) \right)^2\\
\Rightarrow  (1 + o(1)) \frac{n^2}{k}  & \geq \sum_{v \in V} \deg_i(v).
\end{align*}

However, $P_3$ also directly gives us 
\begin{align*}
\sum_{v \in V} \deg_i(v) \geq (1+o(1)) \cdot \frac{n^2}{k}.
\end{align*}
Combining these inequalities yields $P_0$.
\end{proof}

\begin{prop}
$P_3 \Rightarrow P_4$.
\end{prop}
\begin{proof}
Pick any color $i$.  Let $A(G_i)$ be the associated adjacency matrix to the graph $G_i$, $|\lambda_1| \geq \ldots \geq |\lambda_n|$ be its eigenvalues, and $\mathbf{e}_1, \ldots \mathbf{e}_n$ the corresponding eigenvectors.  As well, let $\mathbf{u} = \left( \frac{1}{\sqrt{n}}, \ldots \frac{1}{\sqrt{n}} \right)$.

We claim first that the Perron-Frobenius eigenvector $e_1$ of $A$ is ``roughly'' $\mathbf{u}$:\ i.e. that $||\mathbf{u} - \mathbf{e}_1||$ is $o(1)$.  To see this, simply write $\mathbf{u} = \sum_{j=1}^n a_j e_j$.  Then, on one hand, we have that $A\mathbf{u} = \sum_{j=1}^n a_j \lambda_j e_j$; on the other, we also have that $A \mathbf{u} = \frac{1}{\sqrt{n}} \left( \deg_i(v_1), \ldots \deg_i(v_n) \right)$.  Because $P_3 \Rightarrow P_0$, we know that all but $o(n)$ of these vertices have degree $(1 +o(1))\frac{n}{k}$; therefore, we know that we can write
\begin{align*}
A\mathbf{u} = \left( (1 + o(1)) \frac{n}{k}\right)\cdot \mathbf{u} + \mathbf{w},
\end{align*}
for some vector $\mathbf{w}$ with all but $o(n)$ of its components with magnitude $o(\sqrt{n})$.  This forces $||\mathbf{w}|| = o(n)$.  Now, if we think about what this means for the eigenvalues of $A$, we can use $P_3$ to show that
\begin{align*}
&\sum_{j=1}^n a_j \lambda_j e_j = \left( (1 + o(1)) \frac{n}{k}\right)\cdot \mathbf{u} + \mathbf{w}\\
\Rightarrow &\sum_{j=1}^n a_j \left(\lambda_j - \frac{n}{k}\right) e_j = o(1)\cdot \frac{n}{k}\cdot \mathbf{u} + \mathbf{w}\\
\Rightarrow &\left|\left|\sum_{j=1}^n a_j \left(\lambda_j - \frac{n}{k}\right) e_j \right|\right| = \left|\left|o(1)\cdot \frac{n}{k}\cdot \mathbf{u} + \mathbf{w}\right|\right| = o(n)\\
\Rightarrow &\left(\sum_{j=1}^n a_j^2 \left(\lambda_j - \frac{n}{k}\right)^2 \right)^{1/2} = o(n)\\
\Rightarrow &\left(\sum_{j=2}^n a_j^2\left(\frac{n}{k}\right)^2 \right)^{1/2} = o(n).\\
\end{align*}
Therefore, we have $\sum_{j=2}^n a_j^2 = o(1)$, and consequently that $\mathbf{u} = a_1 \mathbf{e}_1 + \mathbf{v}$, for some vector $\mathbf{v}$ with $||\mathbf{v}|| = o(1)$.  This tells us that $|a_1| = 1 + o(1)$; as well, because $\mathbf{e}_1$ is the eigenvector corresponding to the largest eigenvalue of a nonnegative symmetric matrix $A$, we know by the Perron-Frobenius theorem that $\mathbf{e}_1$ is nonegative, and therefore that $a_1 = 1 + o(1)$.  This proves our claim that  $||\mathbf{u} - \mathbf{e}_1||$ is $o(1)$.

Let us use this fact in proving our current proposition.  Given any subset $S \subseteq V$, set $\mathbf{\chi}_S$ to be the characteristic vector of $S$:\ i.e. $\mathbf{\chi}_S $ has a 1 in its $j$-th slot if $v_j \in S$, and a 0 otherwise.  As well, define $\mathbf{s} = \mathbf{\chi}_S - \langle \mathbf{\chi}_S, \mathbf{e}_1 \rangle \mathbf{e}_1$:\ i.e. $\mathbf{s}$ is the result of taking $\chi_1$ and subtracting off its $\mathbf{e}_1$-component.  

With these definitions made, let us examine the quantity $\langle A\mathbf{s}, \mathbf{s} \rangle$ in two different ways.  On one hand, if we use our claim from earlier, we have 
\begin{align*}
\langle A\mathbf{s}, \mathbf{s} \rangle & = \langle A (\chi_s -  \langle \mathbf{\chi}_S), \chi_s -  \langle \mathbf{\chi}_S \rangle \\
&= \langle A \chi_s, \chi_s \rangle - \langle A \chi_s, \langle \chi_s, \mathbf{e}_1 \rangle \mathbf{e}_1 \rangle - \langle A \langle \chi_s, \mathbf{e}_1 \rangle \mathbf{e}_1, \chi_s \rangle + \langle A \langle \chi_s, \mathbf{e}_1 \rangle \mathbf{e}_1, \langle \chi_s, \mathbf{e}_1 \rangle \mathbf{e}_1 \rangle\\
&= \langle A \chi_s, \chi_s \rangle  -  \lambda_1 \langle\chi_s, \mathbf{e}_1 \rangle^2\\
&= 2 e_i(S) - \lambda_1 \langle \chi_s, \mathbf{e}_1 \rangle^2\\
&= 2 e_i(S) - \lambda_1 \langle \chi_s, \mathbf{u} + \mathbf{v} \rangle^2\\
&= 2 e_i(S) - \lambda_1 \left( \frac{|S|}{\sqrt{n}} + o(\sqrt{|S|}\right)^2\\
&= 2 e_i(S) - \left( \frac{1}{k} + o(1) \right) |S|^2 + o(n^2).
\end{align*}

On the other,  if we use the observation that $\mathbf{s}$ is orthogonal by construction to $\mathbf{e}_1$, we can see that
\begin{align*}
\langle A\mathbf{s}, \mathbf{s} \rangle \leq |\lambda_2| \cdot ||\mathbf{s}||^2 = |\lambda_2| \cdot || \chi_s -  \langle \mathbf{\chi}_S, \mathbf{e}_1 \rangle \mathbf{e}_1||^2 \leq ||\chi_S|| =|\lambda_2| \cdot  |S| = o(n) \cdot |S|.
\end{align*}

Combining these two observations tells us that 
\begin{align*}
 & 2 e_i(S) - \left( \frac{1}{k} + o(1) \right) |S|^2 + o(n^2) \leq o(n) \cdot  |S|\\
\Rightarrow & e_i(S) = (1 + o(1)) \frac{|S|^2}{2k}.
\end{align*}
This is precisely $P_4$.
\end{proof}

\begin{prop}
$P_4 \Rightarrow P_0$.
\end{prop}
\begin{proof}
Suppose that for any color $i$ and any subset $S \subseteq V$, we have
\begin{align*}
\left|e_i(S) - \frac{|S|^2}{2k} \right| < \epsilon^2n^2.
\end{align*}
If for any $\epsilon > 0$ this always eventually holds for large enough $n$, this assumption is precisely $P_4$.

Suppose further that there is some color $i$ and some set $T$ with $|T| = t \geq \epsilon n$ vertices with total $i$-degree greater than $\left(\frac{1}{k} + \epsilon\right)n$.  Then, we have that
\begin{align*}
\sum_{v \in T}  \deg_i(v) \geq \left(\frac{1}{k} + \epsilon\right)  tn.
\end{align*}
However, by assumption, we have
\begin{align*}
e_i(G) <  \frac{n^2}{2k} + \epsilon^2 n^2, \quad e_i(T) < \frac{t^2}{2k} + \epsilon^2 n^2, \quad e_i(G \setminus T) > \frac{(n-t)^2}{2k} - \epsilon^2 n^2.
\end{align*}
Therefore, because
\begin{align*}
e_i(G \setminus T) + \sum_{v \in T}  \deg_i(v) = e_i(G) + e_i(T),
\end{align*}
we have 
\begin{align*}
&\frac{(n-t)^2}{2k} - \epsilon^2 n^2 +   \left(\frac{1}{k} + \epsilon\right) tn <   \frac{n^2}{2k} + \epsilon^2 n^2 +  \frac{t^2}{2k} + \epsilon^2 n^2\\
\Rightarrow & \epsilon t n < 3 \epsilon^2 n^2.
\end{align*}
This is impossible for $t > 3 \epsilon n$.  Therefore, if $P_4$ holds, we know that there cannot be any more than $o(n)$ vertices with degree greater than $\left(\frac{1}{k} + \epsilon\right)n$.  

An identical argument will give you the lower bound on the degrees of vertices in $G$; combining these results yields $P_0$, as claimed.
\end{proof}

\begin{prop}
$P_4 \Leftrightarrow P_5$.
\end{prop}
\begin{proof}
That $P_4$ implies $P_5$ is immediate; so it suffices to prove the other direction.  Fix any color $i$ and any $\epsilon > 0$, and suppose that for any subset $S$ with $|S| = \lfloor n/2 \rfloor$, we have
\begin{align*}
\left| e_i(S) - \frac{n^2}{8k} \right| < \epsilon n^2.
\end{align*}

Take any $T \subseteq V$; we will find some constant $C$ -- in fact, $C \leq 10$ --- such that
\begin{align*}
\left| e_i(T) - \frac{\binom{|T|}{2}}{k}  \right| < C \epsilon n^2.
\end{align*}
Note that doing this will prove that $P_5 \Rightarrow P_4$, by letting $\epsilon \to 0$.

To prove our claim we consider two different cases for $T$:\ either $|T| \geq \frac{n}{2}$, or $|T| < \frac{n}{2}$.  In the first case, we can prove our claim by expressing $e_i(T)$ as the average value of the sets $e_i(S')$, over all subsets $S'$ of $T$ with size $\lfloor n/2 \rfloor$.  To do this, notice that every edge in $T$ occurs in precisely $\binom{|T|-2}{\lfloor n/2 \rfloor - 2}$-many subsets of $T$ of size $\lfloor n/2 \rfloor$; therefore, we have
\begin{align*}
e_i(T) = \frac{1}{\binom{|T|-2}{\lfloor n/2 \rfloor - 2}}\sum_{S' \subset T} e_i(S') &\leq \frac{\binom{|T|}{\lfloor n/2 \rfloor} \left( \frac{n^2}{8k} + \epsilon n^2 \right)}{\binom{|T|-2}{\lfloor n/2 \rfloor - 2}}\\
 &\leq \frac{|T|\cdot (|T|-1)}{\lfloor n/2 \rfloor \cdot \lfloor n/2-1 \rfloor}\left( \frac{n^2}{8k} + \epsilon n^2 \right)\\
&\leq  \binom{|T|}{2}\left( \frac{1}{k} + 8\epsilon \right).
\end{align*}
Bounding the $e_i(S')$'s below by $\left( \frac{n^2}{8k} - \epsilon n^2\right)$ gives the corresponding lower bound
\begin{align*}
e_i(T) \geq \binom{|T|}{2}\left( \frac{1}{k} - 8\epsilon \right);
\end{align*}
by combining these two results, we have demonstrated our claim.

Now, consider the case where $|T| < \lfloor n/2 \rfloor$.  Suppose that $|T| > \frac{1}{k}\binom{|T|}{2} + C \epsilon n^2$, for some constant $C$.  Consider the complement of $T$, $\overline{T}$.  We know that the number of $i$-colored edges from $T$ to $\overline {T}$, $e_i(T, \overline{T})$, is given by
\begin{align*}
e_i(T, \overline{T}) = e_i(G) - e_i(T) - e_i(\overline{T}).
\end{align*}
We know that $\overline{T}$ has $\geq \lfloor n/2 \rfloor$ vertices in it, and therefore we can use our earlier arguments to bound the size of $e_i(\overline{T})$.  $e_i(G)$ is known as well.  Therefore, in theory, this gives us a way to relate the quantity we are interested in ($e_i(T)$) to a potentially easier-to-study quantity $(e(T, \overline{T}))$.  

To do this, pick any subset $S'$ such that $S' \cap T = \emptyset$, $|S \cup T| = \lfloor n/2 \rfloor$.  The average value of $e_i(T \cup S')$ over all such sets is just
\begin{align*}
&\frac{1}{\binom{n - |T|}{\lfloor n/2 \rfloor - |T|}}  \sum_{S'} e_i(T \cup S') \\
=& \frac{1}{\binom{n - |T|}{\lfloor n/2 \rfloor - |T|}}   \cdot \left( e_i(T) \binom{n - |T|}{\lfloor n/2 \rfloor - |T|} + e_i(\overline{T})\binom{n - |T|-2}{\lfloor n/2 \rfloor - |T|-2} + e_i(T, \overline{T})\binom{n - |T-1|}{\lfloor n/2 \rfloor - |T|-1}\right),
\end{align*}
by breaking the edges counted above into three groups.
\begin{itemize}
\item Edges in $T$: these come up with multiplicity equal to the number of possible choices of $S'$.
\item Edges in $\overline{T}$: these come up with multiplicity equal to the number of $S'$'s that can be picked to include that edge.
\item Edges connecting $T$ and $\overline{T}$: these come up with multiplicity equal to the number of times the $\overline{T}$-vertex is chosen in $S'$.
\end{itemize}

Dividing through yields
\begin{align*}
&e_i(T) +  e_i(\overline{T})\frac{(\lfloor n/2 \rfloor - |T|) (\lfloor n/2 \rfloor - |T| - 1)}{(n-|T|)(n-|T|-1)} + 
e_i(T, \overline{T})\frac{\lfloor n/2 \rfloor - |T|}{n-|T|}\\
=& e_i(T) +  e_i(\overline{T})\frac{(\lfloor n/2 \rfloor - |T|) (\lfloor n/2 \rfloor - |T| - 1)}{(n-|T|)(n-|T|-1)} + ( e_i(G) - e_i(T) - e_i(\overline{T})))\frac{\lfloor n/2 \rfloor - |T|}{n-|T|}\\
=& e_i(T) \frac{\lfloor n/2 \rfloor}{n - |T|} - e_i(\overline{T})\frac{(\lfloor n/2 \rfloor) (\lfloor n/2 \rfloor - |T| )}{(n-|T|)(n-|T|-1)} + e_i(G) \frac{\lfloor n/2 \rfloor - |T|}{n-|T|}\\
>&\left(\frac{\binom{|T|}{2}}{k} + C \epsilon n^2\right) \frac{n}{2(n - |T|)} - \binom{n-|T|}{2}\left(\frac{1}{k} + 8\epsilon \right) \frac{( n/2 ) (( n/2 ) - |T| )}{(n-|T|)(n-|T|-1)} + \binom{n}{2}\left( \frac{1}{k} - 8\epsilon \right)\frac{( n/2 ) - |T|}{n-|T|}\\
>& \frac{|T|^2n}{4k(n-|T|)} - \frac{|T|n}{4k(n-|T|)} + \frac{C\epsilon n^3}{2(n-|T|)} + \left(2\frac{n-1}{n-|T|} - 1 \right)\left( n - 2|T| \right)\frac{n}{8k}   - \left(2\frac{n-1}{n-|T|} + 1 \right)  \left( \epsilon n^2 - 2|T|\epsilon n \right)\\
=& \frac{1}{8k(n-|T|)} \left( n^3 + |T|n^2 - 2n^2 - 4|T|n \right) + \frac{C\epsilon n^3}{2(n-|T|)} - \frac{|T|n}{4k(n-|T|)} - \left(2\frac{n-1}{n-|T|} + 1 \right)  \left( \epsilon n^2 - 2|T|\epsilon n \right)\\
=&\left( \frac{n^3 + |T|n^2}{8k(n-|T|)}  \right)+ \left(\frac{C\epsilon n^3}{2(n-|T|)} - \left(2\frac{n-1}{n-|T|} + 1 \right)  \left( \epsilon n^2 - 2|T|\epsilon n \right)\right) - \left(\frac{ 2n^2 + 4|T|n}{8k(n-|T|)} + \frac{|T|n}{4k(n-|T|)} \right).
\end{align*}
Observe that the first quantity in parentheses is minimized when $|T|= 0$, in which case it is $\frac{n^2}{8k}$; the second quantity is also minimized when $|T| = 0$, in which case it is bounded below by $\frac{C}{2} \epsilon n^2 - 3\epsilon n^2$; and the third quantity is maximized when $|T| = \frac{n}{2}$, in which case it is $\frac{5n}{4k}$.  Therefore, this entire average is bounded below by
\begin{align*}
\frac{n^2}{8k} + \left(\frac{C}{2} - 3 \right) \epsilon n^2 - \frac{5n}{4k} > \frac{n^2}{8k} + \left(\frac{C}{2} - 4 \right) \epsilon n^2,
\end{align*}
for sufficiently large values of $n$.

However, we know that the average value of these $ e_i(T \cup S')$'s is $< \frac{n^2}{8k} + \epsilon n^2$, by assumption.  Therefore, we know that $C$ cannot be larger than 10; i.e. that 
\begin{align*}
|T| < \frac{1}{k}\binom{|T|}{2} + 10 \epsilon n^2.
\end{align*}
The exact same logic can be extended to construct the lower bound of $\frac{1}{k}\binom{|T|}{2} - 10 \epsilon n^2$, as well. 

Therefore, we have proven that $P_5 \Leftrightarrow P_4$.
\end{proof}

\begin{prop}
$P_4 \Rightarrow P_6$.
\end{prop}
\begin{proof}
Fix any $\epsilon > 0$ and any color $i$, and let 
\begin{align*}
V_i' = \{v \in V:\ \left| \deg_i(v) - \frac{n}{k} \right| < \epsilon n\}.
\end{align*}
Note that by our proof of $P_4 \Rightarrow P_0$, if $P_4$ holds, this set $V'$ contains all but at most $3 \epsilon n$ of $V$'s elements.  Therefore, if we let 
\begin{align*}
V' = \cap_{i=1}^k V_i',
\end{align*}
this set will contain all but at most $3k\epsilon n$ many elements, and have the property that the $i$-degree of every element in this set is roughly $\frac{n}{k}$, for every color $i$.

As well, if $P_4$ holds, note that for sufficiently large $n$ we can assume that
\begin{align*}
\left| e_i(S) - \frac{|S|^2}{2k} \right| < \epsilon n^2,
\end{align*}
for any subset $S$ of our vertex set $V$.

Given any color $i$, define
\begin{align*}
s_{i}(v,w) = \{ x \in V:\ \textrm{color}(v,x) = i =\textrm{color}(w,x)  \}.
\end{align*}
Notice that under this definition, we have 
\begin{align*}
s(v,w) = \sum_{i=1}^k s_i(v,w).
\end{align*}

Given a vertex $v \in V$, let $X(v) = \{w \in V':\ \left|s(v,w) - \frac{n}{k} \right| > Ck \epsilon n\}$, where $C$ is some constant (that turns out to be no more than $6k+1$) that we will determine later.  There are two possibilities.

1.  $|X(v)| \leq 2k\epsilon n$, for every vertex $v \in V$. In this case, we have
\begin{align*}
\sum_{v, w \in V} \left| s(v,w) - \frac{n}{k} \right| \leq C \epsilon n \cdot n^2 + \frac{k-1}{k} n \cdot n \cdot \left(2k \epsilon n\right) + \frac{k-1}{k} n \cdot n \cdot \left( 3k\epsilon n \right) < (\left( C + 4k\right) \epsilon n^3.
\end{align*}
Because our choice of $\epsilon > 0$ was arbitrary, this effectively says that this quantity is $o(n^3)$:\ in other words, $P_6$ holds.

2. There is some vertex $v_0$ such that $|X(v_0)| > 2k \epsilon n$.  In this case, let $X_+ = \{w \in X(v_0):\ s(v_0, w) > \frac{n}{k} + Ck \epsilon n\}$, and $X_- = \{w \in X(v_0):\ s(v_0, w) < \frac{n}{k} - Ck \epsilon n\}$; one of these sets must have at least $k \epsilon n$ elements in it.  Assume that $X_+$ does for now; the proof for the other case will look identical to the proof we will pursue below.

Take the set $X_+$, and further divide it into the sets $X_+^i = \{w \in X(v_0):\ s_i(v_0, w) > \frac{n}{k^2} + C\epsilon n\}$.  Every vertex in $X_+$ has to lie in at least one of these $X_+^i$'s:\ therefore, there is at least one $X_i^+$ with $\epsilon n$ many elements in it.

If this holds, then look at the quantity $e_i(X_+^i, n(v_0))$.  On one hand, we know that
\begin{align*}
e_i(X_+^i, n_i(v_0)) \geq |X_+^i| \cdot \left(  \frac{n}{k^2} + C\epsilon n \right),
\end{align*}
because every vertex in $X_i^+$ has at least $ \frac{n}{k^2} + C\epsilon n$-many common $i$-neighbors with $v_0$.  

On the other hand, notice that we can calculate $e_i(X_+^i, n(v_0))$ strictly in terms of the sizes of other sets, the sizes of which we can control with $P_4$.  We do this here.
\begin{align*}
 & e_i(X_+^i \cup n(v_0)) - e_i(X_+^i) - e_i(n_i(v_0)) +  3 e_i(X_+^i \cap n(v_0))\\
 \leq &\frac{\left| X_+^i \cup n_i(v_0) \right|^2}{2k} + \epsilon n^2 - \frac{\left|X_+^i \right|^2}{2k} + \epsilon n^2 - \frac{\left| n_i(v_0) \right|^2}{2k} + \epsilon n^2 + 3\frac{\left| X_+^i \cap n_i(v_0) \right|^2}{2k} + 3 \epsilon n^2\\
 = &\frac{\left( \left|X_+^i \right| + \left|n_i(v_0) \right| - \left| X_+^i \cap n_i(v_0) \right| \right)^2 - \left|X_+^i \right|^2 - \left| n_i(v_0) \right|^2 + 3\left| X_+^i \cap n_i(v_0) \right|^2}{2k} + 6 \epsilon n^2\\
 =& \frac{2\left|X_+^i\right|\cdot\left|n_i(v_0)\right| - 2\left( \left|X_+^i\right| + \left|n_i(v_0)\right|  \right)\cdot \left| X_+^i \cap n_i(v_0) \right|   + 4\left| X_+^i \cap n_i(v_0) \right|^2}{2k} + 6 \epsilon n^2\\
 =& \frac{2\left|X_+^i\right|\cdot\left|n_i(v_0)\right| - 2\left( \left|X_+^i\right| + \left|n_i(v_0)\right|  - 2\left| X_+^i \cap n_i(v_0) \right| \right)\cdot \left| X_+^i \cap n_i(v_0) \right|}{2k} + 6 \epsilon n^2\\
 \leq& \frac{2\left|X_+^i\right|\cdot\left|n_i(v_0)\right|}{2k} + 6 \epsilon n^2\\
\leq& \frac{\left|X_+^i\right| \left(\frac{n}{k} + \epsilon n \right)}{k} + 6 \epsilon n^2.\\
\end{align*}

Therefore, we have 
\begin{align*}
 |X_+^i| \cdot \left(  \frac{n}{k^2} + C\epsilon n \right) \leq \frac{\left|X_+^i\right| \left(\frac{n}{k} + \epsilon n \right)}{k} + 6 \epsilon n^2.
\end{align*}
However, if we set $C \geq 6k+1$, this is impossible!  Therefore, for such a choice of $C$, this second possibility never occurs; therefore, we are always in the first case that we discussed earlier, in which we showed that $P_6$ holds.
\end{proof}

\begin{prop}
$P_6 \Rightarrow P_1(s)$.
\end{prop}
\begin{proof}
Suppose that $P_6$ holds; i.e. that
\begin{align*}
\sum_{v,v' \in V} \left| s(v,v') - \frac{n}{k} \right| = o(n^3).
\end{align*}

We will prove, by induction on $s$, that the number of labeled occurrences of graphs on $s$ vertices grows as claimed by $P_1$:\ specifically, that
\begin{align*}
(\ddag) \quad N_G^*(M(s)) = (1 + o(1)) \frac{n!}{s!} \cdot k^{-\binom{s}{2}}.
\end{align*}
For $s = 1$, the above claim is immediate, as there are $n = (1+ o(1)) \cdot n$ distinct $1$-vertex labeled subgraphs of any graph on $n$ vertices.  

For the inductive step, assume that for some $r$ that the equation $(\ddag)$ holds.  To extend this result to labeled subgraphs on $r+1$ vertices, make the following definitions:\ let $\alpha = (\alpha_1, \ldots \alpha_r$ denote a subset of $r$ distinct vertices from $V$, and $\mathbf{\epsilon} = (\epsilon_1, \ldots \epsilon_r) \subset [k]^r$ denote an $r$-tuple of possible colors ranging from $1$ to $k$.  Using these definitions, let
\begin{align*}
f_r(\alpha, \epsilon) = \{ v \in V:\ v \neq \alpha_i \textrm{ and color}(v, \alpha_i) = \epsilon_i, 1 \leq i \leq k \}.
\end{align*}
Notice that given any graph $M(s)$ and any extension $M(s+1)$ of this graph, $N_G^*(M(s+1)) $ is just the sum of $N_G^*(M(s))$ copies of $f_r(\alpha, \epsilon)$'s.  This is because $f_r(\alpha, \epsilon)$ counts precisely the number of ways of extending a given labeled subgraph on $r$ vertices to an additional vertex with edge colors specified by $\epsilon$.  As well, notice that there are precisely $\frac{n!}{(n-r)!} \cdot k^r$-many different possible $f_r(\alpha, \epsilon)$'s, as there are $\frac{n!}{(n-r)!}$-many different ways to choose $\alpha$ and $k^r$-many different ways to choose $\epsilon$.

Using the same proof methods as in Chung, Graham, and Wilson's paper, we seek to control the first and second moments of the $f_r(\alpha, \epsilon)$'s.

Specifically, observe that the sum
\begin{align*}
\sum_{\alpha, \epsilon} f_r(\alpha, \epsilon) = \sum_{\alpha} \sum_\epsilon f_r(\alpha, \epsilon) = \sum_{\alpha} (n-r) = \frac{n!}{(n-r)!}(n-r) = \frac{n!}{(n-r-1)!},
\end{align*}
where we used the fact that any given vertex $v \notin \{\alpha_1, \ldots \alpha_r\}$ has a unique $\epsilon$ for which it will be counted in $ f_r(\alpha, \epsilon)$.  Denote the average value $\frac{(n-r)!}{n! \cdot k^r} \cdot \sum_{\alpha, \epsilon} f_r(\alpha, \epsilon) = \frac{n-r}{k^r}$ of these $f_r(\alpha, \epsilon)$'s as $\overline{f_r}$.

Now, look at the quantity
\begin{align*}
S_r := \sum_{\alpha, \epsilon} f_r(\alpha, \epsilon)(f_r(\alpha, \epsilon) - 1).
\end{align*}

We claim that 
\begin{align*}
S_r = \sum_{v \neq w} \frac{s(v,w)!}{(n-(s(v,w))!}.
\end{align*}
This can be proven by counting $S_r$ in two ways.  First, observe that $S_r$ is just the number of ways of picking $\alpha, \epsilon$ and two ordered vertices $v, w$ not in $\alpha$ such that 
\begin{align*}
\textrm{color}(v, \alpha_i) = \epsilon_i = \textrm{color}(w, \alpha_i), 1 \leq i \leq k.
\end{align*}
Because we are summing over all possible values of $\epsilon$, we can see that we are actually just choosing $\alpha, v, w$ such that 
\begin{align*}
\textrm{color}(v, \alpha_i)  = \textrm{color}(w, \alpha_i), 1 \leq i \leq k.
\end{align*}
However, if you now imagine that we choose $v, w$ first before picking $\alpha$, we can see that the choices of vertices for $\alpha$ have to be precisely those at which the colors $\textrm{color}(v, \alpha_i)  = \textrm{color}(w, \alpha_i)$:\ in other words, we are picking from precisely the pool of vertices counted by $s(v,w)$.  Because we are choosing $r$ of these vertices, we have proven our claim.

Now, we claim that we can use $P_6$ to show 
\begin{align*}
\sum_{v \neq w} \frac{s(v,w)!}{(n-(s(v,w))!} = (1 + o(1)) n^{r+2} \cdot k^{-r}.
\end{align*}
To do this, first notice that because $\left| s(v,w) - \frac{n}{k} \right| \leq n$ for any pair of vertices $v,w$, $P_6$ tells us that
\begin{align*}
\sum_{v \neq w} \left| s(v,w) - \frac{n}{k} \right|^d  \leq n^{d-1} \cdot \sum_{v \neq w} \left| s(v,w) - \frac{n}{k} \right| = o(n^{d+2}).
\end{align*}
Therefore, we have
\begin{align*}
\sum_{v \neq w}  \frac{s(v,w)!}{(n-(s(v,w))!} &= \sum_{v \neq w}  \frac{\left( \frac{n}{k} +s(v,w) - \frac{n}{k} \right)!}{\left(n - \frac{n}{k} -s(v,w) + \frac{n}{k} \right)!}\\
&= \sum_{k=0}^r \sum_{v \neq w}  c_k \left( \frac{n}{k} \right)^k \left(s(v,w) - \frac{n}{k} \right)^{r-k}, \textrm{ (for appropriate constants }c_k)\\
&= \left( \frac{n}{k} \right)^r \cdot(n)(n-1) + \sum_{k=0}^{r-1} \sum_{v \neq w}  c_k \left( \frac{n}{k} \right)^k \left(s(v,w) - \frac{n}{k} \right)^{r-k}\\
&\leq \left( \frac{n}{k} \right)^r \cdot(n)(n-1) + c\sum_{k=0}^{r-1} \sum_{v \neq w}   n^k \cdot \left|s(v,w) - \frac{n}{k} \right|^{r-k}\\
&\leq \left( \frac{n}{k} \right)^r \cdot(n)(n-1) + c\sum_{k=0}^{r-1} n^k \sum_{v \neq w}   \cdot \left|s(v,w) - \frac{n}{k} \right|^{r-k}\\
&\leq \left( \frac{n}{k} \right)^r \cdot(n)(n-1) + c\sum_{k=0}^{r-1} n^k o(n^{r-k+2})\\
&\leq \left( \frac{n}{k} \right)^r \cdot(n)(n-1) + o(n^{r+2})\\
&=(1 + o(1))n^{r+2}k^{-r}.
\end{align*}

Consequently, if we return to our desire to control the second moment of the $f_r(\alpha, \epsilon)$'s, we can see that
\begin{align*}
\sum_{\alpha, \epsilon} \left( f_r(\alpha, \epsilon) - \overline{f_r}\right)^2 &= \left(\sum_{\alpha, \epsilon} \left( f_r(\alpha, \epsilon)\right)^2\right) - 2\left(\sum_{\alpha, \epsilon} f_r(\alpha, \epsilon) \cdot \overline{f_r}\right) +\left(\sum_{\alpha, \epsilon}\left(\overline{f_r} \right)^2 \right)\\
&= \left(\sum_{\alpha, \epsilon}  \left( f_r(\alpha, \epsilon)\right)^2 - f_r(\alpha, \epsilon)  \right) + \left(\sum_{\alpha, \epsilon} f_r(\alpha, \epsilon)\right)  - \frac{n!}{(n-r)!} \cdot k^r \cdot \left( \frac{n-r}{k^r} \right)^2\\
\end{align*}
\begin{align*}
&= S_r + \frac{n!}{(n-r-1)!} + \frac{n!}{(n-r)!}\cdot(n-r)^2 \cdot k^{-r}\\
&= o(n^{r+2}).
\end{align*}

If we plug these relations into our earlier observation that
\begin{align*}
 N_G^*(M(r+1)) = \sum_{\substack{ N_G^*(M(r))  \\ \textrm{choices of }(\alpha, \epsilon)}} f_r(\alpha, \epsilon),
\end{align*}
we get that 
\begin{align*}
 \left| N_G^*(M(r+1)) - N_G^*(M(r)) \overline{f_r} \right|^2  &= \left| \sum_{\substack{ N_G^*(M(r))  \\ \textrm{choices of }(\alpha, \epsilon)}} \left( f_r(\alpha, \epsilon) - \overline{f_r}  \right)\right|\\
&\leq N_G^*(M(r))  \sum_{\substack{ N_G^*(M(r))  \\ \textrm{choices of }(\alpha, \epsilon)}} \left( f_r(\alpha, \epsilon) - \overline{f_r}  \right)^2 \\
&\leq N_G^*(M(r))  \sum_{\alpha, \epsilon} \left( f_r(\alpha, \epsilon) - \overline{f_r}  \right)^2 \\
&= o\left( N_G^*(M(r)) \cdot n^{r+2} \right) = o(n^{2r+2}).
\end{align*}

Therefore, we have
\begin{align*}
\left| N_G^*(M(r+1)) - N_G^*(M(r)) \cdot \overline{f_r} \right| = o(n^{r+1}),
\end{align*}
and thus
\begin{align*}
N_G^*(M(r+1)) = &N_G^*(M(r)) \cdot \overline{f_r} + o(n^{r+1}) \\
=& (1 + o(1)) \frac{n!}{r!} \cdot k^{-\binom{r}{2}} \cdot \frac{n-r}{k^r} + o(n^{r+1})\\
=& (1 + o(1))\frac{n!}{(n-r-1)!}k^{-\binom{r+1}{2}},\\
\end{align*}
which is precisely our inductive claim.  
\end{proof}

\begin{prop}
$P_2(4) \Rightarrow P_7$.
\end{prop}
\begin{proof}
Simply notice that given any color $i$,
\begin{align*}
(1+ o(1) \frac{n^4}{k^4} = N_G(C_{4,i}) = \sum_{v, w \in V}  \left| n_i(v) \cap n_i(w) \right|^2;
\end{align*}
this is because we can generate all of the $i$-colored 4-cycles uniquely by taking all pairs of vertices $v,w$ along with any pair of elements from $n_i(v) \cap n_i(w)$.  Applying Cauchy-Schwarz gives us that
\begin{align*}
&(1+ o(1) \frac{n^6}{k^4}  = n^2\sum_{v, w \in V}  \left| n_i(v) \cap n_i(w) \right|^2 \geq \left( \sum_{v,w \in V} n \left| n_i(v) \cap n_i(w) \right| \right)^2\\
\Rightarrow & (1+ o(1) \frac{n^3}{k^2} \geq  \sum_{v,w \in V} n \left| n_i(v) \cap n_i(w) \right| \\
\Rightarrow & (1 + o(1) \frac{n^2}{k^2} \geq \sum_{v, w \in V}  \left| n_i(v) \cap n_i(w) \right| \\
\Rightarrow & \sum_{v, w \in V}  \left| n_i(v) \cap n_i(w) \right|  = o(n^3).\\
\end{align*} 
\end{proof}

\begin{prop}
$P_7 \Rightarrow P_3$.
\end{prop}
\begin{proof}
Choose any color $i$.  Let $A = A(G_i)$ be the associated adjacency matrix to the graph $G_i$, $|\lambda_1| \geq \ldots \geq |\lambda_n|$ be its eigenvalues.  As well, let $\mathbf{u} = (1,1,\ldots 1)$.

If we have property $P_7$, we know that all but $o(n^2)$ pairs $(v,w)$ have 
\begin{align*}
\left| n_i(v) \cap n_i(w) \right| = (1 + o(1))\frac{n}{k^2};
\end{align*}
in other words, all but $o(n^2)$ entries of $A$ are $(1 +o(1))\frac{n}{k^2}$.  Therefore, we have that
\begin{align*}
&\lambda_1^2 \cdot ||\mathbf{u}||^2 = \lambda_1^2 n \geq ||A\mathbf{u}||^2 = \langle A\mathbf{u}, A\mathbf{u} \rangle = \langle A^2 \mathbf{u}, \mathbf{u} \rangle = (1 +o(1))\frac{n^3}{k^2}\\
\Rightarrow &|\lambda_1| \geq (1 + o(1))\frac{n}{k}\\
\Rightarrow &\lambda_1 \geq (1 + o(1))\frac{n}{k} \quad \textrm{(because } \lambda_1 > 0, \textrm{ by Perron-Frobenius.)}\\
\end{align*}

As well, if we want to control the other eigenvalues, it suffices to examine the trace of $A^4$ in two different ways.  On one hand,
\begin{align*}
\textrm{tr}(A^4) &= \sum \lambda_i^4 > \lambda_1^4 = (1 + o(1))\frac{n^4}{k^4}.\\
\end{align*}

On the other hand,
\begin{align*}
\textrm{tr}(A^4) &= \sum_{v \in V} (\textrm{closed paths of length }4\textrm{ starting at }v)\\
&=\sum_{v,w \in V} (\textrm{paths of length 2 }v \to w)\cdot  (\textrm{paths of length 2 }w \to v)\\
&=\sum_{v,w \in V} \left| n_i(v) \cap n_i(w) \right| \cdot\left| n_i(v) \cap n_i(w) \right| \\
&= \sum_{v,w \in V} \left((1 + o(1)) \frac{n}{k^2}\right)^2\\
&= (1 + o(1)) \frac{n^4}{k^4}.\\
\end{align*}
Therefore, we have $\lambda_1 = (1 + o(1)) \frac{n}{k}$, and $\lambda_2 = o(n)$, as claimed.  

Finally, notice that because
\begin{align*}
(1 + o(1)) \frac{n^2}{k^2} = \frac{1}{n}\sum_{v,w \in V} \left| n_i(v) \cap n_i(w) \right| &= \sum_{u \in V} \frac{1}{n}(\deg_i(u))(\deg_i(u) - 1)  \\
&\leq \sum_{u \in V}\frac{1}{n} (\deg_i(u))^2 \leq \left( \sum_{u \in V} \frac{1}{n}\deg_i(u) \right)^2
\end{align*}
we have 
\begin{align*}
\sum_{u \in V} (\deg_i(u)) \geq (1 + o(1)) \frac{n^2}{k};\\
\end{align*}
i.e. that $e_i(G) \geq   (1 + o(1)) \frac{n^2}{2k}$.  This is the last property we need to demonstrate for $P_3$; therefore, we have shown that $P_7 \Rightarrow P_3$.
\end{proof}

This ends our chain of equivalences, and completes our proof of Theorem $\ref{qrthm}$. Before closing this section, we should note that there are interesting examples of these graphs.  Specifically, consider the generalized Paley graphs, as defined in e.g.\ \cite{lim_prager_2009}. We tweak their definition slightly here.
\begin{defn}
Let $\mathbb{F}_q$ be a finite field of order $q$, and $k$ be a divisor of $q-1$ such that $k \geq 2$.  If $q$ is odd, ask that $\dfrac{q-1}{k}$ is even.  Let $S$ be the subgroup of order $\frac{q-1}{k}$ of $\mathbb{F}_q^*$.  There are $k$ possible cosets of $S$, counting $S$ itself; identify each coset with a color $\{1,\ldots k\}$.  The \textbf{generalized $k$-edge-colored Paley graph} generated by $\mathbb{F}_q$ is the graph with vertex set $\mathbb{F}_q$, where each edge $\{x, y\}$ is colored by the coset that $x-y$ belongs to.  
\end{defn}
It is relatively easy to show that this graph satisfies property $P_6$, in the same way that we did for the Paley graphs.  Consequently, these graphs are quasirandom, and therefore satisfy the other quasirandom properties $P_1, \ldots P_7$.

\section{Future Directions}

There are a number of interesting directions to take these results.  One path would be to investigate $k$-colorings of hypergraphs.  Another would be to investigate extensions to $k$-colored $m$-partite graphs, which should be a relatively simple proof to come up with.  More interestingly, an extension of these results to $k$-colored tournaments would allow us to examine directed analogues of the generalized Paley graphs.  

While these results would be interesting in their own right, they are also influenced in part by their potential applications to other interesting questions.  One of these applications is the study of Hadamard matrices \cite{FRWilson_1988}, which Frankl, R\"odl, and Wilson have already shown to have some quasirandom-like properties.

\bibliographystyle{plain}	
\bibliography{myrefs}
\end{document}